\documentclass[11pt]{amsart}
\pdfoutput=1
\usepackage{amsmath, amsthm, amssymb, amsfonts, , enumerate}
\usepackage[colorlinks=true,linkcolor=black,urlcolor=black]{hyperref}
\usepackage{xcolor}
\usepackage{geometry}
\usepackage{bm}
\usepackage{ulem}

\usepackage{natbib}

\usepackage{mathtools}

\def \N{\mathbb{N}}
\def \R{\mathbb{R}}

\newtheorem{theorem}{Theorem}[section]
\newtheorem{lemma}[theorem]{Lemma}

\newtheorem{proposition}[theorem]{Proposition}
\newtheorem{definition}[theorem]{Definition}
\newtheorem{remark}[theorem]{Remark}

\title[]{$N$-player games and mean field games of moderate interactions}
\author[]{Franco Flandoli  \and Maddalena Ghio \and Giulia Livieri}
\thanks{M.~Ghio and G.~Livieri acknowledge the financial support of UniCredit Bank R$\&$D group through the \textit{Dynamical and Information Research Institute} at the Scuola Normale Superiore. All authors thank Prof. Fausto Gozzi (LUISS Guido Carli), Prof. Luciano Campi (University of Milan) and Prof. Markus Fisher (University of Padova) for useful suggestions.}
\address{F.~Flandoli and M.~Ghio and G.~Livieri: Scuola Normale Superiore, Piazza dei Cavalieri 7, 56126 Pisa, Italy}
\email{\href{franco.flandoli@sns.it}{franco.flandoli@sns.it}}
\email{\href{maddalena.ghio@sns.it}{maddalena.ghio@sns.it}}
\email{\href{giulia.livieri@sns.it}{giulia.livieri@sns.it}}

\date{\today}

\numberwithin{equation}{section}

\begin{document}

\begin{abstract}
We study the asymptotic organization among many optimizing individuals interacting in a suitable ``moderate" way. We justify this limiting game by proving that its solution provides approximate Nash equilibria for large but finite player games. This proof depends upon the derivation  of a law of large numbers for the empirical processes in the limit as the number of players tends to infinity. Because it is of independent interest, we prove this result in full detail. We characterize the solutions of the limiting game via a verification argument.
\end{abstract}

\maketitle
\begin{small}
\noindent \textbf{AMS}: 49N90, 60G09, 60H30, 60K35.\\
\end{small}
\section{Introduction}
\indent The theory of Mean Field Games (MFGs, henceforth) began with the pioneering works of \cite{lasry2007mean} and \cite{huang2006large} to describe the asymptotic organization among a large population of optimizing individuals interacting with each other in a mean-field way and subject to constraints of economic or energetic type. The mean-field interaction enables to reduce the analysis to a control problem for one single representative player, interacting with, and evolving in, the environment created by the aggregation of the other individuals. Intuitively, the system's symmetries will force the players to obey a form of law of large numbers and satisfy a propagation of chaos phenomenon as the size of the population grows. The literature on MFGs is rapidly growing and the application of MFG theory is catching on in areas as diverse as Economics, Biology, Physics, and Machine Learning; hence, it is impossible to give an exhaustive account of the activity on the topic. For this reason, we refer the reader to the lecture notes by \cite{cardaliaguet2012notes} and the two-volume monograph by \cite{carmona2018probabilistic} for a comprehensive presentation of the MFG theory and its applications; the first reference presents the theory from an analytic perspective, whereas the second one from a probabilistic point of view.\\
\indent However, in many practical situations (e.g., in evacuation planning and crowd management at mass gatherings), it stands to reason that a single person interacts only with the few people in the surrounding environment, i.e., each individual has her/his space. A possible mathematical way to describe this type of interaction is through an appropriate rescaling of a given reference function $V$, where $V$ is a sufficiently regular probability density function; see, e.g.,  \cite{oelschlager1985law} and \cite{morale2005interacting}. Denoting by $x$ and $y$ the positions of two individuals (out of a population of $N$) in a $d$-dimensional space, then their interaction can be modelled by:
\begin{equation*}\label{eq:interaction_1}
N^{-1} V^N(x-y),
\end{equation*}
where 
\begin{equation}\label{eq:interaction_2}
V^N(z) = N^{\beta} V(N^{\beta/d}z).
\end{equation}
The parameter $\beta \in (0,1)$ describes how $V$ is rescaled for the total number $N$ of individuals and expresses the so-called \textit{moderate} interaction among the individuals; see \cite{oelschlager1985law}. On the other hand, $\beta = 0$ expresses an interaction of \textit{mean-field} type, whereas $\beta = 1$ generates the so-called \textit{nearest-neighbour} interaction. This paper aims to analyze the asymptotic organization among many optimizing individuals moderately interacting with each other. To the best of our knowledge, the study of this type of asymptotic organization has been performed only in
 \cite{aurell2018mean} and \cite{cardaliaguet2017convergence}. In the former work, authors introduced models for crowd motion, although in a more simplified setting. Indeed, they account for the moderate interaction among the individuals in the cost functional only, \textcolor{black}{although they consider that the position of each pedestrian (in a crowd of $N$ pedestrians) belongs to $\mathbb{R}^{d}$.} Also, in \cite{cardaliaguet2017convergence} only the payoff of a player depends in an increasingly singular way on the players which are very close to her/him. In addition, to avoid issues related to boundary conditions or problems at infinity, in the latter work data are assumed periodic in space. \textcolor{black}{The fact that data are assumed periodic in space and (mostly) that the moderate interaction enters only in the cost functional has a consequence in proving the existence and uniqueness of solutions of the Partial Differential Equation (PDE) MFG system associated with our model; see the discussion here below in the introduction and Section \ref{sec:meanfieldgame}.}\\
\indent\textit{The model.}\,\,The motion of a single-player $X^{N,i}_t$, $t \in [0, T]$, in a population of $N$ individuals is assumed to be modelled as 
\begin{equation}\label{eq:privatestate_evolution}
\begin{split}
X_t^{N, i} = \textcolor{black}{X_0}^{N,i} &+ \int_{0}^{t}\Bigg(\alpha^{N,i}(s) + b\Big(X_s^{N,i}, \frac{1}{N} \sum_{j = 1}^{N} V^{N}(X_s^{N,i}-X_s^{N, j})\Big)\Bigg)\,ds \\
							   &+ W_t^{N, i},\quad t \in [0, T],\quad i \in \left\{1,\ldots,N \right\}.
\end{split}
\end{equation}
Here, $\bm{\alpha}^{N} \doteq (\alpha^{N,1},\ldots, \alpha^{N,N})$ is a vector of strategies that we will specify below, $b$ is a given deterministic function and $W^{N,1}, \ldots, W^{N,N}$ are independent $d$-dimensional Wiener processes defined on some filtered probability space $(\Omega, \mathcal{F}, (\mathcal{F}_t)_{t \in [0,T]}, \mathbb{P})$. We will denote by $\mathbf{X}^{N}_{\textcolor{black}{t}} \doteq (X^{N,1}_{\textcolor{black}{t}}, \ldots, X^{N, N}_{\textcolor{black}{t}})$ the vector of the positions \textcolor{black}{at time $t$} of the $N$ individuals. In addition, $\textcolor{black}{X_0}^{N,i}$, $i = 1, \ldots, N$, are $\mathbb{R}^{d}$-valued independent and identically distributed (\textit{i.i.d}) random variables, \textcolor{black}{ independent of the Wiener processes,} such that $\textcolor{black}{X_0}^{N,i}\overset{d}{\sim} \xi$ (notice that $``\overset{d}{\sim}"$ stands for ``distributed as") where \textcolor{black}{$\xi$ is an auxiliary random variable with} law $\mu_0$ with density $p_0$, i.e. $\mu_0$ is absolutely continuous with respect to (\textit{w.r.t}) the Lebesgue measure. Eq.~\eqref{eq:privatestate_evolution} says that each individual $i$ partially controls its velocity through her/his strategy $\alpha^{N, i}$. However, the velocity depends on her/his position and on the other individuals' in a neighbourhood of $X^{N,i}$. Indeed, the functions $V^{N}(\,\cdot\,)$ (see Eq.~ \eqref{eq:interaction_2}) are mollifiers (see Appendix \ref{app:usefulLemmas} for a precise definition) describing the intermediate regime between the mean-field and the nearest-neighbour interaction. For large $N$ they have a relatively small support and therefore the individual $i$ interacts, via the term $V^{N}(X_s^{N,i}-X_s^{N,j})$, only with few players, indexed by $j$, in a neighbourhood of $X_s^{N,i}$. In particular, the rate of convergence to zero of the support of $V^{N}$ will be such that the number of players $i$ is still very large, in the limit as $N$ tends to infinity, but very small compared to the full population size $N$. It is worth mentioning that it is also possible to let a common disturbance affect all the individuals \citep{huang2006large}\textcolor{black}{, commonly referred to in the MFGs literature as \textit{common noise}; we refer to the second volume by \cite{carmona2018probabilistic} for an overview of this theory}. The common disturbance could be used -- as also pointed out by \cite{aurell2018mean} -- to model an evacuation during, for instance, a fire or a earthquake. We leave, however, the study of this case for future research.\\
\noindent Each player acts to minimize her/his own expected costs according to a given functional over a finite time horizon $[0, T]$. More precisely, player $i$ evaluates a strategy vector $\bm{\alpha}^{N}$ according to the following cost functional
\begin{equation}\label{eq:functional_to_minimize}
J^{N}_{i}(\bm{\alpha}^{N}) \doteq \mathbb{E}\left[ \int_{0}^{T}\Bigg(\frac{1}{2}|\alpha^{N,i}(s)|^2 + f\Big(X_s^{N,i},   \frac{1}{N} \sum_{j = 1}^{N} V^{N}(X_s^{N,i}-X_s^{N, j})\Big)\Bigg)\,ds + g(X_T^{N,i})\right],
\end{equation}
where $\bm{X}^{N}_{\textcolor{black}{t}}$ is the solution of Eq.~\eqref{eq:privatestate_evolution} under $\bm{\alpha}^{N}$. Notice that the cost coefficients $f$ and $g$ are the same for all players. The cost functional $J_{i}^{N}(\bm{\alpha}^{N})$ can be interpreted practically in the following way; see, also, \cite{aurell2018mean}. The first term penalizes the usage of energy, the second term, instead, the trajectories passing through densely crowded areas.  Finally, the final cost $g(\,\cdot\,)$ penalizes deviation from specific target regions. More details on the setting with all the technical assumptions will be given in the next sections.\\
\indent For the class of games just introduced, we focus on the construction of approximate Nash equilibria \citep{lasry2007mean} for the game with a finite number of individuals (i.e.,\,for the $N$-player game) via the solution of the corresponding control problem for one single representative player (i.e.,\,through the solution of the corresponding MFG). Hereafter, we will use the words ``intermediate interactions" and ``moderate interactions" interchangeably.\\ \indent Our main contributions are as follows:
\begin{itemize}
\item We introduce the limit model corresponding to the above $N$-player games as $N$ tends to infinity, namely the MFG of moderate interaction. We formulate both the 
PDE approach to MFGs with moderate interaction and the stochastic formulation; see Definition \ref{def:mfgsystemsolution} and Definition \ref{def:mfgsolutionstoc}, respectively.
\item We prove that the PDE system (or the equivalent mild formulation; see  Lemma \ref{lem:equiv_mild}) admits a solution \textcolor{black}{$(0, \infty)$; see Theorem \ref{th:globalExistence}. Also, we prove that the same system admits a unique solution for $T$ sufficiently small; see Theorem \ref{th:localExistenceUniqueness}}.
\item We prove the existence of a solution in the feedback form to the MFG of moderate interaction; see Theorem \ref{th:VerificationTheorem}. 
 \item We derive, in the limit as the number of different processes in Eq.~\eqref{eq:privatestate_evolution} tends to infinity, law of large numbers for the empirical processes, and we characterize the limit dynamics; see Theorem \ref{th:oelschlager85New}.
\item We prove that any feedback solution of the MFG induces a sequence of approximate Nash equilibria for the $N$-player games with approximation error tending to zero as $N$ tends to infinity; see Theorem \ref{th:Nash_equilibria_MFG}.
\end{itemize}
\indent \textcolor{black}{The MFG system of PDEs associated with our model takes the form of a backward Hamilton-Jacobi equation coupled with a forward Kolmogorov equation. In particular, it is a second-order MFG system with \textit{local} coupling or of \textit{local} type. Many authors have studied this type of system in the last years; see \cite{lasry2006jeux, lasry2007mean, porretta2015weak, gomes2016regularity, cardaliaguet2020introduction}. However, the framework in these works deviates from ours' for two main reasons. First, the authors consider that the state space is the $d$-dimensional torus $\mathbb{T}^{d}$ and not all the space $\mathbb{R}^{d}$. Second, and most importantly, they do not consider dependence on the local density of measure in the dynamics; see the term $b(x, p(t, x))$ in the first equation in Eq.~\eqref{eq:pdesystem}. We prove\footnote{\textcolor{black}{The authors warmly thank one of the two anonymous Referees for her/his suggestion to look at the Hopf-Cole reduction, to prove global in time existence, because of the quadratic structure of our Hamiltonian.}} the existence of solutions of the PDE MFG system for any $T >0$ via the Brouwer-Schauder fixed point theorem. Instead, we will not be able to prove the uniqueness of such solutions under the standard monotonicity assumption for any $T > 0$ but only for small $T$ via the contraction principle, the difficulty arising precisely from the dependence on the local density in the dynamics.}\\
\indent The proof of the existence of a MFG solution is based on a verification argument. We identify the unique solution of the PDE system of the MFG with moderate interaction with the feedback control solution of the MFG in its stochastic formulation.  In our case, the value function of the representative player is not  ``regular enough", and so, in order to apply  It\^{o} formula, some work based on standard mollification arguments will be needed; see Appendix \ref{app:hjandfpmild}, Subsection \ref{subsec:proofVerificationTheorem}.\\
\indent The proof of Theorem \ref{th:oelschlager85New} on the characterization of the limit dynamics of the empirical processes is one of the main achievements of this work. It represents a version of the superb result of \cite{oelschlager1985law} on the study of the macroscopic limit of moderately interacting diffusion particles. Contrary to us, \cite{oelschlager1985law} does not assume the absolute continuity of $\mu_0$ with respect to the Lebesgue measure. Admittedly, this would be an additional technicality that would not add to the present work's conceptual advancements. On the other hand, we can show the validity of Theorem \ref{th:oelschlager85New} under a more general assumption on the SDE drift in Eq.~\eqref{eq:privatestate_evolution}. In \cite{oelschlager1985law} a more strict Lipschitz condition on the drift (see Eq.~(1.5) in his work) is imposed; this condition is used to prove the uniqueness of the solution of a certain (deterministic) equation that characterizes the limit dynamics of the empirical processes. We believe that this paper's assumptions lead to a much more comprehensive understanding of the problem at hand. Because it is of independent interest, we will devote the entire Section \ref{sec:moderate} to the proof of the propagation of chaos result.\\
\indent The proof of Theorem \ref{th:Nash_equilibria_MFG} of approximate Nash equilibria is based on weak convergence arguments and controlled martingale problems, whose use has a longstanding tradition; \textcolor{black}{see, for instance, \cite{funaki1984certain}, \cite{oelschlager1984martingale},  \cite{huang2006large}, as well \cite{carmona2018probabilistic}, Section 6.1 of the second volume}. However, contrary to those works, we have to study the passage to the many player (particle) limit in the presence of a deviating player, which destroys the prelimit systems' symmetry. We will use an argument based on relaxed controls.\\
\indent\textit{Structure of the paper}\,\, The rest of this paper is organized as follows. Section \ref{sec:preliminaries} introduces some terminology and notation and sets the main assumptions on the dynamics and on the cost functionals. Section \ref{sec:N_player} describes the setting of $N$-player games with moderate interaction, while Section \ref{sec:meanfieldgame} introduces the corresponding MFG.  In Section \ref{sec:moderate}, one of the main results, namely the derivation of a law of large numbers for the empirical processes, is stated and proved. Section \ref{sec:nash} contains the result on the construction of approximate Nash equilibria for the $N$-player game from a solution of the limit problem. The technical results used in the paper are all gathered in the Appendix, including the aforementioned \textcolor{black}{existence} and uniqueness result for the PDE system and the proof of the existence of a MFG solution in Appendix \ref{app:hjandfpmild}, and bounds on H\"{o}lder-type semi-norm to prove the results of Section \ref{sec:moderate} in Appendix \ref{app:holder_heat_bound} and Appendix \ref{app:holder_heat_bound_1}.

\section{Preliminaries and assumptions}\label{sec:preliminaries}
\noindent Let $d \in \mathbb{N}$ be the dimension of the space of private state and of the noise.  We equip the spaces $\mathbb{R}^{d}$, $d \in \mathbb{N}$, with the standard Euclidean norm, which will be denoted by $|\,\cdot\,|$.  Instead $T > 0$ is the finite time horizon.\\
\indent For $\mathcal{S}$ Polish space we let $\mathcal{P}(\mathcal{S})$ denote the space of probability measures on $\mathcal{B}(\mathcal{S})$, the Borel sets of $\mathcal{S}$. For $s \in \mathcal{S}$ we let $\delta_s$ indicate the Dirac measure concentrated in $s$. If $\mathcal{P}(\mathcal{S})$ is equipped with the topology of weak convergence of probability measures, then $\mathcal{P}(\mathcal{S})$ is a Polish space. \textcolor{black}{In particular, $\text{C}([0,T] ; \mathcal{P}(\mathcal{S}))$ denotes the space of continuous flow of measures.}\\
\indent\textcolor{black}{We set $\mathcal{X} \doteq \text{C} ([0,T] ; \mathbb{R}^{d})$ and we equip it with the topology of uniform convergence; the space $\mathcal{X}$ with this topology is a Polish space. Given $N \in \mathbb{N}$, we will use the usual identification of $\mathcal{X}^{N} = \times^{N} \mathcal{X}$ with the space $ \text{C} ([0,T] ; \mathbb{R}^{d\cdot N})$;  $\mathcal{X}^{N}$ is equipped with the topology of uniform convergence.  For $\ell \in \mathbb{R}_{+}$, we denote by $\text{C}_b^{\ell}(\mathbb{R}^{d} ; \mathbb{R}^{d})$ the set of $\mathbb{R}^{d}$-valued functions on $\mathbb{R}^{d}$ with bounded $\ell$-th derivative, and by $\text{C}_c^{\ell}(\mathbb{R}^{d}; \mathbb{R}^{d})$ the set of $\mathbb{R}^{d}$-valued functions on $\mathbb{R}^{d}$ with compact support and continuous $\ell$-th derivative. We will use simply $\text{C}_b(\mathbb{R}^{d})$, $\text{C}_b^{\ell}(\mathbb{R}^{d})$ and $\text{C}_c^{\ell}(\mathbb{R}^{d})$ when the functions are real-valued. Moreover, $\text{C}^{\ell}([0,T]; \text{C}_b(\mathbb{R}^{d}))$ denotes the space of $\text{C}_b(\mathbb{R}^{d})$-valued functions on $[0,T]$ with continuous $\ell$-th derivative; analogous definitions hold if $\text{C}_b(\mathbb{R}^{d})$ is replaced with either $\text{C}_b^{\ell}(\mathbb{R}^{d})$ or $\text{C}_c^{\ell}(\mathbb{R}^{d})$.\\
\indent Similarly, we denote by $\text{C}([0,T] \times \mathbb{R}^{d}; \mathbb{R}^{d})$ the set of $\mathbb{R}^{d}$-valued continuous functions on $[0,T] \times \mathbb{R}^{d}$ and with $\text{C}^{1, 2}([0,T] \times \mathbb{R}^{d}; \mathbb{R}^{d})$ the set of $\mathbb{R}^{d}$-valued continuous functions on $[0,T] \times \mathbb{R}^{d}$ with continuous first (resp.~second) derivative with respect to the time (resp.~space); analogous definitions (cfr. the characterizations in the previous paragraph) hold for the spaces $\text{C}_b^{1, 2}([0,T] \times \mathbb{R}^{d}; \mathbb{R}^{d})$, $\text{C}_c^{1, 2}([0,T] \times \mathbb{R}^{d}; \mathbb{R}^{d})$. Again, we will use simply $\text{C}([0,T] \times \mathbb{R}^{d})$, $\text{C}^{1,2}([0,T] \times \mathbb{R}^{d})$, $\text{C}_b^{1, 2}([0,T] \times \mathbb{R}^{d})$, $\text{C}_c^{1, 2}([0,T] \times \mathbb{R}^{d})$ when the functions are real-valued. In particular, notice that $\text{C}([0,T];\text{C}_b(\mathbb{R}^{d})) \subset \text{C}_b([0,T]\times\mathbb{R}^{d})$.}
\indent As usual, $\nabla$ and $\Delta$ denote the gradient and the Laplacian operator, respectively. Finally, for the sake of simplicity, we write $i \in [[N]]$ in place of $i = 1, \ldots, N$.\\\\
\indent Now let
\begin{equation*}
\begin{split}
& b\,:\,\mathbb{R}^{d} \times \mathbb{R}_{+} \rightarrow \mathbb{R}^{d},\\
& f\,:\,\mathbb{R}^{d} \times \mathbb{R}_{+} \rightarrow \mathbb{R},\quad\quad g\,:\,\mathbb{R}^{d} \rightarrow \mathbb{R}.
\end{split}
\end{equation*}
\noindent The function $b$ will denote the drift, while $f$ and $g$ will quantify the running and the terminal costs, respectively. Let us make the following assumptions:
\begin{itemize}
\item[(H1)] $b$ and $f$ are Borel measurable functions, continuous and such that there exist two constants $C, L > 0$ for which it holds that 
\begin{equation*}
\begin{split}
& |b(x, p)| + |f(x, p)| \leq C,\\
& |b(x,p) - b(y,q)| + |f(x,p) - f(y,q)| \leq L(|x-y| + |p-q|)
\end{split}
\end{equation*}
for all $x, y \in \mathbb{R}^d$, $p, q \in \mathbb{R}_{+}$.
\item[(H2)] $g$ is a Borel measurable function such that $g, \partial_{x_i} g \in \text{C}_b(\mathbb{R}^{d})$, $i \in [[d]]$.
\item[(H3)] For each $N\in\mathbb{N}$, for some $\beta \in (0, 1/2)$ and some $V \in \text{C}_{c}^{1}(\mathbb{R}^{d}) \cap \mathcal{P}(\mathbb{R}^{d})$ we have
\begin{equation}\label{eq:Vn}
V^{N}(x) \doteq N^{\beta} V(N^{\frac{\beta}{d}} x),\quad x \in \mathbb{R}^{\textcolor{black}{d}},
\end{equation}
\textcolor{black}{where, we remind, $\text{C}_{c}^{1}(\mathbb{R}^{d})$ is the space of continuous functions on $\mathbb{R}^d$ with compact support and continuous first derivatives, while $\mathcal{P}(\mathbb{R}^{d})$ denotes the probability measures on $\mathbb{R}^d$. In particular, $\text{C}_{c}^{1}(\mathbb{R}^{d}) \cap \mathcal{P}(\mathbb{R}^{d})$ denotes the set of probability measures with a density that has compact support and that is differentiable.}
\item[(H4)]  \textcolor{black}{The law} $\mu_0 \in \mathcal{P}(\mathbb{R}^{d})$ \textcolor{black}{is} absolutely continuous  with respect to the Lebesgue measure on $\mathbb{R}^{d}$ and with density $p_0 \in \text{C}_{b}(\mathbb{R}^{d})$ satisfying the following condition:
$$\int_{\mathbb{R}^{d}} e^{\lambda |x|} p_0(x)\,dx < \infty $$
for all $\lambda > 0$. 
\end{itemize}
\section{N-player games}\label{sec:N_player}
\noindent Let $N \in \mathbb{N}$ be the number of players. Denote by $X^{N,i}_t$ the private state of player $i$ at time $t \in [0,T]$. The evolution of the players' state depends on the strategies they choose and on the initial distribution of states, which we indicate by $\mu^{N}_0$ (thus, $\mu^{N}_0 \in \mathcal{P}(\mathbb{R}^{N \times d})$). We assume that $\mu^{N}_0$ can be factorized and that for each $\mu_0$ hypothesis \text{(H4)} is in force. \textcolor{black}{Here, we consider players using feedback strategies with full state information, i.e. strategies $\alpha_t^{N,i} = \alpha(t, \bm{X}_t^{N})$ where $\alpha \in \text{C}_b([0,T] \times \mathbb{R}^{d\cdot N} ; \mathbb{R}^{d})$  that are uniformly bounded by some constant $C>0$. Thus, let $\mathcal{A}_{C}^{N, 1, fb}$ denote the set of all these individual strategies. A vector $\textcolor{black}{\bm{\alpha}^{N}\doteq}(\alpha^{N,1},\ldots,\alpha^{N,N})$ of individual strategies is called a strategy vector or strategy profile. \textcolor{black}{We denote with $\mathcal{A}_C^{N, fb}$ the set of all vectors $\bm{\alpha}^{N}$ of feedback strategies for the $N$-player game that are uniformly bounded by some constant $C>0$.}  Given a  vector of $N$-player feedback strategies $\bm{\alpha}^{N}$, consider the system of equations}
\begin{equation}\label{eq:privatestate_evolution_N_player}
\begin{split}
X_t^{N, i} = \textcolor{black}{X_0}^{N, i} &+ \int_{0}^{t}\Bigg(\alpha(s, \textcolor{black}{\bm{X}_s^{N}}) + b\Big(X_s^{N,i}, \frac{1}{N} \sum_{j = 1}^{N} V^{N}(X_s^{N,i}-X_s^{N, j})\Big)\Bigg)\,ds \\
							   &+ W_t^{N, i},\quad t \in [0, T],\, i \in \textcolor{black}{[[N]]},
\end{split}
\end{equation}
where $\bm{X}^{N}_{\textcolor{black}{t}} = (X^{N,1}_{\textcolor{black}{t}}, \ldots, X^{N,N}_{\textcolor{black}{t}})$ and $W^{N,1}, \ldots, W^{N,N}$ are independent Wiener processes defined on some filtered probability space $(\Omega, \mathcal{F}, (\mathcal{F}_t), \mathbb{P})$ \textcolor{black}{satisfying the usual conditions. The initial conditions $X_0^{N, i}$ are \textit{i.i.d.} $\mathcal{F}_0$-measurable random variables, each with law $\mu_0 \in \mathcal{P}(\mathbb{R}^{d})$ and independent of the Wiener processes,} the functions $V^{N}(\,\cdot\,)$ are mollifiers (see hypothesis (H3)) through which we obtain the interaction of moderate type among the players. A solution of Eq.~\eqref{eq:privatestate_evolution_N_player} under $\bm{\alpha}^{N}$ with initial distribution $\mu_0^{N}$ is a triple $((\Omega, \mathcal{F}, (\mathcal{F}_t), \mathbb{P}), \bm{W}^{N}, \bm{X}^{N})$ where $(\Omega, \mathcal{F}, (\mathcal{F}_t), \mathbb{P})$ is a filtered probability space satisfying the usual hypotheses, $\bm{W}^{N} = (W^{N,1},\ldots, W^{N,N})$ a vector of independent $d$-dimensional $(\mathcal{F}_t)$-Wiener processes, and $\bm{X}^N = (X^{N,1}, \ldots, X^{N,N})$ a vector of continuous $\mathbb{R}^{d}$-valued $(\mathcal{F}_t)$-adapted processes such that Eq.~\eqref{eq:privatestate_evolution_N_player} holds $\mathbb{P}$-almost surely with strategy vector $\bm{\alpha}^N$ and $\mathbb{P} \circ (\bm{\textcolor{black}{X_0}}^{N})^{-1} = \mu_0^{N}$, \textcolor{black}{each $X_0^{N, i}$ for $i\in[[N]]$ being independent of the Wiener processes}. The \textit{i}-th player evaluates a (feedback) strategy vector $\bm{\alpha}^{N}$ according to the cost functional
\begin{equation}\label{eq:cost_functional_N_player}
\begin{split}
J^{N}_{i}(\bm{\alpha}^{N}) \doteq \mathbb{E}\Bigg[ \int_{0}^{T}\Bigg(\frac{1}{2}|\alpha(s, \textcolor{black}{\bm{X}_s^{N}})|^2 + f\Big(X_s^{N,i},   \frac{1}{N} \sum_{j = 1}^{N} V^{N}(X_s^{N,i}  -X_s^{N, j})\Big)\Bigg)\, &ds \\
																					&+ g(X_T^{N,i})\Bigg],
\end{split}
\end{equation}
where $\bm{X}^{N}_{\textcolor{black}{t}} = (X^{N,1}_{\textcolor{black}{t}}, \ldots, X^{N,N}_{\textcolor{black}{t}})$ and $((\Omega, \mathcal{F}, (\mathcal{F}_t), \mathbb{P}), \bm{W}^{N}, \bm{X}^{N})$ is a solution of Eq.~\eqref{eq:privatestate_evolution_N_player}  under $\mu_0^N$. The cost functional is well defined thanks to the hypothesis \text{(H1)}.\\
\indent Given a strategy vector \textcolor{black}{$\bm{\alpha}^{N}\in \mathcal{A}^{N, fb}_C$} and an individual strategy \textcolor{black}{$\beta\in \mathcal{A}^{N, 1, fb}_C$}, let \textcolor{black}{$[\bm{\alpha}^{N, -i}, \beta]\in \mathcal{A}^{N, fb}_C$} indicate the strategy vector that is obtained from $\bm{\alpha}^N$ by replacing $\alpha^{N,i}$, the strategy of player $i$, with $\beta$. The correct interpretation of optimization of the cost functional $J_i^{N}(\bm{\alpha}^{N})$ in Eq.~\eqref{eq:cost_functional_N_player} -- classical in game theory -- would be the concept of \textit{Nash Equilibrium}. In the case of a large number of players, our goal will be to prove the validity of a weaker \textcolor{black}{equilibrium} concept, that is the concept of \textit{$\varepsilon$-Nash equilibrium}, introduced in the theory of MFGs.
\begin{definition}[$\varepsilon$-Nash equilibria]
Let $\varepsilon \geq 0$. A strategy vector $\bm{\alpha}^{N}$ is called an $\varepsilon$-Nash equilibrium for the $N$-player game if for every $i \in [[N]]$
\begin{equation}
J^{N}_i(\bm{\alpha}) \leq J^{N}_{i}([\bm{\alpha}^{N, -i}, \beta]) + \varepsilon,
\end{equation}
for all \textcolor{black}{\textbf{admissible}} single player strategies $\beta$, \textcolor{black}{i.e., strategies that belong to $\mathcal{A}_{C}^{N, 1, fb}$}.
\end{definition}
\indent\textcolor{black}{If $\bm{\alpha}^{N}$ is an $\varepsilon$-Nash equilibrium with $\varepsilon = 0$, then $\bm{\alpha}^{N}$ is called \textit{Nash equilibrium}.}\\
\textcolor{black}{In our framework, we consider strategy vectors $\bm{\alpha}^N$ belonging to $\mathcal{A}^{N, fb}_C$, where we will later in the work fix the constant $C$ to be equal to $K\left( T,b,f,p_{0},g\right)$ defined in Eq.\eqref{eq:basic_const}. We say that a single player strategy $\beta$ is admissible (i.e. it is an admissible deviation from equilibrium) for a player $i\in[[N]]$ if it belongs to $\mathcal{A}^{N, 1, fb}_C$ where the constant $C$ is intended to be fixed.}
\section{Mean field games}\label{sec:meanfieldgame}
\noindent Let $T>0$ be the finite time horizon and $b, f, p_0, g$ as in Section \ref{sec:preliminaries}. Let us introduce the PDE approach to MFGs with moderate interaction via the following coupled system of backward Hamilton-Jacobi Bellman equation and Kolmogorov forward equation, called \textit{PDE system}:
\begin{equation}\label{eq:pdesystem}
\begin{cases}
-\partial _{t}u-\frac{1}{2}\Delta u-b(x,p(t,x))\cdot \nabla u+\frac{1}{2}\left\vert \nabla u\right\vert ^{2}=f(x,p(t,x)),\quad (t,x)\in [0,T)\times \mathbb{R}^{d},\\ 
\partial _{t}p-\frac{1}{2}\Delta p+\text{div}{[p(t,x)(-\nabla u(t,x)+b(x,p(t,x)))]}=0,\quad\quad (t,x)\in (0,T]\times \mathbb{R}^{d},\\ 
p(0,\,\cdot\,)=p_{0}(\,\cdot\,)\quad x\in \mathbb{R}^{d},\quad u(T,\,\cdot\,)=g(\,\cdot\,),\quad\quad\quad\quad\quad\quad\quad\,\, x\in \mathbb{R}^{d},
\end{cases}
\end{equation}
for all $(x,p) \in \mathbb{R}^{d}\times \mathbb{R}_{+}$. Precisely, the first equation of the PDE system is the Hamilton-Jacobi Bellman equation with a quadratic cost for the value function $u$ of the representative player. Instead, the second one is the Kolmogorov forward equation for the density $p(t,\,\cdot\,)$ of the representative player. \textcolor{black}{As said in the introduction, the PDE MFG system is of local type with the dependence on the local density $p(t, x)$ appearing both on the dynamics, via the term $b(x, p(t, x))$, and on the running cost, via the term $f(x, p(t, x))$. In addition, the state space is $\mathbb{R}^{d}$.}\\
\indent The notion of solution we consider for the PDE system is the one in Definition \ref{def:mfgsystemsolution} below, where we let $\mathcal{A}$ denote the following operator:
\begin{equation}\label{eq:operator}
\mathcal{A} \doteq \partial_t - \frac{1}{2}\Delta.
\end{equation}
\begin{definition}[MFG solution, PDE formulation]\label{def:mfgsystemsolution}
A weak solution of the PDE system is a pair $(u, p)$ such that:
\begin{itemize}
\item[\textit{(i)}] \textcolor{black}{$u$, $\partial_i u$ and $p \in \text{C}_b([0,T] \times \mathbb{R}^{d})$ for all $i \in [[\,d\,]]$};  
\item[\textit{(ii)}] for all \textcolor{black}{$\varphi, \psi \in \text{C}^{1,2}_{c}([0,T] \times \mathbb{R}^{d})$} and all $t \in [0, T]$ the following two equations
\begin{equation}\label{eq:mfgsystem_eq1}
\begin{split}
\left\langle u\left( t \right) ,\varphi \left( t\right) \right\rangle&- \left\langle g,\varphi \left( T\right) \right\rangle +\int_{t}^{T}\left\langle
u\left( s\right) ,\mathcal{A} \varphi \left(s\right) \right\rangle ds \\
&=\int_{t}^{T}\left\langle b(\,\cdot\,,p(s))\cdot \nabla u\left( s\right) -\frac{1}{2}\left\vert \nabla u\left( s\right) \right\vert ^{2}+f(\,\cdot\,,p(s)),\varphi \left(s\right) \right\rangle ds,
\end{split}
\end{equation}
\begin{equation}\label{eq:mfgsystem_eq2}
\begin{split}
\left\langle p\left( t\right) ,\psi \left( t\right) \right\rangle &-\left\langle p_{0},\psi \left( 0\right) \right\rangle
-\int_{0}^{t}\left\langle u\left( s\right) ,\mathcal{A}\psi \left( s\right) \right\rangle ds \\
&= \int_{0}^{t}\left\langle {p(s)(-\nabla u(s)+b(\,\cdot\,,p(s))),\nabla }\psi\left( s\right) \right\rangle ds.
\end{split}
\end{equation}
hold.
\end{itemize}
\end{definition}
\vspace{5mm}

\noindent We now state and prove that under the regularity condition \textit{(i)} in Definition \ref{def:mfgsystemsolution} the system in Eqs.~\eqref{eq:mfgsystem_eq1}--\eqref{eq:mfgsystem_eq2} admits an equivalent \textit{mild formulation}. To this end, set \textcolor{black}{$G(t, x-y)$}  the density of $x + W_t$, where $W_t$ is a standard blackian motion, $t \in [0, T]$ and $x, y \in \mathbb{R}^{d}$, and introduce the notation $\mathcal{P}_t$ for the associated semi-group,
\begin{equation}\label{eq:heat_kernel}
(\mathcal{P}_t h)(x)\doteq \int_{\mathbb{R}^{d}} G(t, x-y) h(y)\,dy,
\end{equation}
defined on functions $h \in \text{C}_b(\mathbb{R}^{d})$. By taking, for all $t \in [0,T]$, in the Eqs.~\eqref{eq:mfgsystem_eq1} and \eqref{eq:mfgsystem_eq2} the functions \textcolor{black}{$\varphi( t )$} and \textcolor{black}{$\psi( t )$} as the function $y \mapsto G(t, x-y)\,h(y)$, with $x$ a given parameter, one can show the equivalence between the weak formulations of Eq.~\eqref{eq:mfgsystem_eq1} and \eqref{eq:mfgsystem_eq2} and the following mild formulation. This is the content of the following lemma.
\begin{lemma}
\label{lem:equiv_mild}
\noindent Let $(u, p)$ a pair with the regularity of point \textit{(i)} in Definition \ref{def:mfgsystemsolution}. Then \textit{(ii)} in the same definition is equivalent to the validity, for all $t\in \left[ 0,T\right] $, of the following system:
\begin{equation}\label{eq:mfgsystemmild_eq1}
\begin{split}
&u\left( t\right) =\mathcal{P}_{T-t}g-\int_{t}^{T}\mathcal{P}_{s-t}\left(b\left(\,\cdot\,,p\left( s\right) \right) \cdot \nabla u\left( s\right) 
-\frac{1}{2}\left\vert \nabla u\left( s\right) \right\vert ^{2}+f(\,\cdot\,,p(s))\right) ds\\
\end{split}
\end{equation}
and
\begin{equation}\label{eq:mfgsystemmild_eq2}
\begin{split}
&p\left( t\right) =\mathcal{P}_{t}p_{0}-\int_{0}^{t}\nabla \mathcal{P}_{t-s}\left( p\left( s\right) \left( \nabla u\left( s\right) -{b(\,\cdot\,,p(s))}\right) \right) ds,\\
\end{split}
\end{equation}
where in the last integral we understand that 
\begin{equation}\label{eq:mfgsystemmild_eq3}
\left( \nabla \mathcal{P}_{t-s}h\right) \left( x\right) =\int_{\mathbb{R}%
^{d}}\nabla _{x}G\left( t-s,x-y\right) h\left( y\right) dy.
\end{equation}
A solution of this integral system with the regularity of point \textit{(i)} in Definition \eqref{def:mfgsystemsolution} is called a \textbf{mild solution}.
\end{lemma}
\begin{proof}
See Appendix \ref{app:hjandfpmild}, Subsection \ref{subsec:equiv_mild}, where we give a sketch of the (less classical) proof for the backward equation \eqref{eq:mfgsystemmild_eq1}.
\end{proof}

\textcolor{black}{\\
\noindent Now, we prove that there exists $(u, p)$ weak solution (cfr. Definition~\ref{def:mfgsystemsolution}) of the PDE MFG system \ref{eq:pdesystem} in $(0, \infty)$. In order to do so, we use the Hopf-Cole transform for quadratic Hamiltonians (see, e.g. Remark 1.13 in \cite{cardaliaguet2020introduction}) and we consider the following auxiliary system
\begin{equation}\label{eq:pdesystem_hopf_cole}
\begin{cases}
\partial_t w + \frac{1}{2} \Delta w + b(x,p(t,x))\cdot \nabla w = w\,f(x,p(t,x)),\quad\quad\,\, (t,x)\in [0,T)\times \mathbb{R}^{d}, \\ 
\partial _{t}p-\frac{1}{2}\Delta p+\text{div}{\left[p(t,x)\left(\frac{\nabla w}{w} + b(x, p(t, x))\right)\right]}=0,\quad\quad (t,x)\in (0,T]\times \mathbb{R}^{d}, \\ 
p(0,\,\cdot\,)=p_{0}(\,\cdot\,)\quad x\in \mathbb{R}^{d},\quad w(T,\,\cdot\,)=\exp(-g(\,\cdot\,)),\quad\quad\quad\quad\quad\quad\,\,\,\, x\in \mathbb{R}^{d}.
\end{cases}
\end{equation}
Notice that if $(w, p)$ is a weak solution of the previous system such that $p, w, \partial_i w \in \text{C}_b([0,T]\times\mathbb{R}^{d})$, $i \in [[d]]$, then $w(t, x) \geq e^{-(\|g \|_{\infty} + T \|f \|_{\infty})}$ by strong maximum principle. Therefore, the ratio $\frac{\nabla w}{w} \in \text{C}_b([0,T] \times \mathbb{R}^{d} ; \mathbb{R}^{d})$ with a bound that depends only on the infinity norms of the coefficients; precisely:
\begin{equation}
\left\| \frac{\nabla w}{w}\right\|_{\infty} \leq C_w(g, f, b, T).
\end{equation}
This observation justifies the following definition, analogous to Definition \ref{def:mfgsystemsolution}.
\begin{definition}[MFG solution, PDE formulation - I]\label{def:mfgsystemsolutionHopf}
Let $p_{0} \in \text{C}_{b}\left( \mathbb{R}^{d}\right)$ a given probability density and $g\in \text{C}_{b}\left( \mathbb{R}^{d}\right)$, also given. A weak solution of the PDE system \eqref{eq:pdesystem_hopf_cole} is a pair $(w, p)$ such that $w, \partial_i w$ and $p \in \text{C}_{b}\left( \left[ 0,T\right] \times \mathbb{R}^{d}\right)$ for all $i \in [[ d ]]$, $w\left( t,x\right) \geq e^{-\left( \left\Vert g\right\Vert _{\infty }+T\left\Vert f\right\Vert _{\infty }\right) }$ and the system is satisfied in the weak sense as in Definition \ref{def:mfgsystemsolution}.
\end{definition}
\noindent In particular, the weak formulation in Definition \ref{def:mfgsystemsolutionHopf} is equivalent to the validity, for all $t \in [0,T]$, of the following system
\begin{equation}\label{eq:mfgsystemmild_1_eq1}
\begin{split}
w\left( t\right) =\mathcal{P}_{T-t}\exp \left( -g\right) -\int_{t}^{T}%
\mathcal{P}_{s-t}\left( b\left( \cdot ,p\left( s\right) \right) \cdot \nabla
w\left( s\right) -w\left( s\right) f\left( \cdot ,p\left( s\right) \right)
\right) ds
\end{split}
\end{equation}
and
\begin{equation}\label{eq:mfgsystemmild_2_eq2}
\begin{split}
p\left( t\right)  =\mathcal{P}_{t}p_{0}+\int_{0}^{t}\nabla \mathcal{P}%
_{t-s}\left( p\left( s\right) \left( \frac{\nabla w\left( s\right) }{w\left(
s\right) }+b\left( \cdot ,p\left( s\right) \right) \right) \right) ds,
\end{split}
\end{equation}
where the quantity $\nabla \mathcal{P}_{t-s}$ is defined in Lemma \ref{lem:equiv_mild}, Eq.~\eqref{eq:mfgsystemmild_eq3}. The proof of such equivalence is the same as in Lemma \ref{lem:equiv_mild} and we decide to omit it for the sake of space.\\
\noindent To prove global existence of weak solutions, we need the following additional assumption on $p_0$:
\begin{itemize}
\item[(H5)] There exists a continuous function $\rho :\mathbb{R}^{d}\rightarrow \left( 0,\infty \right) $ such that 
\begin{equation*}
\lim_{\left\Vert x\right\Vert \rightarrow \infty }\rho \left( x\right) =0\quad\text{and}\quad p_{0}\left( x\right) \leq \rho \left( x\right) 
\end{equation*}
for all $x\in \mathbb{R}^{d}$. Moreover $p_{0}\in \text{C}_{b}^{\alpha}(\mathbb{R}^{d})$ for some $\alpha >0$ and $\rho ^{-1}\in
\text{C}^{2}\left( \mathbb{R}^{d}\right) $ with $\left\Vert \Delta \rho ^{-1}\right\Vert _{\infty }+\left\Vert \nabla \rho^{-1}\right\Vert _{\infty }<\infty$.
\end{itemize}
\noindent Notice that the latter assumption on $\rho^{-1}$ is not restrictive. Indeed, smoothness of $\rho^{-1}$ can be obtained by regularization and the bounds on $\left\Vert \Delta \rho ^{-1}\right\Vert _{\infty }$ and $\left\Vert \nabla \rho ^{-1}\right\Vert _{\infty }$ are true if $\rho $ decays slowly, monotonically and radially, which can always be assumed without loss of generality. We are now ready to prove the existence of a weak solution of the PDE system \eqref{eq:pdesystem_hopf_cole}; this is the content of the following theorem, whose proof  is relatively standard but some new details -- up to our knowledge -- are due to the fact that the space is $\mathbb{R}^{d}$ instead of a bounded set.
\begin{theorem}\label{th:globalExistence}
There exists a weak solution $\left(
w,p\right) $ on $\left[ 0,T\right] $ of system \eqref{eq:pdesystem_hopf_cole}. Moreover, the pair
\begin{equation*}
\left( u,p\right) \doteq \left( -\log w,p\right) 
\end{equation*}%
is a weak solution of the system \eqref{eq:pdesystem}.
\end{theorem}
\begin{proof}
See Appendix \ref{app:hjandfpmild}, Subsection \ref{subsec:global_existence}.
\end{proof}}

\noindent Now, we prove that the system \eqref{eq:pdesystem} admits a unique solution for $T$ sufficiently small via the contraction principle; indeed, the following theorem holds.
\begin{theorem}[Local well posedness]\label{th:localExistenceUniqueness}
There exists a unique weak (or mild) solution of the MFG system \eqref{eq:mfgsystemmild_eq1}-\eqref{eq:mfgsystemmild_eq2}, for $T$ sufficiently small.
\end{theorem}
\begin{proof}
See Appendix \ref{app:hjandfpmild}, Subsection \ref{subsec:local_existence_uniqueness}.
\end{proof}

\noindent Next, let $T>0$ indicate (as before) the finite time horizon, and let $b, f, p_0, g$ as in Section \ref{sec:preliminaries}. If the PDE system in Eq.~\eqref{eq:pdesystem} has a unique weak (or mild) solution $(u, p)$, then we denote by $K(T, b, f, p_0, g)$ the following constant:
\begin{equation}\label{eq:basic_const}
K\left( T,b,f,p_{0},g\right)  \doteq \sup_{t\in \left[ 0,T\right] ,x\in \mathbb{R}%
^{d}}\left\vert \nabla u\left( t,x\right) \right\vert .
\end{equation}

\subsection{Feedback MFG with given density}\label{subsec:feedback}
We started the section by formulating the PDE approach to MFGs of moderate interaction. Here, instead, we introduce the corresponding stochastic (\textcolor{black}{feedback first and open-loop in the next subsection}) formulation. \\
\textcolor{black}{Let $K>0$. In order to make precise our definition of (feedback) MFG solution, we introduce the following notation:}
\begin{itemize}
\item[\textit{(i)}] \textcolor{black}{We denote by $\mathcal{A}^{fb}_K$ the set of feedback controls for the MFG, which is defined as the set of functions \textcolor{black}{$\alpha \in \text{C}_b([0,T] \times \mathbb{R}^{d} ; \mathbb{R}^{d})$} bounded by $K$.}
\item[\textit{(ii)}]\textcolor{black}{Next, given the function $p$ as in Definition \ref{def:mfgsystemsolution}, given an admissible control $\alpha\in\mathcal{A}^{fb}_K$, we consider the equation}
\begin{equation}
\label{eq:mfgstate_fb}
\textcolor{black}{X_t = X_0 + \int_{0}^{t} (\alpha(s,X_s) + b(X_s, p(s, X_s)))\,ds + W_t,\quad t \in [0, T],}
\end{equation}
\textcolor{black}{where $X_0$ is  a $\mathcal{F}_0$-measurable random variable distributed as $\mu_0$ having density $p_0$ while $W$ is a $d$-dimensional Wiener process defined on some filtered probability space $(\Omega, \mathcal{F}, (\mathcal{F}_t), \mathbb{P})$.}
\item[\textit{(iii)}]\textcolor{black}{Finally, we consider the following cost functional}
\begin{equation*}
\textcolor{black}{J(\alpha) \doteq \mathbb{E}\left[\int_{0}^{T} \frac{1}{2}|\alpha(s,X_s)|^2 + f(X_s, p(s, X_s))\,ds + g(X_T)\right]}
\end{equation*}
\textcolor{black}{and we say that $\alpha^{*}\in \mathcal{A}^{fb}_{K}$ is an optimal control if it is a minimizer of $J$ over $\mathcal{A}^{fb}_{K}$, i.e. if $J(\alpha^{*}) = \inf_{\alpha \in \mathcal{A}^{fb}_K} J(\alpha)$.} 
\end{itemize}

\noindent \textcolor{black}{The notion of solution we will consider in the feedback case is then the following:}
{\color{black}{
\begin{definition}[MFG solution, stochastic \textbf{feedback} formulation]\label{def:mfgsolutionstocFB}
Let $T>0$ be the finite time horizon and $b, f, p_0, g$ as in (H1)-(H2) and (H4); see Section \ref{sec:preliminaries}. 
Then a \textbf{feedback MFG solution} for bound $K>0$ is a pair $(\alpha^*,p)$ such that:
\begin{itemize}
\item[\textit{(i)}] $p\in C_b([0,T]\times \mathbb{R}^d)$ and $\alpha^* \in \mathcal{A}^{fb}_K$;
\item[\textit{(ii)}] Given $p\in C_b([0,T]\times \mathbb{R}^d)$, $\alpha^* \in \mathcal{A}^{fb}_K$ is an optimal control for the cost functional $J(\cdot)$ (in the sense of item (iii) above);
\item[\textit{(iii)}] For any weak solution $(\Omega, \mathcal{F}, (\mathcal{F}_t), \mathbb{P},X,W)$  of Eq.\eqref{eq:mfgstate_fb}, $X_t$ has law $\mu_t$ with density $p(t,\cdot)$ for every $t\in[0,T]$.
\end{itemize}
\end{definition}
\noindent Assume that the MFG system in Eq.~\eqref{eq:pdesystem} has a unique weak solution $(u, p)$ and let $K$ be any constant such that
\begin{equation*}
\textcolor{black}{K \geq K(T, b, f, p_0, g),}
\end{equation*}
where $K(T, b, f, p_0, g)$ is the constant in Eq.~\eqref{eq:basic_const}. 
From an operative point of view, in order to find a (feedback) MFG solution in the sense of Definition \ref{def:mfgsolutionstocFB}, we look for an optimal control $\alpha^*\in \mathcal{A}^{fb}_K$ such that, given $p\in C_b([0,T]\times \mathbb{R}^d)$ and given any weak solution $(\Omega, \mathcal{F}, (\mathcal{F}_t), \mathbb{P},X^*,W)$  of Eq.\eqref{eq:mfgstate_fb} (controlled by $\alpha^*$ and with density $p$ appearing in the drift), the law of $X^*_t$ has density $p^*\in C_b([0,T]\times \mathbb{R}^d)$ such that $p^*\equiv p$.\\
}}

\noindent \textcolor{black}{Given the \textbf{environment} $(\Omega, \mathcal{F}, (\mathcal{F}_t), \mathbb{P},W, p)$,~i.e. a filtered probability space with Wiener process $W$ and with a given distribution of players specified by its density function $p$, where $p$ is as in Definition \ref{def:mfgsystemsolution}, we notice that} path-wise uniqueness and existence of a strong solution of Eq.~\eqref{eq:mfgstate_fb} is provided by \cite{veretennikov1981strong}. \textcolor{black}{Then, we define the unique solution $X$ of Eq.\eqref{eq:mfgstate_fb} in the given environment $(\Omega, \mathcal{F}, (\mathcal{F}_t), \mathbb{P},W, p)$ and with $\alpha \doteq -\nabla u$, to be the \textbf{state} of the PDE system in Eq.~\eqref{eq:pdesystem} in the given environment with density $p$. Nevertheless, we decide to introduce and work with weak solutions in view of the approximation result of Section \ref{sec:nash}, where we exploit weak convergence of the laws of the $N$-player system and provide a stochastic representation of the limiting dynamics by means of the martingale problem of Stroock and Varadhan \citep{stroock2007multidimensional}}. 

\subsection{Open-loop MFG with given density}\label{subsec:openloop}
\textcolor{black}{We now introduce a more general notion of control, that of open-loop control, together with what we intend with a solution of the MFG in open-loop form.}\\
\textcolor{black}{Let $K>0$. In order to make precise our definition of (open-loop) MFG solution, we introduce the following notation:}
\begin{itemize}
\item[\textit{(i)}] \textcolor{black}{We denote by $\mathcal{A}_{K}$ the set of admissible open-loop controls for the MFG, which is defined as the set of tuples $(\Omega, \mathcal{F}, (\mathcal{F}_t), \mathbb{P}, X, W, \alpha)$ where $\alpha = (\alpha(t))_{t \in [0,T]}$ is $\mathcal{F}_t$-progressively measurable, continuous and bounded by $K$ a.s. for all $t\in[0,T]$, while $(\Omega, \mathcal{F}, (\mathcal{F}_t), \mathbb{P}, X, W)$ is a weak solution of}
\begin{equation}\label{eq:mfgstate}
\textcolor{black}{X_t = \textcolor{black}{X_0} + \int_{0}^{t} (\alpha(s) + b(X_s, p(s, X_s)))\,ds + W_t,\quad t \in [0, T]}
\end{equation}
\textcolor{black}{where $X_0 \overset{d}{\sim} \mu_0$, having density $p_0$, is independent of the $\mathcal{F}_t$-Wiener process $W$.}
\textcolor{black}{For the sake of brevity and where no confusion is possible we will denote a control for the MFG simply with $\alpha$, in place of the full tuple.}
\item[\textit{(ii)}] \textcolor{black}{We consider the following cost functional}
\begin{equation}\label{eq:admissible_MFG}
\textcolor{black}{J(\alpha) \doteq \mathbb{E}\left[\int_{0}^{T} \frac{1}{2}|\alpha(s)|^2 + f(X_s, p(s, X_s))\,ds + g(X_T)\right]}
\end{equation}
\textcolor{black}{and we say that $\alpha^{*} \doteq (\alpha^{*}(t))_{t \in [0, T]} \in \mathcal{A}_{K}$ is an optimal control if it is a minimizer of $J$ over $\mathcal{A}_{K}$, i.e. if $J(\alpha^{*}) = \inf_{\alpha \in \mathcal{A}_K} J(\alpha)$. }
\end{itemize}
\textcolor{black}{Thereafter, we will denote by $\mathbf{OC}$ the just-introduced optimal control problem. The notion of solution we will consider in the open-loop case is then the following:}
{\color{black}{
\begin{definition}[MFG solution, stochastic \textcolor{black}{\textbf{open-loop}} formulation]\label{def:mfgsolutionstoc}
Let $T>0$ be the finite time horizon and $b, f, p_0, g$ as in (H1)-(H2) and (H4); see Section \ref{sec:preliminaries}. 
Then a \textbf{open-loop MFG solution} for bound $K>0$ is a pair $(\alpha^*,p)$ such that:
\begin{itemize}
\item[\textit{(i)}] $p\in C_b([0,T]\times \mathbb{R}^d)$ and $\alpha^* \in \mathcal{A}_K$, $\alpha^*$ standing for the full tuple:
\[(\Omega, \mathcal{F}, (\mathcal{F}_t), \mathbb{P}, X, W, \alpha^*);\]
\item[\textit{(ii)}] Given $p\in C_b([0,T]\times \mathbb{R}^d)$, $\alpha^* \in \mathcal{A}_K$ is an optimal control for problem $\mathbf{OC}$ (in the sense of item (ii) above);
\item[\textit{(iii)}] $(\Omega, \mathcal{F}, (\mathcal{F}_t), \mathbb{P},X,W)$ is a weak solution of Eq.\eqref{eq:mfgstate} such that $X_t$ has law $\mu_t$ with density $p(t,\cdot)$ for every $t\in[0,T]$.
\end{itemize}
\end{definition}
}
}

\noindent \textcolor{black}{As for the feedback case, given the \textbf{environment} $(\Omega, \mathcal{F}, (\mathcal{F}_t), \mathbb{P},W, p)$ where $p$ is as in Definition \ref{def:mfgsystemsolution}, given an admissible control $\alpha \in \mathcal{A}_K$, we notice that path-wise uniqueness and existence of a strong solution of Eq.~\eqref{eq:mfgstate} is provided by \cite{veretennikov1981strong} but we will continue working with weak solutions in view of the approximation result of Section \ref{sec:nash}}.\\
\indent\textcolor{black}{We point out that feedback controls induce stochastic open-loop controls so, as a consequence, the computation of the infimum of $J(\alpha)$ over the class of stochastic open-loop controls would, in principle, lead to a lower value with respect to performing the same computation over the set of stochastic feedback controls. However, thanks to Proposition 2.6 in \cite{karoui}, the two minimization problems are equivalent from the point of view of the value function.}\\

\noindent We state now the main result of this section, \textit{the Verification Theorem}, which gives an optimal control for $\mathbf{OC}$. \textcolor{black}{In particular, we are going to show that   $\alpha^{\ast}$ is the optimal feedback \textcolor{black}{control}, namely the optimal \textcolor{black}{strategy} to play at time $t$ \textcolor{black}{for a given state $x$}.}

\begin{theorem}[Verification Theorem]\label{th:VerificationTheorem} Consider the PDE system in Eq.~\eqref{eq:pdesystem} and let $(u, p)$ be a weak (or mild) solution. Consider the optimal control problem $\mathbf{OC}$ as in Definition \ref{def:mfgsolutionstoc}-(iii) and set $\alpha^{*}(t) = \alpha^{*}(t, x) \doteq -\nabla u(t, x)$. Then,
\begin{itemize}
\item[\textit{(i)}] $\alpha^{*}$ is an optimal control for $\mathbf{OC}$;
\item[\textit{(ii)}] \textcolor{black}{for any weak solution $(\Omega, \mathcal{F}, (\mathcal{F}_t), \mathbb{P},X^*,W)$ of Eq.~\eqref{eq:mfgstate} with $\alpha(s) = \alpha^{*}(s, X_s^{*})$, the state $X^{*}_t$ has law $\mu^*_t$ with density $p(t,\,\cdot\,)$ for every $t \in [0, T]$.}
\end{itemize}
\end{theorem}
\begin{proof}
Let $\alpha \in \mathcal{A}_K$ and $X^{\alpha} \doteq (X_t^{\alpha})_{t \in [0,T]}$ the solution of Eq.~\eqref{eq:mfgstate} controlled by $\alpha$. Besides, let $X_t^{*}$ as in Definition \ref{eq:admissible_MFG}-\textit{(ii)}, i.e., 
$$
X_t^{*} = \textcolor{black}{X_0} + \int_{0}^{t} (-\nabla u(s, X_s^{*}) + b(X_s^{*}, p(s, X_s^{*})))\,ds + W_t.
$$
\textcolor{black}{Notice that, thanks to boundedness of the drift, the previous equation admits both a weak solution and, in any given environment $(\Omega, \mathcal{F}, (\mathcal{F}_t), \mathbb{P},W, p)$, }a strong solution \textcolor{black}{that} is path-wise unique \citep{veretennikov1981strong}.\\
\indent \textit{Proof of (i)}.\quad Heuristically, should the function $u \in \text{C}^{1,2}([0,T] \times \mathbb{R}^{d})$, then we could apply It\^{o} formula to $u(t, X_t^{\alpha})$ and obtain (in expectation)
\begin{equation}\label{eq:euristic}
\begin{split}
\mathbb{E}[g(  & X_T^{\alpha})] = \mathbb{E}[u(T, X_T^{\alpha})]\\
					   &=\mathbb{E}\left[u(0, X_0^{\alpha}) + \int_{0}^{T}\left( \alpha(s) \cdot \nabla u(s, X_s^{\alpha}) - \frac{1}{2}| \nabla u(s, X_s^{\alpha})|^2 - f(X_s^{\alpha}, p(s, X_s^{\alpha})\right)\,ds\right],\\
\end{split}
\end{equation}
where we use the fact that the function $u$ satisfies the first equation of the PDE system in Eq.~\eqref{eq:pdesystem}, which implies
\begin{equation*}
\begin{split}
\mathbb{E}[g(  & X_T^{\alpha})] \geq \mathbb{E}\left[u(0, X_0^{\alpha}) + \int_{0}^{T}\left(- \frac{1}{2}| \alpha(s) |^2 - f(X_s^{\alpha}, p(s, X_s^{\alpha})\right)\,ds\right].
\end{split}
\end{equation*}
Hence for any admissible control $\alpha$ we would have $J(\alpha) \geq \mathbb{E}[u(0, X_0^{\alpha})]$. In particular, the above inequality becomes an equality for $\alpha(s) = \alpha^{*}(s,x) = -\nabla u(s,x)$, i.e. $J(\alpha^{*}) = \inf_{\alpha} J(\alpha) = \mathbb{E}[u(0, X_0^{*})]$. This would prove that $\alpha^{*}$ is an optimal control for $\mathbf{OC}$.\\
\noindent However, the function $u$ is not  ``regular enough" to apply  It\^{o} formula and some work is needed to adapt the heuristic argument to $u$. Given the technicality of this part and being it based on standard mollification arguments, we decide to move the required computations in Appendix \ref{app:hjandfpmild}, Subsection \ref{subsec:proofVerificationTheorem}.\\
\indent\textit{Proof of (ii).}\quad Now, let $\mu^{*}_t$ be the law of $X_t^{*}$ and let $\varphi \in \text{C}_b^{2}(\mathbb{R}^{d})$ be a test function. By It\^{o} formula, 
\begin{equation*}
\begin{split}
\varphi(X_t^{*}) =  \varphi(\textcolor{black}{X_0}) &+ \int_{0}^{t} \nabla \varphi(X_s^{*}) \cdot (-\nabla u(s, X_s^{*}) + b(X_s^{*}, p(s, X_s^{*})))\,ds\\
				   &+ \int_{0}^{t} \nabla \varphi(X_s^{*})\,dW_s + \frac{1}{2} \int_{0}^{t} \Delta  \varphi(X_s^{*})\,ds.
\end{split}
\end{equation*}
Hence, taking expectations on both sides, we have 
\begin{equation*}
\begin{split}
\left\langle \mu_t^{*}, \varphi(\,\cdot\,) \right\rangle & = \left\langle p_0 ,\varphi(\,\cdot\,)\right\rangle + \int_{0}^{t}  \left\langle \mu_s^{*}, \nabla \varphi(\,\cdot\,)  \cdot (-\nabla u(s, \,\cdot\,) + b(\,\cdot\,,p(s,\,\cdot\,)))\right\rangle \,ds \\
&+ \frac{1}{2} \int_{0}^{t} \left\langle \mu_s^{*}, \Delta \varphi(\,\cdot\,)\right\rangle \,ds.
\end{split}
\end{equation*}
Theorem \ref{th:localExistenceUniqueness} guarantees that this equation has a unique weak (or mild) solution $\mu_t$ with density $p(t,\,\cdot\,)$; hence $\mu$ and $\mu^*$ coincide and $\mu_t^{*}$ has density $p(t,\,\cdot\,)$ for every $t \in [0, T]$. This concludes the proof.
\end{proof}

\section{Moderately interacting particles}\label{sec:moderate} 
\noindent Let $N \in \mathbb{N}$ be the number of players and denote by $X_t^{N,i}$ the private state of player $i$ at time $t$, $t \in \left[0, T\right]$. In this section, we assume that the evolution of the players' states is given by Eq.~\eqref{eq:privatestate_evolution_N_player} and, as said, we consider players using feedback strategies, i.e. $\alpha^{N,i}(s) = \alpha(s, \textcolor{black}{\mathbf{X}}_s^{N})$ with $\alpha$ sufficiently smooth. \textcolor{black}{In particular, we will assume -- with the natural identification --  that \textcolor{black}{$\alpha \in \text{C}_b([0,T] \times \mathbb{R}^{d\cdot N} ; \mathbb{R}^{d})$}.} Besides, $b$, $V^{N}$ and $\textcolor{black}{X_0}^{N, i}$, $i \in [[ N ]]$, satisfy the hypotheses $\text{(H1)}$, $\text{(H3)}$ and $\text{(H4)}$ in Section \ref{sec:preliminaries}. Before proceeding, notice that the function
\begin{equation*}
F : [0,T] \times \mathbb{R}^{d\cdot N} \rightarrow \mathbb{R}^{d\cdot N}
\end{equation*}
defined component-wise as  
\begin{equation}\label{eq:function_F}
F_i(t, x_1, \ldots, x_N) \doteq \alpha(t, x_i) + b\Bigg(x_i, \frac{1}{N}\sum_{j = 1}^{N} V^{N}(x_i - x_j)\Bigg)
\end{equation}
is continuous and bounded. Since the blackian motion $\bm{W}^{N}_t \in \mathbb{R}^{d\cdot N}$ in Eq.~\eqref{eq:privatestate_evolution_N_player} is non-degenerate, \textcolor{black}{both} existence of \textcolor{black}{a weak solution and existence of} a pathwise unique strong solution \textcolor{black}{in any given environment $((\Omega_{N}, \mathcal{F}_{N}, (\mathcal{F}_t^{N}), \mathbb{P}^N), \bm{W}^{N}, V)$, where now in the $N$-player case the interaction among players is prescribed by $V$,} holds for this system \citep{veretennikov1981strong}. Let $S^{N}_t$ be the empirical measure on $\mathbb{R}^{d}$ of the players' private states, that is, 
\begin{equation}\label{eq:empirical_measure}
S_t^{N}(B) \doteq \frac{1}{N} \sum_{i = 1}^{N} \delta_{X_t^{N,i}}(B),\quad B \in \mathcal{B}(\mathbb{R}^{d}),\,\,t\in [0,T].
\end{equation}
\textcolor{black}{$S^N=(S_t^{N})$ is a continuous stochastic process with values in $\mathcal{P}(\mathbb{R}^{d})$; hence it can be seen as a random variable
with values in $\text{C}([0,T];\mathcal{P}(\mathbb{R}^{d}))$ (notice that for the sake of notation we do not put the explicit dependence on $\omega \in \Omega$ in these definitions). Therefore, $\mathcal{L}(S_{t}^{N})\in \mathcal{P}(\mathcal{P}(\mathbb{R}^{d}))$ and $\mathcal{L}(S^{N})\in \mathcal{P}(\text{C}([0,T];\mathcal{P}( \mathbb{R}^{d})))$, respectively.}\\\\
\noindent The main goal of this section is the characterization of the convergence of the laws $(\mathcal{L}(S^{N}))_{N\in \mathbb{N}}$ in $\mathcal{P}(\text{C}([0,T];\mathcal{P}(\mathbb{R}^{d})))$. This characterization result is the content of Theorem \ref{th:oelschlager85New} here below.
\begin{theorem}[Moderately interacting particles]\label{th:oelschlager85New}
\citep[cfr.][Theorem 1]{oelschlager1985law} Grant $\text{(H1)}$ and $\text{(H3)}-\text{(H4)}$. Let 
\textcolor{black}{$\alpha \in \text{C}_b([0,T] \times \mathbb{R}^{d \times N} ; \mathbb{R}^{d})$} be given. Then,
\begin{itemize}
\item[\textit{(i)}] the sequence of laws $(\mathcal{L}(S^{N}))_{N\in \mathbb{N}}$ converges
weakly in $\mathcal{P}(C([0,T];\mathcal{P}(\mathbb{R}^{d})))$ to $\delta
_{\mu }\in \mathcal{P}(C([0,T];\mathcal{P}(\mathbb{R}^{d})))$ for a
\textcolor{black}{flow of probability measures} $\mu \in C([0,T];\mathcal{P}(\mathbb{R}^{d}))$; hence
also $S^{N}$ converges in probability to $\mu $;
\item[\textit{(ii)}] for each $t\in \lbrack 0,T]$, $\mu _{t}$ is absolutely continuous with
respect to the Lebesgue measure on $\mathbb{R}^{d}$, with density $p(t,\,\cdot
\,)$; the \textcolor{black}{flow of density functions} satisfies%
\begin{equation*}
p \in \textcolor{black}{\text{C}_b([0,T] \times \mathbb{R}^{d})}
\end{equation*}%
and it is the unique solution in this space of the equation 
\begin{equation}\label{eq:deterministic}
p\left( t\right) =\mathcal{P}_{t}p_0 +\int_{0}^{t}\nabla \mathcal{P}%
_{t-s}\left( p\left( s\right) \left( \alpha \left( s\right) +{b(\,\cdot\,,p(s))}%
\right) \right) ds.
\end{equation}
\end{itemize}
\end{theorem}
\noindent The proof of the previous theorem is divided into four parts. The
first one is the tightness of the sequence of laws $(\mathcal{L}(S^{N}%
))_{N\in\mathbb{N}}$ in $\mathcal{P}(\text{C}([0,T];\mathcal{P}(\mathbb{R}%
^{d}))$; see Subsection \ref{subsec:tightness}. The second one is the
collection of estimates on $V^{N}\ast S_{t}^{N}$; see Subsection
\ref{subsec:empirical}. The third one is the characterization of the limits:
all the possible limits are a random solutions of the deterministic equation in
Eq.~\eqref{eq:deterministic}, with the required regularity; see Subsection
\ref{subsec:proof_oelschlager}. The fourth one is the proof of the uniqueness
of solutions of this deterministic equation.

\subsection{Tightness of the empirical measure}
\label{subsec:tightness} \noindent On $\mathcal{P}(\mathbb{R}^{d})$ the weak
topology is generated by the following complete metric:
\[
d_{w}(\mu,\nu)\doteq\sup_{f\in\text{Lip}_{1}(\mathbb{R}^{d})\cap
\text{C}_{b}(\mathbb{R}^{d})}\left(  \langle\mu,f\rangle-\langle\nu,f\rangle\right)
.
\]
We refer to \cite{oelschlager1985law}, Page 285, and
\cite{dudley1966convergence}, Theorem 18, for a complete proof of the previous result. Also, we consider the regularized empirical measures
\[
\left(  V^{N}\ast S_{t}^{N}\right)  (x)=\int_{\mathbb{R}^{d}}V^{N}%
(x-y)S_{t}^{N}(dy).
\]

\noindent In particular, these are probability densities, because they are non-negative functions with

\[
\int_{\mathbb{R}^{d}}\left(  V^{N}\ast S_{t}^{N}\right)  (x)dx=\int%
_{\mathbb{R}^{d}}\left(  \int_{\mathbb{R}^{d}}V^{N}(x-y)dx\right)  S_{t}%
^{N}(dy)=1.
\]

\noindent \textcolor{black}{Therefore, we consider the probability measure with
density $V^{N}\ast S_{t}^{N}$ as a random time-dependent element of
$\mathcal{P}(\mathbb{R}^{d})$} (for each $t$ and a.s.\,on the probability
space). In the next lemma, when we mention the laws $(\mathcal{L}(V^{N}\ast
S^{N}))_{N\in\mathbb{N}}$ on $\mathcal{P}(C([0,T];\mathcal{P}(\mathbb{R}%
^{d})))$, we adopt this interpretation.

\begin{lemma}
[Tightness]\label{lem:tightness}The laws $(\mathcal{L}(S^{N}))_{N\in
\mathbb{N}}$ are tight in $\mathcal{P}(\text{C}([0,T];\mathcal{P}(\mathbb{R}^{d})))$.
Similarly, the laws $(\mathcal{L}(V^{N}\ast S^{N}))_{N\in\mathbb{N}}$ are
tight in $\mathcal{P}(\text{C}([0,T];\mathcal{P}(\mathbb{R}^{d})))$.
\end{lemma}
\begin{proof}
\indent\textit{Part 1.} Recall that the initial conditions  $\textcolor{black}{X_0}^{N, i}$, $i \in [[N]]$, admit a density $p_0$ which is integrable. Therefore,
\[
\mathbb{E}\left[  \int_{\mathbb{R}^{d}} |x| S_{0}^{N}(dx)\right]  \leq C
\]
for some constant $C>0$, uniformly in $N \in\mathbb{N}$. To establish the tightness in  $\text{C}([0,T];\mathcal{P}(\mathbb{R}^{d}))$, we have to show \citep[see, for instance][Problem 2.4.11]{karatzas1998blackian} that the following two conditions are satisfied:
\begin{itemize}
\item[\textit{(i)}] $\mathbb{E}\left[  \sup_{t\in\lbrack0,T]}\int%
_{\mathbb{R}^{d}}|x|S_{t}^{N}(dx)\right]  \leq C$, $t\in\lbrack0,T]$,
\item[\textit{(ii)}] $\mathbb{E}\left[  d_{w}(S_{t}^{N},S_{s}^{N})^{p}\right]
\leq C|t-s|^{1+\epsilon}$, $t,s\in\lbrack0,T]$
\end{itemize}
for some constants $C>0$, $p\geq2$ and $\epsilon>0$. In order to verify
$\mathit{(i)}$, we compute
\[
\int_{\mathbb{R}^{d}}|x|S_{t}^{N}(dx)=\frac{1}{N}\sum_{i=1}^{N}|X_{t}^{N,i}|,
\]
\noindent where
\[
|X_{t}^{N,i}|\leq|\textcolor{black}{X_0}^{N,i}|+\int_{0}^{t}|\alpha\left(  s,X_{s}^{N,i}\right)
+b\Big(X_{s}^{N,i},\frac{1}{N}\sum_{j=1}^{N}V^{N}(X_{s}^{N,i}-X_{s}%
^{N,j})\Big)|ds+|W_{t}^{N,i}|.
\]
Hence,
\begin{equation*}
\Vert X^{N,i}\Vert_{\infty,t}\leq|\textcolor{black}{X_0}^{N,i}|+CT+\Vert W^{i}\Vert_{\infty
,t}\label{eq:estimate_particles},
\end{equation*}
which implies
\[
\mathbb{E}\left[  \Vert X^{N,i}\Vert_{\infty,T}\right]  \leq\mathbb{E}\left[
|\textcolor{black}{X_0}^{N,i}|\right]  +CT+C_{T}^{W}(d),
\]
\noindent where we use the boundedness (uniformly in $N$) of $\alpha$, $b$ and
$\mathbb{E}\left[  |\textcolor{black}{X_0}^{N,i}|\right]  $; the quantity $C_{T}^{W}(d)$ only
depends on $T$ and $d$. As regards $(ii)$, instead,
\begin{equation*}
\begin{split}
\mathbb{E}\left[  d_{w}(S_{t}^{N},S_{s}^{N})^{p}\right]   &  \leq
\mathbb{E}\left[  \sup_{f}\left\vert \frac{1}{N}\sum_{i=1}^{N}(f(X_{t}%
^{N,i})-f(X_{s}^{N,i}))\right\vert ^{p}\right]  \\
&  \leq\mathbb{E}\left[  \sup_{f}\frac{1}{N}\sum_{i=1}^{N}\left\vert
f(X_{t}^{N,i})-f(X_{s}^{N,i})\right\vert ^{p}\right]  \\
&  \leq\mathbb{E}\left[  \frac{1}{N}\sum_{i=1}^{N}\left\vert X_{t}^{N,i}%
-X_{s}^{N,i}\right\vert ^{p}\right]  \\
&  \leq C(|t-s|^{p}+|t-s|^{\frac{p}{2}}),\\
&
\end{split}
\end{equation*}
\noindent where we apply Jensen's inequality, the $1$-Lipschitz continuity of
$f$, boundedness of $\alpha$ and $b$ and Burkholder-Davis-Gundy inequality,
respectively. To conclude it suffices to choose $p>2$.\\
\indent\textit{Part 2.} To prove the statement for the random \textcolor{black}{flow of} probability measures $V^{N}\ast S_{t}^{N}$, let us first notice that, denoting by $R>0$ a real
number such that the support of $V$ is included in $B_{R}(0)$, the open ball of radius $R$ around the origin, for all $y\in\mathbb{R}^{d}$ we have
\[
\int_{\mathbb{R}^{d}}\left\vert x\right\vert V^{N}\left(  x-y\right)
dx=\int_{\mathbb{R}^{d}}\left\vert z+y\right\vert V^{N}\left(  z\right)
dz\leq\sup_{\left\vert w\right\vert \leq R}\left\vert w+y\right\vert
\leq\left\vert y\right\vert +R
\]
and thus
\begin{align*}
\int_{\mathbb{R}^{d}}\left\vert x\right\vert \left(  V^{N}\ast S_{t}%
^{N}\right)  \left(  x\right)  dx  & =\int_{\mathbb{R}^{d}}\left\vert
x\right\vert \left(  \int_{\mathbb{R}^{d}}V^{N}\left(  x-y\right)  S_{t}%
^{N}\left(  dy\right)  \right)  dx\\
& =\int_{\mathbb{R}^{d}}\left(  \int_{\mathbb{R}^{d}}\left\vert x\right\vert
V^{N}\left(  x-y\right)  dx\right)  S_{t}^{N}\left(  dy\right)  \\
& \leq\int_{\mathbb{R}^{d}}\left\vert y\right\vert S_{t}^{N}\left(  dy\right)
+R.
\end{align*}
We conclude by going back to the previous estimate without the mollifier. Moreover, denoted
$V^{N,-}\left(  x\right)  =V^{N}\left(  -x\right)  $, if $f$ has Lipschitz
constant less or equal to one, then
\begin{align*}
& \left\vert \left(  V^{N,-}\ast f\right)  \left(  x\right)  -\left(
V^{N,-}\ast f\right)  \left(  y\right)  \right\vert \\
& =\left\vert \int_{\mathbb{R}^{d}}V^{N}\left(  x^{\prime}-x\right)  f\left(
x^{\prime}\right)  dx^{\prime}-\int_{\mathbb{R}^{d}}V^{N}\left(  x^{\prime
}-y\right)  f\left(  x^{\prime}\right)  dx^{\prime}\right\vert \\
& =\left\vert \int_{\mathbb{R}^{d}}V^{N}\left(  z\right)  f\left(  z+x\right)
dz-\int_{\mathbb{R}^{d}}V^{N}\left(  z\right)  f\left(  z+y\right)
dz\right\vert \\
& \leq\int_{\mathbb{R}^{d}}V^{N}\left(  z\right)  \left\vert f\left(
z+x\right)  -f\left(  z+y\right)  \right\vert dz\leq\left\vert x-y\right\vert
\int_{\mathbb{R}^{d}}V^{N}\left(  z\right)  dz=\left\vert x-y\right\vert
\end{align*}
namely $V^{N,-}\ast f$ has also Lipschitz constant less or equal to one.
Therefore,
\begin{align*}
& \left\vert \left\langle V^{N}\ast S_{t}^{N},f\right\rangle -\left\langle
V^{N}\ast S_{s}^{N},f\right\rangle \right\vert \\
& =\left\vert \left\langle S_{t}^{N},V^{N,-}\ast f\right\rangle -\left\langle
S_{s}^{N},V^{N,-}\ast f\right\rangle \right\vert \\
& \leq\frac{1}{N}\sum_{i=1}^{N}\left\vert \left(  V^{N,-}\ast f\right)
(  X_{t}^{N,i} )  -\left(  V^{N,-}\ast f\right)  \left(  X_{s}%
^{N,i}\right)  \right\vert \\
& \leq\frac{1}{N}\sum_{i=1}^{N}\left\vert X_{t}^{N,i}-X_{s}^{N,i}\right\vert,
\end{align*}
and we are again led back to the previous estimate without the mollifier. 
\end{proof}
\subsection{Estimates on mollified empirical measures}\label{subsec:empirical} 
\noindent In this subsection we obtain estimates on mollified empirical measures. More precisely, we first prove that the empirical measure $S_t^{N}$ satisfies the following identity for a test function $\varphi \in \text{C}_b^{1,2}([0, T] \times \mathbb{R}^{d})$:
\begin{equation*}
\begin{split}
\left\langle S_{t}^{N}, \varphi \left(t,\,\cdot\,\right) \right\rangle & = \left\langle
S_{0}^{N}, \varphi\left(0,\,\cdot\,\right) \right\rangle \\
&+\int_{0}^{t}\left(
\left\langle S_{s}^{N}, \frac{\partial \varphi}{\partial s} \left(s,\,\cdot\,\right) +\frac{1}{2}\Delta
\varphi \left(s,\,\cdot\,\right) \right\rangle +\left\langle S_{s}^{N},\alpha \left(
s \right) \cdot \nabla \varphi \left(s, \cdot \right) \right\rangle \right) ds\\
&+\int_{0}^{t}\left\langle S_{s}^{N},b\left(\,\cdot\,,\left( V^{N}\ast
S_{s}^{N}\right) \left(\,\cdot\,\right) \right) \cdot \nabla \varphi \left(s,\,\cdot\,\right)
\right\rangle ds+M_{t}^{N, \varphi},
\end{split}
\end{equation*}
where $M_{t}^{N, \varphi}$ is a martingale to be defined below. Then, in Lemma \ref{lem:mild_form_empirical_density} we obtain an identity in mild form for the \textit{empirical density}; the latter is defined as any
convolution of the empirical measure with a smooth mollifier. In our paper, we work with the following particular convolution:
\begin{equation}\label{eq:empirical_density}
p^{N}(t,x) \doteq \left( V^{N}\ast S_{t}^{N}\right) (x) = \int_{\mathbb{R}
^{d}}V^{N}(x-y)S_{t}^{N}(dy) = \frac{1}{N}\sum_{i=1}^{N}V^{N}(x-X_{t}^{N,i}), 
\end{equation}
\noindent where $t \in [0, T]$ and $x \in \mathbb{R}^{d}$. Then, in Lemma \ref{lem:estimate_martingale} we derive H\"{o}lder-type semi-norm bound for the martingale $M_t^{N,\varphi}$, and in Lemma  \ref{lem:hoelder_estim_empirical_density}, instead, H\"{o}lder-type semi-norm bound for the empirical density \eqref{eq:empirical_density}. In particular, we will see that in order to understand the limit of $(\mathcal{L}(S^{N}))_{N \in \mathbb{N}}$ it is crucial to study rigorously the regularity properties of $p^{N}$ that remain stable in the limit as $N$ tends to infinity.\\
\indent First, we obtain the identity for the empirical measure. Let $\varphi \in \text{C}_b^{1,2}([0, T] \times \mathbb{R}^{d})$ be a test function. By It\^{o} formula,
\begin{equation*}
\begin{split}
d\left\langle S_{t}^{N},\varphi \left( t,\,\cdot\,\right) \right\rangle & =\frac{1}{N}%
\sum_{i=1}^{N}d\varphi  \left( t,X_{t}^{N,i}\right) \\
&=\frac{1}{N}\sum_{i=1}^{N}\frac{\partial \varphi}{\partial t} \Big( t,X_{t}^{N,i}  \Big) dt  +\frac{1}{N}\sum_{i=1}^{N}\nabla \varphi \left( t,X_{t}^{N,i}\right) \cdot
\alpha(t,X_{t}^{N,i}) dt \\
&+\frac{1}{N}\sum_{i=1}^{N}\nabla \varphi \left(t,  X_{t}^{N,i}\right) \cdot
b  \Big(X_{t}^{N,i},\frac{1}{N}\sum_{j=1}^{N}V^{N}(X_{t}^{N,i}-X_{t}^{N,j}) \Big)dt \\
&+\frac{1}{N}\sum_{i=1}^{N}\nabla \varphi\Big(t, X_{t}^{N,i} \Big) \cdot
dW_{t}^{N, i}+\frac{1}{2N}\sum_{i=1}^{N} \Delta \varphi\Big(t, X_{t}^{N,i} \Big) dt\\
&= \left\langle S_t^{N}, \frac{\partial \varphi}{\partial t}(t,\,\cdot\,) \right\rangle + \left\langle S_{t}^{N},\alpha ( t ) \cdot \nabla \varphi(t,\,\cdot\,)
\right\rangle dt \\
&+\left\langle S_{t}^{N},b(\,\cdot\,, \left( V^{N}\ast S_{t}^{N}\right) \left(
\,\cdot\,\right)) \cdot \nabla \varphi(t,\,\cdot\,) \right\rangle dt \\
&+\frac{1}{N}\sum_{i=1}^{N}\nabla \varphi \left(t, X_{t}^{N,i}\right) \cdot
dW_{t}^{N, i}+\frac{1}{2}\left\langle S_{t}^{N},\Delta \varphi(t,\,\cdot\,)\right\rangle dt.
\end{split}
\end{equation*}
In particular, the previous expression can be rewritten in integral form as:
\begin{equation}\label{eq:identity_empirical_meas}
\begin{split}
\left\langle S_{t}^{N}, \varphi \left(t,\,\cdot\,\right) \right\rangle & = \left\langle
S_{0}^{N}, \varphi\left(0,\,\cdot\,\right) \right\rangle \\
&+\int_{0}^{t}\left(
\left\langle S_{s}^{N}, \frac{\partial \varphi}{\partial s} \left(s,\,\cdot\,\right) +\frac{1}{2}\Delta
\varphi \left(s,\,\cdot\,\right) \right\rangle +\left\langle S_{s}^{N},\alpha \left(
s \right) \cdot \nabla \varphi \left(s, \cdot \right) \right\rangle \right) ds\\
&+\int_{0}^{t}\left\langle S_{s}^{N},b\left(\,\cdot\,,\left( V^{N}\ast
S_{s}^{N}\right) \left(\,\cdot\,\right) \right) \cdot \nabla \varphi \left(s,\,\cdot\,\right)
\right\rangle ds+M_{t}^{\varphi ,N},
\end{split}
\end{equation}
where $M_{t}^{N, \varphi}$ is the martingale
\begin{equation}\label{eq:martingale_generic_test_function}
M_{t}^{N, \varphi}=\int_{0}^{t}\frac{1}{N}\sum_{i=1}^{N}\nabla \varphi \left(
s,X_{s}^{N,i}\right) \cdot dW_{s}^{N, i}.
\end{equation}
\indent Second, we obtain the identity in mild form for the empirical density. \textcolor{black}{Henceforth, we will use the classical notational conventions used in the semigroups theory \citep[see][]{pazy2012semigroups}. Sometimes, it may happen that we will indicate the explicit dependence on the state variable to clarify the results; see e.g. the second integral in the lemma here below.}

\begin{lemma}\label{lem:mild_form_empirical_density}
\textcolor{black}{Let $p^{N}$ as in Eq.~\eqref{eq:empirical_density}}. Grant assumptions of Theorem \ref{th:oelschlager85New}. Then,
\begin{equation*}
\begin{split}
p^{N}(t) =\mathcal{P}_{t}p^{N}(0) & +\int_{0}^{t}\nabla \mathcal{P}%
_{t-s}\left( V^{N}\ast \left( \alpha(s)\, S_{s}^{N}\right) \right) ds \\ 
															&+ \int_{0}^{t}\nabla \mathcal{P}_{t-s}\left( V^{N}\ast \left( b\left(\,\cdot\,, p^{N}(s,\,\cdot\,)\right) S_{s}^{N}\right) \right) ds+M_{t}^{N}(\,\cdot\,)
\end{split}
\end{equation*}
where
\begin{equation}\label{eq:martingale_to_be_used}
M_{t}^{N}(\,\cdot\,) = \int_{0}^{t}\frac{1}{N}\sum_{i=1}^{N}\mathcal{P}_{t-s}\nabla
V^{N}\left(\,\cdot\, -X_{s}^{N,i}\right) dW_{s}^{N, i}.
\end{equation}
\end{lemma}
\begin{proof}
\noindent\textcolor{black}{For the reader convenience, let us first recall the definition of $\mathcal{P}_{t}$; cfr. Eq.~\eqref{eq:heat_kernel}. If we set $G(t, x-y)$ the density of $x + W_t$, where $W_t$ is a standard blackian motion, $t \in [0, T]$ and $x, y \in \mathbb{R}^{d}$, then $\mathcal{P}_t$ is defined on functions $h \in \text{C}_b(\mathbb{R}^{d})$ as}
\begin{equation*}
(\mathcal{P}_t h)(x)\doteq \int_{\mathbb{R}^{d}} G(t, x-y) h(y)\,dy,
\end{equation*}
\noindent Now, consider for a given $t \in [0, T]$ the identity in Eq.~\eqref{eq:identity_empirical_meas} with the following choice
\begin{equation*}
\varphi^{(t)} \left(s, x\right) =\left(\mathcal{P}_{t-s}(V^{N,-}\ast h)\right) \left( x\right),\quad s \in [0, t],
\end{equation*}
with $h\in \text{C}_{b}^{2}(\mathbb{R}^{d})$ and $V^{N,-}\left( x\right)
\doteq V^{N}\left( -x\right)$. Recall that the convolution commutes and hence $\mathcal{P}_t (V^{N,-}\ast h) = (V^{N,-}\ast \mathcal{P}_{t}h)$. Besides, it holds that $\nabla \mathcal{P}_{t}(V^{N,-}\ast h) = (V^{N,-} \ast \nabla \mathcal{P}_{t}h)$. Therefore, 
\begin{equation*}
\begin{split}
\left\langle V^{N}\ast S_{t}^{N},h\right\rangle & =\left\langle V^{N}\ast
S_{0}^{N},\mathcal{P}_{t}h\right\rangle -\int_{0}^{t}\left\langle V^{N}\ast
\left( \alpha \left( s \right) S_{s}^{N}\right) ,\nabla \mathcal{P}_{t-s}h\right\rangle ds \\
&+\int_{0}^{t}\left\langle V^{N}\ast \left( b\left(\,\cdot\,,\left( V^{N}\ast
S_{s}^{N}\right) \left(\,\cdot\,\right) \right) S_{s}^{N}\right) ,\nabla \mathcal{P}_{t-s}h\right\rangle ds \\
&+\int_{0}^{t}\frac{1}{N}\sum_{i=1}^{N}V^{N,-}\ast \nabla \left( \mathcal{P}
_{t-s}h\right) \left( X_{s}^{N,i}\right) \cdot dW_{s}^{N, i}.
\end{split}
\end{equation*}

\noindent By Fubini-Tonelli theorem and stochastic Fubini theorem, we can move the semigroup on the first argument and use integration by parts to obtain:
\begin{equation*}
\begin{split}
\left\langle p^{N}\left( t\right) ,h\right\rangle &=\left\langle \mathcal{P}_{t}p^{N}(0),h\right\rangle -\int_{0}^{t}\left\langle \nabla \mathcal{P}_{t-s}\left( V^{N}\ast \left( \alpha \left( s \right)  S_{s}^{N}\right)
\right) ,h\right\rangle ds  \\
&+\int_{0}^{t}\left\langle \nabla \mathcal{P}_{t-s}\left( V^{N}\ast \left(
b\left(\,\cdot\,,\left( V^{N}\ast S_{s}^{N}\right) \left(\,\cdot\,\right) \right)
S_{s}^{N}\right) \right) ,h\right\rangle ds \\
& +\left\langle M_{t}^{N}(\,\cdot\,),h\right\rangle.
\end{split}
\end{equation*}
By the arbitrarily of $h$, this concludes the proof. 
\end{proof}

\noindent Now, let denote by $\left[ f\right] _{\gamma }$ the H\"{o}lder semi-norm on $\mathbb{R}^{d}$ and by $\left\Vert f\right\Vert _{\gamma }$ the associated norm, i.e.:
\begin{equation}\label{eq:holder_seminorm}
\left[ f\right] _{\gamma }=\sup_{\substack{ x,y\in \mathbb{R}^{d}  \\ x\neq
y }}\frac{\left\vert f\left( x\right) -f\left( y\right) \right\vert }{%
\left\vert x-y\right\vert ^{\gamma }},\qquad \left\Vert f\right\Vert
_{\gamma }=\left[ f\right] _{\gamma }+\left\Vert f\right\Vert _{\infty }
\end{equation}%
where, as usual, $\left\Vert f\right\Vert _{\infty }=\sup_{x\in \mathbb{R}^{d}}\left\vert f\left( x\right) \right\vert $. We state the following lemma.

\begin{lemma}\label{lem:estimate_martingale}
Let $M_t^{N}(\,\cdot\,)$ be the martingale in Eq.~\eqref{eq:martingale_to_be_used} and $\beta \in (0, 1/2)$ \textcolor{black}{the constant as in the definition of $V^{N}$; see Eq.~\eqref{eq:Vn}}. Then, there exists $\gamma \in \left( 0,1\right) $ such that, for all $p\geq 2$, there is a constant $C_{p}>0$ such that $\mathbb{E}\left[ \left\Vert
M_{t}^{N}\right\Vert _{\gamma }^{p}\right] \leq C_{p}$, for all $N\in \mathbb{%
N}$ and $t\in \left[ 0,T\right] $.
\end{lemma}
\begin{proof}
It is enough to check the sufficient conditions \eqref{eq:cond_2}-\eqref{eq:cond_3} of Lemma \ref{lem:suff_hoelder} in Appendix  \ref{app:holder_heat_bound}.\\
\noindent \textcolor{black}{Let $\epsilon_N^{-1} = N^{\frac{\beta}{d}}$}. Using Eq.~\eqref{eq:bound1_heat_kernel}, the bound in Eq.~\eqref{eq:cond_2} reads
\begin{equation*}
\begin{split}
 \mathbb{E}\left[ \left\vert M_{t}^{N}\left( x\right) \right\vert ^{p}\right]
 &= \frac{1}{N^{p}}\mathbb{E}\left[ \left\vert
\sum_{i=1}^{N}\int_{0}^{t}\nabla \mathcal{P}_{t-s}V^{N}\left(
x-X_{s}^{N,i}\right) dW_{s}^{N, i}\right\vert ^{p}\right]\\
&\leq  \frac{C_{p}}{N^{p}}\mathbb{E}\left[ \left\vert
\sum_{i=1}^{N}\int_{0}^{t}\left\vert \nabla \mathcal{P}_{t-s}V^{N}\left(
x-X_{s}^{N,i}\right) \right\vert ^{2}ds\right\vert ^{p/2}\right]\\
&\leq \frac{C_{p}C_{T,R,V}^{p}\epsilon _{N}^{-pd-p\delta }}{N^{p}}\mathbb{E}%
\left[ \left\vert \sum_{i=1}^{N}\int_{0}^{t}\frac{1}{\left( t-s\right)
^{1-\delta }} e^{-\frac{\left\vert x-X_{s}^{N,i}\right\vert }{4T}}\,ds\right\vert ^{p/2}\right] \\
&\leq \frac{\widetilde{C}_{p,T,R,V}\epsilon _{N}^{-pd-p\delta }}{N^{p}}\,e^{-\frac{\left\vert x\right\vert}{8T}} \mathbb{E}\left[e^{\frac{|| X^{N, i} ||_{\infty, T}}{4 T}} \left\vert \sum_{i=1}^{N}\int_{0}^{t}\frac{1}{\left( t-s\right)^{1-\delta }}ds\right\vert ^{p/2}\right] \\
&\leq \frac{C_{p,T,R,V,\delta }^{\prime }\epsilon _{N}^{-pd-p\delta }}{
N^{p/2}}\,\textcolor{black}{e^{-\frac{|x|}{8 T}}}\mathbb{E}
\left[ e^{p\,\frac{|| X^{N, 1} ||_{\infty, T}}{8 T}} \right]
\end{split}
\end{equation*}
where, to ease notation, we set $|| X^{N, i} ||_{\infty, T}\doteq\sup_{s \in [0,T]}| X_s^{N,i} |,\,\,i\in[[N]]$. The last expected value is finite \textcolor{black}{thanks to $(\text{H4})$}; therefore,
\begin{equation*}
\mathbb{E}\left[ \left\vert M_{t}^{N}\left( x\right) \right\vert ^{p}\right]
\leq C_{p,T,R,V,\delta }^{\prime \prime }\frac{\epsilon _{N}^{-pd-p\delta }}{%
N^{p/2}}g^{p}\left( x\right),
\end{equation*}
where (up to a constant) $g\left( x\right) \doteq e^{-\frac{\left\vert
x\right\vert }{8T}}$ is integrable at any power. Now, recall that $\epsilon _{N}^{-1}=N^{\frac{\beta }{d}}$. Then
\begin{equation*}
\frac{\epsilon _{N}^{-pd-p\delta }}{N^{p/2}}=\frac{N^{\frac{\beta }{d}\left(
pd+p\delta \right) }}{N^{p/2}}=N^{\left( \frac{1}{2}-\beta \right) p-\frac{%
\beta p\delta }{d}},
\end{equation*}%
which is bounded for $\beta <\frac{1}{2}$ by choosing $\delta $
(depending on $p$) small enough.\\
\noindent As regards the bound in Eq.~ \eqref{eq:cond_2}, we use estimate \eqref{eq:bound2_heat_kernel} with $
\gamma $ small enough compared to $\delta $ so to have $( \gamma -\delta
\left( 1-\gamma \right)) <0$.  To ease notation and for the sake of space, we denote 
\begin{equation*}
\begin{split}
& \Delta_h M_t^{N}(x) \doteq M_{t}^{N}\left( x\right) -M_{t}^{N}\left(
x+h\right)\\
& \Delta_h \mathcal{P}_{t-s} V^{N} (x - X_s^{N,i}) \doteq V^{N} (x - X_s^{N,i}) - V^{N} (x + h - X_s^{N,i})
\end{split}
\end{equation*}
and, as before, $|| X^{N, i} ||_{\infty, T}\doteq\sup_{s \in [0,T]}| X_s^{N,i} |,\,\,i\in[[N]]$. We get
\begin{equation*}
\begin{split}
 \mathbb{E}\left[ \left\vert  \Delta_h M_t^{N}(x) \right\vert ^{p}\right] &= \frac{1}{N^{p}}\mathbb{E}\left[
\left\vert \sum_{i=1}^{N}\int_{0}^{t} \Delta_h \mathcal{P}_{t-s} V^{N} (x - X_s^{N,i})  dW_{s}^{N, i}\right\vert ^{p}\right] \\
& \leq \frac{C_{p}}{N^{p}}\mathbb{E}\left[ \left\vert
\sum_{i=1}^{N}\int_{0}^{t}\left\vert  \Delta_h \mathcal{P}_{t-s} V^{N} (x - X_s^{N,i})  dW_{s}^{N, i}  \right\vert ^{2}ds\right\vert ^{p/2}\right]\\
&\leq \frac{C_{p}}{N^{p}}\mathbb{E}\left[ \left\vert
\sum_{i=1}^{N}\int_{0}^{t}\frac{C_{T,R,V}^{2}}{\left( t-s\right) ^{1+\bar{\gamma}}}\left\vert h\right\vert ^{2\gamma
}\epsilon _{N}^{-2d-2\delta \left( 1-\gamma \right) } e^{-2\lambda
_{T,R,V}\left\vert x-X_{s}^{N,i}\right|}\, ds\right\vert ^{p/2}%
\right] \\
&\leq \frac{\widetilde{C}_{p,T,R,V}\epsilon _{N}^{-pd-p\delta \left(
1-\gamma \right) }}{N^{p/2}}\left\vert h\right\vert ^{p\gamma }e^{-2\lambda
_{T,R,V}\left\vert x \right|} \mathbb{E}\left[ e^{p\,\lambda _{T,R,V} \|X^{N,i}\|_{\infty, T}} \right]
\end{split}
\end{equation*}
and the conclusion is the same as for the previous term.
\end{proof}

\begin{remark}
Lemma \ref{lem:estimate_martingale} is a non-trivial achievement of this paper. Indeed, the Kolmogorov-Chentsov criterion \citep[see][Theorem 2.2.8]{karatzas1998blackian} would provide with much fewer computations a similar result on bounded sets. However, the dominating constant would diverge when passing to the full space. Notice that we will need the passage to the full space in Lemma \ref{lem:hoelder_estim_empirical_density} below. For this reason, we use a more complicated strategy --  summarized by the results in Appendix \ref{app:holder_heat_bound} -- based on Sobolev embedding theorem. 
\end{remark}

\begin{lemma}\label{lem:hoelder_estim_empirical_density} Let $p^{N}(t)$ as in Lemma \ref{lem:mild_form_empirical_density}. If $\beta \in \left( 0, 1/2 \right) $ and $\sup_{N}\left\Vert p^{N}(0)\right\Vert _{\gamma}^{2}<\infty $, then there exist $p\geq 2$, $\gamma \in \left( 0,1\right) $ and a constant $C_{\gamma }>0$ such that $\mathbb{E}\left[ \left\Vert
p^{N}(t)\right\Vert _{\gamma }^{p}\right] \leq C$.
\end{lemma}
\begin{proof}
Lemma  \ref{lem:mild_form_empirical_density} provides the following bound
\begin{equation*}
\begin{split}
\mathbb{E}\left[ \left\Vert p^{N}(t)\right\Vert _{\gamma }^{p}\right]
^{1/p} &\leq  \mathbb{E}\left[ \left\Vert \mathcal{P}_{t}p^{N}(0)\right
\Vert _{\gamma }^{p}\right] ^{1/p}+\int_{0}^{t}\mathbb{E}\left[ \left\Vert
\nabla \mathcal{P}_{t-s}\left( V^{N}\ast \left( \alpha(s)S_{s}^{N}\right) \right) \right\Vert _{\gamma }^{p}\right] ^{1/p}ds\\
&+\int_{0}^{t}\mathbb{E}\left[ \left\Vert \nabla \mathcal{P}_{t-s}\left(
V^{N}\ast \left( b\left(\,\cdot\,,p^{N}\left(s, \,\cdot\,\right) \right)
S_{s}^{N}\right) \right) \right\Vert _{\gamma }^{p}\right] ^{1/p}ds+\mathbb{E
}\left[ \left\Vert M_{t}^{N}\right\Vert _{\gamma }^{p}\right] ^{1/p},
\end{split}
\end{equation*}
where we use the first inequality of  Lemma \ref{lem:heat_semigroup_estim} in Appendix \ref{app:holder_heat_bound} and the bound of Lemma \ref{lem:estimate_martingale}. Therefore,
\begin{equation*}
\begin{split}
\mathbb{E}\left[ \left\Vert  p^{N}(t)\right\Vert _{\gamma }^{p}\right]
^{1/p} &\leq  C+\int_{0}^{t}\frac{C}{\left( t-s\right) ^{\frac{1+\gamma }{2}}%
}\mathbb{E}\left[ \left\Vert V^{N}\ast \left( \alpha( s,\,\cdot\, )S_{s}^{N}\right) \right\Vert _{\infty }^{p}\right] ^{1/p}ds\\
&+\int_{0}^{t}\frac{C}{\left( t-s\right) ^{\frac{1+\gamma }{2}}}\mathbb{E}%
\left[ \left\Vert V^{N}\ast \left( b\left(\,\cdot\,,p^{N}\left(s, \,\cdot\,\right)
\right) S_{s}^{N}\right) \right\Vert _{\infty }^{p}\right] ^{1/p}ds+C.
\end{split}
\end{equation*}
At this point, we need to find a bound for the last two expected values. We start from the first.
\begin{equation*}
\begin{split}
& \left\vert \left( V^{N}\ast \left( \alpha(s,\,\cdot\,)S_{s}^{N}\right) \right)
\left( x\right) \right\vert \leq \int_{\mathbb{R}^{d}}V^{N}\left(
x-y\right) \left\vert  \alpha(s, y) \right\vert
S_{s}^{N}\left( dy\right)\\
& \leq \left\Vert \alpha(s,\,\cdot\,) \right\Vert _{\infty }\int_{\mathbb{R}^{d}}V^{N}\left( x-y\right) S_{s}^{N}\left( dy\right) =\left\Vert  \alpha(s,\,\cdot\,) \right\Vert _{\infty } p^{N}(s,x)
\end{split}
\end{equation*}
hence%
\begin{equation*}
\resizebox{1.0 \textwidth}{!}{$
\mathbb{E}\left[ \left\Vert V^{N}\ast \left( \alpha(s,\,\cdot\,)S_{s}^{N}\right)
\right\Vert _{\infty }^{p}\right] ^{1/p}  \leq \left\Vert  \alpha(s,\,\cdot\,) \right\Vert _{\infty }\mathbb{E}\left[ \left\Vert  p^{N}(s)\right\Vert
_{\infty }^{p}\right] ^{1/p}  \leq \left\Vert  \alpha(s,\,\cdot\,) \right\Vert _{\infty }\mathbb{E}\left[
\left\Vert  p^{N}(s)\right\Vert _{\gamma }^{p}\right] ^{1/p}$}
\end{equation*}
As regards the second expected value, we similarly obtain 
\begin{equation*}
\mathbb{E}\left[ \left\Vert V^{N}\ast \left( b\left(\,\cdot\,, p^{N}(s, \,
\cdot\,) \right) S_{s}^{N}\right) \right\Vert _{\infty }^{p}\right]
^{1/p}\leq \left\Vert b\right\Vert _{\infty }\mathbb{E}\left[ \left\Vert
 p^{N}(s)\right\Vert _{\gamma }^{p}\right] ^{1/p}.
\end{equation*}
Therefore,
\begin{equation*}
\mathbb{E}\left[ \left\Vert  p^{N}(t)\right\Vert _{\gamma }^{p}\right]
^{1/p}\leq C+C\int_{0}^{t}\frac{\left\Vert \alpha(s,\,\cdot\,) \right\Vert _{\infty
}+\left\Vert b\left(\,\cdot\,, p^{N}(s, \,
\cdot\,) \right) \right\Vert _{\infty }}{\left( t-s\right) ^{\frac{1+\gamma }{2}%
}}\mathbb{E}\left[ \left\Vert  p^{N}(s)\right\Vert _{\gamma }^{p}\right]
^{1/p}ds.
\end{equation*}%
The conclusion follows by a generalized version of Gronwall's lemma.
\end{proof}
\noindent We are now ready to prove Theorem \ref{th:oelschlager85New}; its proof is the content of the next subsection.
\subsection{Identification of the limit}
\label{subsec:proof_oelschlager}
Let us denote by $P_{N}$ and $Q_{N}$ the laws of $S^{N}$ and $V^{N}\ast S^{N}$,
respectively, on $\text{C}([0,T];\mathcal{P}(\mathbb{R}^{d}))$, for each
$N\in\mathbb{N}$. By Lemma \ref{lem:tightness}, we know that both the families
$(P_{N})_{N\in\mathbb{N}}$ and $(Q_{N})_{N\in\mathbb{N}}$ are tight in
$\text{C}([0,T];\mathcal{P}(\mathbb{R}^{d}))$. In particular, their convergent sub-sequences have the same limit, in the following strong sense. 

\begin{lemma}\label{lem:lemma_previous}
Assume a subsequence $(P_{N_{k}})_{k\in\mathbb{N}}$ converges weakly to a
probability measure $P$ on $\text{C}([0,T];\mathcal{P}(\mathbb{R}^{d}))$. Then
also $(Q_{N_{k}})_{k\in\mathbb{N}}$ converges weakly to $P$.
\end{lemma}
\begin{proof}
To prove the lemma, we are going to show that every convergent subsequence of $(Q_{N_{k}})_{k\in\mathbb{N}}$ has limit $P$; indeed, this implies that $(Q_{N_{k}})_{k\in\mathbb{N}}$ converges to $P$. To this end, let $(Q_{N_{k}^{\prime}})_{k\in\mathbb{N}}$ be a subsequence of $(Q_{N_{k}})_{k\in\mathbb{N}}$ converging to a probability measure $Q$ on $\text{C}([0,T];\mathcal{P}(\mathbb{R}^{d}))$. In particular, for every positive integer $m$
and every finite sequence $t_{1}<...<t_{m}\in\left[0,T\right]$, both $\pi_{\left(t_{1},...,t_{m}\right)}P_{N_{k}^{\prime}}$ and $\pi_{\left(t_{1},...,t_{m}\right)}Q_{N_{k}^{\prime}}$ converge weakly on $\mathcal{P}(\mathbb{R}^{d})^{m}$, where $\pi_{\left(  t_{1},...,t_{m}\right)  }$ is the projection on the finite dimensional marginal at times $\left(  t_{1}%
,...,t_{m}\right)$. The limits are, respectively, $\pi_{\left(t_{1},...,t_{m}\right)  }P$ and $\pi_{\left(  t_{1},...,t_{m}\right)  }Q$. If we prove that they are equal, then $P=Q$ \textcolor{black}{as a consequence of Kolmogorov extension theorem \citep[see e.g.][Theorem 1.1.10]{stroock2007multidimensional}}.\\
\indent Now, by Skorokhod representation theorem, on a new probability space $\left(\widetilde{\Omega},\widetilde{\mathcal{F}},\widetilde{\mathbb{P}}\right)$ we may consider a sequence $\widetilde{S}_{t}^{N_{k}^{\prime}}$ of continuous
processes with values in $\mathcal{P}(\mathbb{R}^{d})$ and a continuous
process $\widetilde{\mu}_{t}$ with values in $\mathcal{P}(\mathbb{R}^{d})$
such that their laws on $\text{C}([0,T];\mathcal{P}(\mathbb{R}^{d}))$ are
$P_{N_{k}^{\prime}}$ and $P$ respectively; and $V^{N_{k}^{\prime}}\ast\widetilde{S}_{\cdot}^{N_{k}^{\prime}}$ has law $Q_{N_{k}^{\prime}}$, which we know to be convergent, weakly, to $Q$. As remarked at the beginning of Appendix \ref{app:holder_heat_bound_1}, given $t\in\left[  0,T\right] $, with probability one, $\left\langle V^{N_{k}^{\prime}}\ast\widetilde{S}%
_{t}^{N_{k}^{\prime}},\varphi\right\rangle $ converges to $\left\langle
\widetilde{\mu}_{t},\varphi\right\rangle $ for all $\varphi\in C_{c}\left(
\mathbb{R}^{d}\right) $, and therefore for all $\varphi\in \text{C}_{b}\left(
\mathbb{R}^{d}\right) $ because $\widetilde{\mu}_{t}\in\mathcal{P}\left(
\mathbb{R}^{d}\right) $. Therefore, with $\widetilde{\mathbb{P}}$-probability
one, $V^{N_{k}^{\prime}}\ast\widetilde{S}_{t}^{N_{k}^{\prime}}$ converges to
$\widetilde{\mu}_{t}$ in the topology of $\mathcal{P}(\mathbb{R}^{d})$. Hence, also the law of $V^{N_{k}^{\prime}}\ast\widetilde{S}_{t}^{N_{k}^{\prime}}$ converges weakly to the law of $\widetilde{\mu}_{t}$ in the topology of $\mathcal{P}(\mathbb{R}^{d})$; namely $\pi_{t}Q_{N_{k}^{\prime}}$ converges weakly to $\pi_{t}P$.  Similarly, if $t_{1}<...<t_{m}\in\left[  0,T\right]  $, the $\mathcal{P}%
(\mathbb{R}^{d})^{m}$-valued random variable $\left(  V^{N_{k}^{\prime}}%
\ast\widetilde{S}_{t_{1}}^{N_{k}^{\prime}},...,V^{N_{k}^{\prime}}%
\ast\widetilde{S}_{t_{m}}^{N_{k}^{\prime}}\right)  $ converges a.s. to
$\left(  \widetilde{\mu}_{t_{1}},...,\widetilde{\mu}_{t_{m}}\right)  $ in the
topology of $\mathcal{P}(\mathbb{R}^{d})^{m}$. Therefore, also the law of
$\left(  V^{N_{k}^{\prime}}\ast\widetilde{S}_{t_{1}}^{N_{k}^{\prime}%
},...,V^{N_{k}^{\prime}}\ast\widetilde{S}_{t_{m}}^{N_{k}^{\prime}}\right)  $
converges weakly to the law of $\left(  \widetilde{\mu}_{t_{1}}%
,...,\widetilde{\mu}_{t_{m}}\right)  $ in the topology of $\mathcal{P}%
(\mathbb{R}^{d})^{m}$, which means that $\pi_{\left(  t_{1},...,t_{m}\right)
}Q_{N_{k}^{\prime}}$ converges weakly to $\pi_{\left(  t_{1},...,t_{m}\right)
}P$. 
\end{proof}

Now, let $(P_{N_{k}})_{k\in\mathbb{N}}$ be a convergent subsequence of
$(P_{N})_{N\in\mathbb{N}}$ (which exists thanks to Lemma \ref{lem:tightness})
with limit $P$ on $\text{C}([0,T];\mathcal{P}(\mathbb{R}^{d}))$. We shall
prove the following two statements.
\begin{itemize}
\item[\textit{(i)}]  The probability measure $P$
is equal to $\delta_{\mu}$ for a suitable $\mu\in\text{C}([0,T];\mathcal{P}%
(\mathbb{R}^{d}))$ which does not depend on the subsequence $\left(
N_{k}\right)  _{k\in\mathbb{N}}$; hence the full sequence $(P_{N}%
)_{N\in\mathbb{N}}$ will converge weakly to $\delta_{\mu}$ and $S^{N}$ will
converge in probability to $\mu$.
\item[\textit{(ii)}] $\mu$ satisfies the conditions in Theorem \ref{th:oelschlager85New}. 
\end{itemize}

\noindent To this end, with the purpose of simplifying notations, we shall prove that the original sequence $(P_{N})_{N\in\mathbb{N}}$ admits a subsequence $(P_{N_{k}})_{k\in\mathbb{N}}$ which converges weakly to $\delta_{\mu}$ for a unique $\mu\in
C([0,T];\mathcal{P}(\mathbb{R}^{d}))$ satisfying all the conditions of Theorem
\ref{th:oelschlager85New}. The same argument applied to any subsequence
$(P_{N_{k}})_{k\in\mathbb{N}}$ in place of the original $(P_{N})_{N\in
\mathbb{N}}$, proves the claim above; this will be the content of Proposition \ref{prop:identification}.

\noindent Denote by $\Lambda\subset\text{C}([0,T];\mathcal{P}(\mathbb{R}%
^{d}))$ the set of all $\left(  \mu_{t}\right)  _{t\in\left[  0,T\right]  }$
such that there exists $p:[0,T]\times\mathbb{R}^{d}\rightarrow\mathbb{R}$ with
the property that $x\mapsto p\left(  t,x\right)  $ is continuous, bounded, non
negative, $\int_{\mathbb{R}^{d}}p\left(  t,x\right)  dx=1$ and $\mu_{t}\left(
dx\right)  =p\left(  t,x\right)  dx$ for all $t\in\left[  0,T\right]  $.
Since
\[
t\mapsto\int_{\mathbb{R}^{d}}p\left(  t,x\right)  \varphi\left(  x\right)
dx=\left\langle \mu_{t},\varphi\right\rangle
\]
is continuous for every $\varphi\in\text{C}_{b}\left(  \mathbb{R}^{d}\right)
$, $p$ is measurable in $\left(  t,x\right)  $ and weakly continuous in $t$,
in the previous sense.\newline\noindent Given \textcolor{black}{$\alpha \in \text{C}_b([0,T] \times \mathbb{R}^{d} ; \mathbb{R}^{d})$}, $\varphi\in\text{C}_{c}^{1,2}\left(  [0,T]\times\mathbb{R}^{d}\right)  $ and
$\mu\in\Lambda$, set
\begin{equation}
\resizebox{1.0 \textwidth}{!}{$
\Phi_{\varphi}\left(  \mu\right)  =\sup_{t\in\left[  0,T\right]  }\left\vert
\left\langle \mu_{t},\varphi(t,\,\cdot\,)\right\rangle -\left\langle \mu_{0}%
,\varphi(0,\cdot)\right\rangle -\int_{0}^{t}\left\langle \mu_{s}, \mathcal{A}\varphi(s,\,\cdot\,) -\left(  \alpha(s)+b\left(
p(s)\right)  \right)  \cdot\nabla\varphi(s,\,\cdot\,)\right\rangle ds\right\vert
$}
\end{equation}
where for the sake of space $b\left( p(s)\right)  $ denotes the function $b\left(\,\cdot\,,p\left(
s,\,\cdot\,\right)  \right)  $ and $p\left(  s,\,\cdot\,\right)  $ is the density of
$\mu_{s}$. Moreover, we remind that $\mathcal{A}$ is the operator defined in Eq.~ \eqref{eq:operator}.

\begin{proposition}\label{prop:identification}
Let $\left(  N_{k}\right)  $ be a subsequence such that $P_{N_{k}}$ converges
in law to $P$ on $\text{C}([0,T];\mathcal{P}(\mathbb{R}^{d}))$. Then:

\begin{itemize}
\item[\textit{(i)}] $P\left(  \Lambda\right)  =1$.

\item[\textit{(ii)}] $\int\left(  \Phi_{\varphi}\left(  \mu\right)
\wedge1\right)  P\left(  d\mu\right)  =0$ for every $\varphi\in$C$_{c}%
^{1,2}\left(  [0,T]\times\mathbb{R}^{d}\right)  $.
\end{itemize}
\end{proposition}

\begin{proof}
\noindent The proof is divided in four steps. Before proceeding, notice that by Lemma \ref{lem:lemma_previous} also $(Q_{N_{k}})_{k\in\mathbb{N}}$ converges weakly to $P$.\\
\indent\textbf{\textit{Step 1}}\quad On an auxiliary probability space, let \textcolor{black}{$(\mu_{t})_{0 \leq t \leq T}$}
be a process with law $P$. Given $t\in\left[  0,T\right]  $, $S_{t}^{N_{k}}$
converges in law to $\mu_{t}$. Moreover, $V^{N_{k}}\ast S_{t}^{N_{k}}$
satisfies the assumptions of Lemma \ref{lem:upgrade} of Appendix
\ref{app:holder_heat_bound_1}. Therefore $P\left(  \Lambda\right)  =1$.

\indent\textbf{\textit{Step 2}}\quad For every $\delta\in(0,1)$ and $\mu\in
\mathcal{P}(\mathbb{R}^{d})$, let $\mathcal{P}_{\delta}\mu$ denote the
following function:
\[
\left(  \mathcal{P}_{\delta}\mu\right)  \left(  x\right)  =\int_{\mathbb{R}%
^{d}}G\left(  \delta,x-y\right)  \mu\left(  dy\right)  .
\]
Moreover, introduce for $\varphi\in\text{C}_{c}^{1,2}\left(  [0,T]\times
\mathbb{R}^{d}\right)  $ and $\delta\in\left(  0,1\right)  $, the regularized
functional, defined on $\mu\in C([0,T];\mathcal{P}(\mathbb{R}^{d}))$ (instead
of $\Lambda$)
\[
\resizebox{1.0 \textwidth}{!}{$
\Phi_{\varphi,\delta}\left(  \mu\right)  =\sup_{t\in\left[  0,T\right]
}\left\vert \left\langle \mu_{t},\varphi(t,\,\cdot\,)\right\rangle -\left\langle
\mu_{0},\varphi(0,\,\cdot\,)\right\rangle -\int_{0}^{t}\left\langle \mu_{s}
,  \mathcal{A}\varphi(s,\,\cdot\,)-\left(  \alpha(s)+b\left(  \mathcal{P}_{\delta}\mu_{s}\right)  \right)  \cdot\nabla
\varphi(s,\,\cdot\,)\right\rangle ds\right\vert .$}
\]
It is easy to check the previous functional is continuous on $\text{C}%
([0,T];\mathcal{P}(\mathbb{R}^{d}))$. Therefore, being $\Phi_{\varphi,\delta
}\left(  \cdot\right)  \wedge1$ continuous and bounded,
\[
\lim_{k\rightarrow\infty}\int\left(  \Phi_{\varphi,\delta}\left(  \mu\right)
\wedge1\right)  Q_{N_{k}}\left(  d\mu\right)  =\int\left(  \Phi_{\varphi
,\delta}\left(  \mu\right)  \wedge1\right)  P\left(  d\mu\right)  .
\]
Recall we know that $P\left(  \Lambda\right)  =1$. For each $\mu\in\Lambda$
and $s\in\left[  0,T\right]  $ it holds that:
\[
\lim_{\delta\rightarrow0}\mathcal{P}_{\delta}\mu_{s}=p(s),
\]
locally in the uniform topology, where $p(s)$ is the density of $\mu_{s}$;
therefore,
\[
\lim_{\delta\rightarrow0}b\left(  \mathcal{P}_{\delta}\mu_{s}\right)
=b\left( p(s)\right)
\]
locally in the uniform topology, and it is a bounded convergence. Hence,
thanks to the local cut-off given by $\varphi(s)$ we have:
\begin{equation*}
\begin{split}
\left\langle \mu_{s},b\left(  \mathcal{P}_{\delta}\mu_{s}\right)  \cdot
\nabla\varphi(s ,\,\cdot\,)\right\rangle &=\left\langle p(s),b\left(  \mathcal{P}%
_{\delta}\mu_{s}\right)  \cdot\nabla\varphi(s,\,\cdot\,)\right\rangle \\
&\overset{\delta
\rightarrow0}{\rightarrow}\left\langle p(s),b\left( p(s)\right)
\cdot\nabla\varphi(s,\,\cdot\,)\right\rangle =\left\langle \mu_{s},b\left(
p(s)\right)  \cdot\nabla\varphi(s,\,\cdot\,)\right\rangle .
\end{split}
\end{equation*}

\noindent By Lebesgue dominated convergence we conclude that
\[
\lim_{\delta\rightarrow0}\Phi_{\varphi,\delta}\left(  \mu\right)
=\Phi_{\varphi}\left(  \mu\right)
\]
and thus again, by the same theorem,
\[
\lim_{\delta\rightarrow0}\int\left(  \Phi_{\varphi,\delta}\left(  \mu\right)
\wedge1\right)  P\left(  d\mu\right)  =\int\left(  \Phi_{\varphi}\left(
\mu\right)  \wedge1\right)  P\left(  d\mu\right)  .
\]
Therefore
\[
\int\left(  \Phi_{\varphi}\left(  \mu\right)  \wedge1\right)  P\left(
d\mu\right)  =\lim_{\delta\rightarrow0}\lim_{k\rightarrow\infty}\int\left(
\Phi_{\varphi,\delta}\left(  \mu\right)  \wedge1\right)  Q_{N_{k}}\left(
d\mu\right)  .
\]
In the next step, we prove that this double limit, taken in the specified
order, is zero.\newline\indent\textbf{\textit{Step 3}}\quad We have the following
identity:
\begin{equation}
\int\left(  \Phi_{\varphi,\delta}\left(  \mu\right)  \wedge1\right)  Q_{N_{k}%
}\left(  d\mu\right)  =\mathbb{E}\left[  \Phi_{\varphi,\delta}\left(
V^{N_{k}}\ast S^{N_{k}}\right)  \wedge1\right]  .\label{eq:expexted}%
\end{equation}
Choosing $V^{N_{k},-}\ast\varphi$ as test function,
\[
\resizebox{1.0 \textwidth}{!}{$
\begin{split}
&  \left\langle S_{t}^{N_{k}},(V^{N_{k},-}\ast\varphi)\left(t, \,\cdot\,\right)
\right\rangle -\left\langle S_{0}^{N_{k}},(V^{N_{k},-}\ast\varphi)\left(
0,\,\cdot\,\right)  \right\rangle -\int_{0}^{t}\left\langle S_{s}^{N_{k}}, \mathcal{A}(V^{N_{k},-}\ast\varphi)(s,\,\cdot\,)
\right\rangle ds\\
&  =\int_{0}^{t}\left\langle S_{s}^{N_{k}},\left(  \alpha(s)+b\left(
V^{N_{k}}\ast S_{s}^{N_{k}}\right)  \right)  \cdot (\nabla V^{N_{k},-}
\ast\varphi)(s,\,\cdot\,)\right\rangle ds+M_{t}^{N_{k},V^{N_{k},-}\ast\varphi},
\end{split}$}
\]
where $M_t^{N_k, V^{N_k,-} \ast \varphi}$ denotes the martingale \eqref{eq:martingale_generic_test_function} in which $N$ and $\varphi$ have been replaced by $N_k$ and $V^{N_k,-} \ast \varphi$, respectively. Thus, 
\begin{equation}
\resizebox{1.0 \textwidth}{!}{$
\begin{split}
\left\langle V^{N_{k}}\ast S_{t}^{N_{k}},  \varphi\left(t,\,\cdot\,\right)
\right\rangle  & -\left\langle V^{N_{k}}\ast S_{0}^{N_{k}},\varphi\left(
0,\,\cdot\,\right)  \right\rangle -\int_{0}^{t}\left\langle V^{N_{k}}\ast S_{s}^{N_{k}%
},\mathcal{A}\varphi(s,\,\cdot\,)\right\rangle ds\\
&  =\int_{0}^{t}\left\langle V^{N_{k}}\ast\left[  \left(  \alpha(s)+b\left(
V^{N_{k}}\ast S_{s}^{N_{k}}\right)  \right)  S_{s}^{N_{k}}\right]
,\nabla\varphi(s,\,\cdot\,)\right\rangle ds+M_{t}^{N_{k},V^{N_{k},-}\ast\varphi}.\\
\end{split}$}
\end{equation}
For the sake of space, we set for $t\in\lbrack0,T]$:
\begin{equation*}
\begin{split}
& V_{t}^{k,\ast}\doteq V^{N_{k}}\ast S_{t}^{N_{k}}\quad\quad\quad\,\, V_{t}^{k,\alpha,\ast
}\doteq V^{N_{k}}\ast(\alpha(t)S_{t}^{N_{k}})\\
& V_{t}^{k,\alpha,b}%
\doteq\alpha(s)+b(V_{s}^{k,\ast})\quad V_{t}^{k,\alpha,b,\delta}\doteq\alpha(s)+b(\mathcal{P}_{\delta}(V_{s}^{k,\ast})).
\end{split}
\end{equation*}
Now, we compute the expected value on the right-hand side of Eq.~\eqref{eq:expexted}.
\begin{equation}
\resizebox{1.0 \textwidth}{!}{$
\begin{split}
&  \mathbb{E}\left[  \Phi_{\varphi,\delta}(V_{t}^{k,\ast})\wedge1\right]  \\
&  \leq\mathbb{E}\left[  \sup_{t\in\left[  0,T\right]  }\left\vert
\left\langle V_{t}^{k,\ast},\varphi(t,\,\cdot\,)\right\rangle -\left\langle
V_{0}^{k,\ast},\varphi(0,\,\cdot\,)\right\rangle -\int_{0}^{t}\left\langle
V_{s}^{k,\ast},\mathcal{A}\varphi(s,\,\cdot\,)
-V_{t}^{k,\alpha,b,\delta}\cdot\nabla\varphi(s,\,\cdot\,)\right\rangle ds\right\vert
\right]  \\
&  =\mathbb{E}\left[  \sup_{t\in\left[  0,T\right]  }\left\vert \int_{0}%
^{t}\left\langle V^{N_{k}}\ast\left[  V_{t}^{k,\alpha,b}\,S_{s}^{N_{k}%
}\right]  ,\nabla \varphi(s,\,\cdot\,)\right\rangle ds+M_{t}^{N_{k},V^{N_{k},-}\ast\varphi
}-\int_{0}^{t}\left\langle V_{s}^{k,\ast},V_{s}^{k,\alpha,b,\delta}\cdot
\nabla\varphi(s,\,\cdot\,)\right\rangle ds\right\vert \right]  \\
&  \leq\mathbb{E}\left[  \left\vert M_{T}^{N_{k},V^{N_{k},-}\ast\varphi
}\right\vert ^{2}\right]  ^{1/2}+\mathbb{E}\left[  \int_{0}^{T}\left\vert
\left\langle V_{t}^{k,\alpha,\ast},\nabla\varphi(s,\,\cdot\,)\right\rangle -\left\langle
V_{s}^{k,\ast},\alpha(s)\cdot\nabla\varphi(s,\,\cdot\,)\right\rangle \right\vert ds\right]  \\
&  +\mathbb{E}\left[  \int_{0}^{T}\left\vert \left\langle V^{N_{k}}\ast\left[
b(V_{s}^{k,\ast})S_{s}^{N_{k}}\right]  ,\nabla\varphi(s,\,\cdot\,)\right\rangle
-\left\langle V_{s}^{k,\ast},b(\mathcal{P}_{\delta}(V_{s}^{k,\ast}%
))\cdot\nabla\varphi(s,\,\cdot\,)\right\rangle \right\vert ds\right]  \doteq(i)+(ii)+(iii).
\end{split}$}
\label{eq:expected_value}
\end{equation}
In the previous equation, we use the following bound
\[
\mathbb{E}\left[  \sup_{t\in\left[  0,T\right]  }\left\vert M_{t}%
^{N_{k},V^{N_{k},-}\ast\varphi}\right\vert \right]  \leq C\mathbb{E}\left[
\left\vert M_{T}^{N_{k},V^{N_{k},-}\ast\varphi}\right\vert ^{2}\right]  ^{1/2}%
\]
due to Doob's inequality. At this point, we have that the terms $(i)-(iii)$ in
Eq.~\eqref{eq:expected_value} converge to zero as $N_{k}\rightarrow\infty$. By
hypothesis $\Vert V^{N_{k},-}\ast\nabla\varphi\Vert_{\infty}$ is bounded
because $\nabla\varphi$ is uniformly continuous and hence $V^{N_{k},-}%
\ast\nabla\varphi$ converges uniformly to $\nabla\varphi$. This implies that
\eqref{eq:expected_value}-$(i)$ converges to zero. Indeed,
\[
\mathbb{E}\left[  \left(  M_{T}^{N_{k},V^{N_{k},-}\ast\varphi}\right)
^{2}\right]  =\frac{1}{N^{2}}\sum_{i=1}^{N}\int_{0}^{t}\mathbb{E}\left[
\left\vert \nabla V^{N_{k},-}\ast\varphi(s,X_{s}^{N,i})\right\vert
^{2}\right]  ds\leq\frac{1}{N}\left\Vert V^{N_{k},-}\ast\nabla\varphi
\right\Vert _{\infty}^{2}T
\]
The uniform converges of $V^{N_{k},-}\ast\nabla\varphi$ and $V^{N_{k},-}%
\ast\left(  \alpha(\,\cdot\,) \cdot\nabla\varphi\right)  $, the weak convergence of
$S_{s}^{N_{k}}$ (realized a.s. on an auxiliary probability space, by Skorohod
theorem) and Lebesgue dominated convergence theorem implies that also the term
\eqref{eq:expected_value}-$(ii)$ converges to zero. The converges to zero of
the third term is more delicate and it will be proved here below in the third
step.\newline\indent\textbf{\textit{Step 4}}\quad Let us consider
\begin{equation*}%
\begin{split}
&  \mathbb{E}\left[  \int_{0}^{T}\left\vert \left\langle S_{s}^{N_{k}}%
,b(V_{s}^{k,\ast})\cdot (V^{N_{k}}\ast\nabla\varphi)(s,\,\cdot\,)-V^{N_{k}}\ast\left[
b(\mathcal{P}_{\delta}(V_{s}^{k,\ast}))\cdot\nabla\varphi(s,\,\cdot\,)\right]
\right\rangle \right\vert (\,\cdot\,)   ds\right]\\
&  \leq\int_{0}^{T}\mathbb{E}\left[  \int_{\mathbb{R}^{d}}\int_{\mathbb{R}%
^{d}}V^{N_{k}}\left(  x-y\right)  \left\vert b(\mathcal{P}_{\delta}%
(V_{s}^{k,\ast})(y))-b(V_{s}^{k,\ast}(x))\right\vert \left\vert \nabla
\varphi\left(s, y\right)  \right\vert \,dy\,S_{s}^{N_{k}}\left(  dx\right)
\right]  \,ds\\
&  \leq L_{b}\int_{0}^{T}\mathbb{E}\left[  \int_{\mathbb{R}^{d}}%
\int_{\mathbb{R}^{d}}V^{N_{k}}\left(  x-y\right)  \left\vert \mathcal{P}%
_{\delta}(V_{s}^{k,\ast})(y)-V_{s}^{k,\ast}(x)\right\vert \left\vert
\nabla\varphi\left(s, y\right)  \right\vert dy\,S_{s}^{N_{k}}\left(  dx\right)
\right]  ds\\
&  \leq\int_{0}^{T}\mathbb{E}\left[  \int_{\mathbb{R}^{d}}\int_{\mathbb{R}%
^{d}}V^{N_{k}}\left(  x-y\right)  \left\vert \mathcal{P}_{\delta}%
(V_{s}^{k,\ast})(y)-V_{s}^{k,\ast}(y)\right\vert \left\vert \nabla
\varphi(s, y)\right\vert dy\,S_{s}^{N_{k}}(dx)\right]  \,ds\\
&  +\int_{0}^{T}\mathbb{E}\left[  \int_{\mathbb{R}^{d}}\int_{\mathbb{R}^{d}%
}V^{N_{k}}\left(  x-y\right)  \left\vert V_{s}^{k,\ast}\left(  y\right)
-V_{s}^{k,\ast}\left(  x\right)  \right\vert \left\vert \nabla\varphi\left(
s, y\right)  \right\vert dy\,S_{s}^{N_{k}}\left(  dx\right)  \right]  ds.
\end{split}
\end{equation*}

\noindent We now compute the following two bounds (notice that we use the
explicit expression). The first is given by:
\[%
\begin{split}
&  \left\vert \mathcal{P}_{\delta}\left( V^{N_{k}}\ast S_{s}^{N_{k}}\right)
\left(  y\right) - (V^{N_{k}}\ast S_{s}^{N_{k}} )\left(  y\right)  \right\vert
=\left\vert \mathbb{E}\left[  \left(  V^{N_{k}}\ast S_{s}^{N_{k}}\right)
\left(  y+W_{\delta}\right)  -\left(  V^{N_{k}}\ast S_{s}^{N_{k}}\right)
\left(  y\right)  \right]  \right\vert \\
&  \leq\mathbb{E}\left[  \left\vert \left(  V^{N_{k}}\ast S_{s}^{N_{k}%
}\right)  \left(  y+W_{\delta}\right)  -\left(  V^{N_{k}}\ast S_{s}^{N_{k}%
}\right)  \left(  y\right)  \right\vert \right]  \leq\left[  V^{N_{k}}\ast
S_{s}^{N_{k}}\right]  _{\gamma}\mathbb{E}\left[  \left\vert W_{\delta
}\right\vert ^{\gamma}\right]  \\
&  \leq C_{\gamma}\left[  V^{N_{k}}\ast S_{s}^{N_{k}}\right]  _{\gamma}%
\delta^{\gamma/2},
\end{split}
\]
whereas the second, since $V$ has compact support, say 1, so that the
support of $V^{N_{k}}$ is $\epsilon_{N_{k}}$, by
\[%
\begin{split}
&  V^{N_{k}}\left(  x-y\right)  \left\vert (V^{N_{k}}\ast S_{s}^{N_{k}})\left(
y\right)  - (V^{N_{k}}\ast S_{s}^{N_{k}})\left(  x\right)  \right\vert \leq
V^{N_{k}}\left(  x-y\right)  \left[  V^{N_{k}}\ast S_{s}^{N_{k}}\right]
_{\gamma}\left\vert x-y\right\vert ^{\gamma}\\
&  \leq\epsilon_{N_{k}}^{\gamma}V^{N_{k}}\left(  x-y\right)  \left[  V^{N_{k}%
}\ast S_{s}^{N_{k}}\right]  _{\gamma}.
\end{split}
\]
Therefore
\[%
\begin{split}
&  \mathbb{E}\left[  \int_{0}^{T}\left\vert \left\langle S_{s}^{N_{k}}%
,b(V_{s}^{k,\ast})\cdot (V^{N_{k}}\ast\nabla\varphi)(s,\,\cdot\,)-V^{N_{k}}\ast\left[
b(\mathcal{P}_{\delta}(V_{s}^{k,\ast}))\cdot\nabla\varphi(s,\,\cdot\,)\right]
\right\rangle \right\vert (\,\cdot\,)   ds\right]\\
&  \leq C_{\gamma}\delta^{\gamma/2}\int_{0}^{T}\mathbb{E}\left[  \left[
V^{N_{k}}\ast S_{s}^{N_{k}}\right]  _{\gamma}\int_{\mathbb{R}^{d}}%
\int_{\mathbb{R}^{d}}V^{N_{k}}\left(  x-y\right)  \left\vert \nabla
\varphi\left(s,y\right)  \right\vert dy\,S_{s}^{N_{k}}\left(  dx\right)
\right]  ds\\
&  +\epsilon_{N_{k}}^{\gamma}\int_{0}^{T}\mathbb{E}\left[  \left[  V^{N_{k}%
}\ast S_{s}^{N_{k}}\right]  _{\gamma}\int_{\mathbb{R}^{d}}\int_{\mathbb{R}%
^{d}}V^{N_{k}}\left(  x-y\right)  \left\vert \nabla\varphi\left(s, y\right)
\right\vert dy\,S_{s}^{N_{k}}\left(  dx\right)  \right]  ds\\
&  \leq\left(  C_{\gamma}\delta^{\gamma/2}+\epsilon_{N_{k}}^{\gamma}\right)
\left\Vert \nabla\varphi\right\Vert _{\infty}\int_{0}^{T}\mathbb{E}\left[
\left[  V^{N_{k}}\ast S_{s}^{N_{k}}\right]  _{\gamma}\int_{\mathbb{R}^{d}}%
\int_{\mathbb{R}^{d}}V^{N_{k}}\left(  x-y\right)  dy\,S_{s}^{N_{k}}\left(
dx\right)  \right]  ds\\
&  =\left(  C_{\gamma}\delta^{\gamma/2}+\epsilon_{N_{k}}^{\gamma}\right)
\left\Vert \nabla\varphi\right\Vert _{\infty}\int_{0}^{T}\mathbb{E}\left[
\left[  V^{N_{k}}\ast S_{s}^{N_{k}}\right]  _{\gamma}\right]  ds,
\end{split}
\]
which converges to zero as $k\rightarrow\infty$ and then $\delta
\rightarrow\infty$ thanks to the first estimate of Lemma
\ref{lem:hoelder_estim_empirical_density}.
\end{proof}

\noindent In order to complete the proof of Theorem \ref{th:oelschlager85New}
we have to prove that $P$ is supported on a class of solutions of equation
\eqref{eq:deterministic} where we may apply the uniqueness result of Appendix \ref{subsec:HJ_and_FP_mild} now we know that $P$ is supported on $\Lambda$ and
satisfies $\int\left(  \Phi_{\varphi}\left(  \mu\right)  \wedge1\right)
P\left(  d\mu\right)  =0$ for every $\varphi\in$C$_{c}^{1,2}\left(
[0,T]\times\mathbb{R}^{d}\right)  $. On an auxiliary probability space
$\left(  \widetilde{\Omega},\widetilde{\mathcal{F}},\widetilde{\mathbb{P}%
}\right)  $ with expectation $\widetilde{\mathbb{E}}$, let \textcolor{black}{$(\widetilde{\mu
}_{t})_{0 \leq t \leq T}$} be a process with law $P$. We know that%
\[
\widetilde{\mathbb{E}}\left[  \Phi_{\varphi}\left(  \widetilde{\mu}\right)
\wedge1\right]  =0
\]
hence
\begin{equation}
\resizebox{1.0 \textwidth}{!}{$
\sup_{t\in\left[  0,T\right]  }\left\vert \left\langle \widetilde{\mu}%
_{t},\varphi(t,\,\cdot\,)\right\rangle -\left\langle \mu_{0},\varphi(0,\,\cdot\,)\right\rangle
-\int_{0}^{t}\left\langle \widetilde{\mu}_{s}, \mathcal{A}\varphi(s,\,\cdot\,) -\left(  \alpha(s)+b\left(  \widetilde{p}_{s}\right)
\right)  \cdot\nabla\varphi(s,\,\cdot\,)\right\rangle ds\right\vert
=0\label{towards weak formulation}$}
\end{equation}
with $\widetilde{\mathbb{P}}$-probability one. The set C$_{c}^{1,2}\left(
[0,T]\times\mathbb{R}^{d}\right)  $ is separable in the natural metric and
therefore we may find a dense countable family $\mathcal{D}\subset$%
C$_{c}^{1,2}\left(  [0,T]\times\mathbb{R}^{d}\right)  $; it follows that we
may reverse the quantifiers and get that with $\widetilde{\mathbb{P}}%
$-probability one identity (\ref{towards weak formulation}) holds for all
$\varphi\in\mathcal{D}$. Obviously we can also write
\[
\resizebox{1.0 \textwidth}{!}{$
\sup_{t\in\left[  0,T\right]  }\left\vert \left\langle \widetilde{p}(t),\varphi(t,\,\cdot\,)\right\rangle -\left\langle p_0,\varphi(0,\,\cdot\,)\right\rangle
-\int_{0}^{t}\left\langle \widetilde{p}(s), \mathcal{A}\varphi(s, \,\cdot\,) -\left(  \alpha(s)+b\left(  \widetilde{p}_{s}\right)
\right)  \cdot\nabla\varphi(s,\,\cdot\,)\right\rangle ds\right\vert =0$}
\]
since $\widetilde{\mu}_{t}$ has density $\widetilde{p}_{t}$, and also $\mu
_{0}$ has density $p_{0}$ by assumption. From the density of $\mathcal{D}$ and
classical limit theorems we get that, with $\widetilde{\mathbb{P}}%
$-probability one, the previous identity holds for every $\varphi\in$%
C$_{c}^{1,2}\left(  [0,T]\times\mathbb{R}^{d}\right)  $. Recall that we denote
by $G\left(  t,x\right)  $ the density of blackian motion in $\mathbb{R}^{d}$
and by $\mathcal{P}_{t}$ the associated heat semigroup. From the previous
identity we deduce
\begin{equation}
\left\langle \widetilde{p}(t),\psi(\,\cdot\,)\right\rangle =\left\langle \mathcal{P}%
_{t}p_0,\psi(\,\cdot\,)\right\rangle -\int_{0}^{t}\left\langle \nabla\mathcal{P}%
_{t-s}\left(  \alpha(s)+b(\widetilde{p}(s))\right)  ,\psi(\,\cdot\,)\right\rangle
ds\label{mild weak}%
\end{equation}
for every $\psi\in \text{C}_{c}^{2}\left(  \mathbb{R}^{d}\right)  $. Indeed, given
$t\in\left[  0,T\right]  $ and $\psi\in \text{C}_{c}^{2}\left(  \mathbb{R}%
^{d}\right)  $, consider the test function $\varphi^{(t)}(s)=\mathcal{P}_{t-s}\psi
$ for $s \in [0,t]$; by approximation by functions of class C$_{c}^{1,2}\left(  [0,T]\times
\mathbb{R}^{d}\right)$, we deduce

\begin{equation}
\begin{split}
\left\langle \widetilde{p}(t),\mathcal{P}_{t-t}\psi\right\rangle
&-\left\langle p_0,\mathcal{P}_{t-0}\psi\right\rangle \\
&-\int_{0}%
^{t}\left\langle \widetilde{p}(s), \mathcal{A}(\mathcal{P}_{t-\cdot}\psi)(s) -\left(  \alpha(s)+b\left(  \widetilde{p}(s)\right)  \right)  \cdot\nabla\mathcal{P}_{t-s}\psi\right\rangle ds
\end{split}
\end{equation}
which simplifies to
\[
\left\langle \widetilde{p}(t),\psi\right\rangle =\left\langle p_0 ,\mathcal{P}_{t}\psi\right\rangle -\int_{0}^{t}\left\langle \alpha
_{s}+b(\widetilde{p}(s)),\nabla\mathcal{P}_{t-s}\psi\right\rangle ds
\]
and therefore leads to equation (\ref{mild weak}) by simple manipulations. By
the arbitrariness of $\psi$ and the continuity in $x$ of $\widetilde{p}_{t}$ and
of both $\mathcal{P}_{t}f$ and $\nabla\mathcal{P}_{t-s}f$\ (this one only for
$s<t$) for every continuous bounded $f$ (here we also use the bound
$\left\Vert \nabla\mathcal{P}_{t-s}f\right\Vert _{\infty}\leq\frac{C}{\left(
t-s\right)  ^{1/2}}\left\Vert f\right\Vert _{\infty}$ and the integrability of
$\frac{C}{\left(  t-s\right)  ^{1/2}}$) we get%
\[
\widetilde{p}\left(t,  x\right)  =\left(  \mathcal{P}_{t}p_0\right)
\left(  x\right)  -\int_{0}^{t}\nabla\mathcal{P}_{t-s}\left(  \alpha(s)+b(\widetilde{p}(s))\right)  \left(  x\right)  ds.
\]
By the same arguments we deduce that $\widetilde{p}$ is continuous in $\left(
t,x\right)  $. Moreover, it is bounded uniformly in $\left(  t,x\right)  $ by
the identity itself, because $\mathcal{P}_{\cdot}p_{0}$ is bounded, $\alpha$
is bounded, $b$ is bounded and again we use $\left\Vert \nabla\mathcal{P}%
_{t-s}f\right\Vert _{\infty}\leq\frac{C}{\left(  t-s\right)  ^{1/2}}\left\Vert
f\right\Vert _{\infty}$. In conclusion $\widetilde{p}$ is of class
$\text{C}_{b}\left(  \left[  0,T\right]  \times\mathbb{R}^{d}\right)  $. In Appendix \ref{subsec:HJ_and_FP_mild} it is proved that in this class there is a unique solution of
the previous mild equation, hence $P$ is supported by a single element. This
completes the proof of Theorem \ref{th:oelschlager85New}.

\section{Approximate Nash equilibria from the mean field game}\label{sec:nash}
\noindent In this section we show that if we have a weak solution $(u, p)$ of the PDE system in Eq.~\eqref{eq:pdesystem}, then we can construct a sequence of approximate Nash equilibria for the corresponding $N$-player game. This is the content of the following theorem.
\begin{theorem}\label{th:Nash_equilibria_MFG}
Let $N \in \mathbb{N}$, $N > 1$. Grant \text{(H1)}-\text{(H4)}. Suppose $(u, p)$ is a weak solution of the PDE system in Eq.~\eqref{eq:pdesystem} and let $\alpha^{*}(t, x) \doteq - \triangledown u(t, x)$ the optimal control of the problem \textbf{OC} in the class \textcolor{black}{$\mathcal{A}^{fb}_{K}$} with $K$ given by Definition \eqref{eq:basic_const}. Set
\begin{equation}
\alpha^{N, i}(t, \mathbf{x}) \doteq \alpha^{*}(t, x_i) \doteq  -\triangledown u(t, x_i),\quad t \in [0,T],\,\mathbf{x} = \left(x_1, \ldots, x_N\right) \in \mathbb{R}^{d \times N},\,\,i \in \textcolor{black}{[[N]]}
\end{equation}
and $\bm{\alpha}^{N} = (\alpha^{N,1}, \ldots, \alpha^{N,N})\textcolor{black}{\in \mathcal{A}^{N;fb}_K}$. Then for every $\varepsilon > 0$, there exist $N_0 = N_0(\varepsilon) \in \mathbb{N}$ such that $\bm{\alpha}^{N}$ is an $\varepsilon$-Nash equilibrium for the $N$-player game whenever $N \geq N_0$.
\end{theorem}
\begin{proof}
The proof is divided in three steps.\\
\indent \textbf{\textit{Step 1}}\, Let $((\Omega_{N}, \mathcal{F}_{N}, (\mathcal{F}_t^{N}), \mathbb{P}^N), \bm{W}^{N}, \bm{X}^{N})$ be a weak solution of Eq.~ \eqref{eq:privatestate_evolution_N_player} under strategy vector $\bm{\alpha}^{N}$.  We note that the function $F$ defined in \eqref{eq:function_F} with $\alpha(s, X_s^{N,i}) = - \triangledown u(s, X_s^{N,i})$  is continuous and bounded; this guarantees the existence of \textcolor{black}{a weak solution of the the system in Eq.~ \eqref{eq:privatestate_evolution_N_player} for any $N\in\mathbb{N}$} 
Let $S^{N}_t$ (resp. $S^{N}$) denote the associated empirical measure on $\mathbb{R}^{d}$ (resp. on the path space $\mathcal{X}$). We are going to show that
\begin{equation}\label{eq:goal_Step1}
\lim_{N \rightarrow \infty} J^{N}_{i}(\bm{\alpha}^{N}) = J(\alpha^{*}).
\end{equation}
\noindent Theorem \eqref{th:oelschlager85New}-$(i)$ enables us to prove the convergence result in Eq.~\eqref{eq:goal_Step1} for the following simplified cost functional, where we do not change the notation for the sake of simplicity:
\begin{equation*}
J_i^{N}(\bm{\alpha}^{N}) = \mathbb{E}\left[\int_{0}^{T}\frac{1}{2}|-\triangledown u(s, X_s^{N,i})|^2\,ds + g(X_T^{N,i})\right].
\end{equation*}
Symmetry of the coefficients allows us to re-write the previous cost functional in terms of $S^{N}_{t}$, $t \in [0,T]$ as
\begin{equation*}
\begin{split}
J_i^{N}(\bm{\alpha}^{N}) &= \mathbb{E}\left[\int_{0}^{T}\frac{1}{N}\sum_{i=1}^{N}\frac{1}{2}|-\triangledown u(s, X_s^{N,i})|^2\,ds + \frac{1}{N}\sum_{i=1}^{N} g(X_T^{N,i})\right]\\
									 &= \mathbb{E}\left[\int_{0}^{T} \langle S_s^{N}, \frac{1}{2}|-\triangledown u(s,\,\cdot\,)|^2\rangle\,ds + \langle S_T^{N}, g(\,\cdot\,)\rangle\right]\\
\end{split}
\end{equation*}
which converges, as $N \rightarrow \infty$, to
\begin{equation*}
\begin{split}
\textcolor{black}{\int_{0}^{T} \langle S_s^{\infty}, \frac{1}{2}|-\triangledown u(s,\,\cdot\,)|^2\rangle\,ds + \langle S_T^{\infty}, g(\,\cdot\,)\rangle}
\end{split}
\end{equation*}
\textcolor{black}{where $S^{\infty}$ is the deterministic limit in probability of the sequence of random empirical measures $(S^{N})_{N\in\mathbb{N}}$ given by Theorem \eqref{th:oelschlager85New}-$(i)$.}

\noindent We claim that $S_t^{\infty} \equiv p(t,\,\cdot\,)$, $t \in [0, T]$, with $p$ the second component of the pair $(u, p)$, i.e. the density of the solution of Eq.~ \eqref{eq:mfgstate} as stated by the Verification Theorem \ref{th:VerificationTheorem}. Theorem \ref{th:oelschlager85New} states that, given $\bm{\alpha}^{N}$, the empirical measure $S_t^{N}$ corresponding to the interacting system with this control converges to a \textcolor{black}{flow of measures} with density $p^{\alpha}(t,\,\cdot\,)$, where we stress the dependence on $\alpha$. In addition, Theorem \ref{th:oelschlager85New}-$(ii)$ states that $p^{\alpha}(t,\,\cdot\,)$  is the mild solution of Eq.~\eqref{eq:deterministic}. By applying the previous result to the optimal control we have that the corresponding empirical measure on $\mathbb{R}^{d}$ converges to $p^{\alpha^{*}}(t,\,\cdot\,)$, mild solution of Eq.~\eqref{eq:deterministic}.
Also $p$, the second component of $(u, p)$, is a mild solution of this equation. The uniqueness Theorem \ref{th:localExistenceUniqueness} now implies that  $p^{\alpha^{*}}(t,\,\cdot\,)$ coincides with $p(t,\,\cdot\,)$. Hence, we can conclude that Eq.~\eqref{eq:goal_Step1} holds.\\
\indent\textbf{\textit{Step 2}} For each $N \in \mathbb{N} \setminus \left\{i\right\}$, \textcolor{black}{let $\beta^{N,i} \in \mathcal{A}^{N;1;fb}_{K}$} such that
\begin{equation*}
J^{N}_i([\bm{\alpha}^{N,-i}, \textcolor{black}{\beta^{N,i}}]) \leq \inf_{\beta \in \textcolor{black}{\mathcal{A}^{N;1;fb}_{K}}} J^{N}_{i}([\bm{\alpha}^{N,-i}, \beta]) + \varepsilon/2.
\end{equation*}
We are going to show the following result:
\begin{equation}\label{eq:goal_Step2}
\liminf_{N \rightarrow \infty} J^{N}_i([\bm{\alpha}^{N,-i}, \textcolor{black}{\beta^{N,i}}]) \geq J(\alpha^{*}).
\end{equation}
\noindent To this aim, we introduce the $N$-player dynamics in the case the first player only deviates from the Nash equilibrium.  For $N \in \mathbb{N}$, consider the system of equations:
\begin{equation}\label{eq:one_player_deviates}
\begin{split}
X_t^{N,1;\beta}&= \textcolor{black}{X_0}^{N,1}+\int_0^t \left( \textcolor{black}{\beta^{N,1}(s,\bm{X}^{N;\beta}_s)} + b(X_s^{N,1;\beta},\frac{1}{N}\sum_{j=1}^NV^N(X_s^{N,1;\beta}-X_s^{N,j;\beta}))\right)\,ds\\
					&+ \textcolor{black}{W_t^{N,1;\beta}}\\
X_t^{N,i;\beta}&= \textcolor{black}{X_0}^{N,i}+\int_0^t\left(\alpha^*(s,X^{N,i;\beta}_s) + b(X_s^{N,i;\beta},\frac{1}{N}\sum_{j=1}^NV^N(X_s^{N,i;\beta}-X_s^{N,j;\beta}))\right)\,ds\\
					& + {W^{N, i;\beta}_t}\\
&\quad\quad\quad  i\in\left\{2, \ldots, N\right\},\,\,t \in [0, T],
\end{split}
\end{equation}
where \textcolor{black}{$\beta^{N,1} \in \textcolor{black}{\mathcal{A}^{N;1;fb}_{K}}$}. We denote with $S^{N;\beta}\doteq(S^{N;\beta}_t)_{t \in [0, T]}$ the empirical measure process on $\mathbb{R}^{d}$ of the previous system.\\
\noindent Now, for each $N \in \mathbb{N}$, \textcolor{black}{let} $((\Omega_{N}, \mathcal{F}_{N}, (\mathcal{F}_t^{N}), \mathbb{Q}^N), \textcolor{black}{\bm{W}^{N;\beta}}, \bm{X}^{N;\beta})$ be a weak solution of Eq.~\eqref{eq:one_player_deviates}. \textcolor{black}{Since} the presence of a deviating player destroys the symmetry of the pre-limit system, \textcolor{black}{following \cite{lacker2020convergence} proof of Theorem 3.10 therein,} we perform a change of measure to restore it. More precisely, we define as $\mathbb{P}^{N}$ the probability measure under which $\bm{X}^{N;\beta}$ has the following dynamics:
\begin{equation*}
\begin{split}
X_t^{N,\textcolor{black}{i};\beta}&= \textcolor{black}{X_0}^{N,\textcolor{black}{i}}+\int_0^t \left( \alpha^*(s,X^{N,i;\beta}_s) + b(X_s^{N,\textcolor{black}{i};\beta},\frac{1}{N}\sum_{j=1}^NV^N(X_s^{N,\textcolor{black}{i};\beta}-X_s^{N,j;\beta}))\right)\,ds +\\
&\quad\quad\quad  + \textcolor{black}{\widehat{W}_t^{N,i;\beta}}, \quad\quad\quad i\in \textcolor{black}{[[N]]},\,\,t \in [0, T],
\end{split}
\end{equation*}
\textcolor{black}{where the $\widehat{W}_t^{N,i;\beta}$ are $\mathbb{P}^N$-Wiener processes, i.e. $\mathbb{P}^N$ is defined via $\frac{d\mathbb{P}^{N}}{d\mathbb{Q}^{N}}\Big\vert_{t=T} \doteq Z_T^N$ where} 
\begin{equation*}
\textcolor{black}{Z^N_t \doteq \mathcal{E}_t\left(\int_0^{\cdot}\left(\bm{\beta}^N(s,\bm{X}^{N;\beta}_s)-\bm{\alpha}^N(s,\bm{X}^{N;\beta}_s)\right)d\bm{W}_s^{N;\beta} \right)}
\end{equation*}
\textcolor{black}{where $\bm{\beta}^N=[\bm{\alpha}^{N,-1},\beta^{N,1}]$ and $\bm{W}^{N;\beta}=(W^{N,1;\beta},\ldots,W^{N,N;\beta})$. We notice that $Z^N$ is a well-defined $\mathbb{Q}^N$-martingale thanks to boundedness of the coefficients.} Theorem  \ref{th:oelschlager85New}-$(i)$ ensures the convergence under $\mathbb{P}^{N}$ of the $S^{N;\beta}$ to $S^{\infty} \equiv \delta_p$. Boundedness of the coefficients also gives uniform integrability of the sequence \textcolor{black}{$((Z^{N}_T)^{-1})_{N \in \mathbb{N}}$}; therefore, the probability measures $\mathbb{Q}^{N}(A) \doteq \mathbb{E}^{\mathbb{P}^{N}}\left[\textcolor{black}{(Z^{N}_T)^{-1}}\mathsf{1}_{A}\right]$, $A \in \mathcal{F}^{N}$, converge to zero whenever $\mathbb{P}^{N}(A)$ converges to zero  in the limit $N \rightarrow \infty$. So the convergence (in law and also in probability) of $S^{N;\beta}$ to $S^{\infty}$ under $\mathbb{P}^{N}$ implies its convergence (in law and also in probability) under $\mathbb{Q}^{N}$ to the same (constant) limit.\\
\noindent Now, \textcolor{black}{in order to gain more compactness in the space of admissible controls,} we interpret the controls in Eq.~\eqref{eq:one_player_deviates} as stochastic relaxed controls (Appendix \ref{app:relaxed_controls}). To this end, we denote with $\textcolor{black}{\overline{B}_{K}(0)} \subset \mathbb{R}^{d}$ the closed ball of radius \textcolor{black}{$K$} around the origin and \textcolor{black}{$\mathcal{R}_K \doteq \mathcal{R}_{\overline{B}_{K}(0)}$}. Then $\textcolor{black}{\mathcal{R}_{K}}$ is compact (Appendix \ref{app:relaxed_controls}). For $N \in \mathbb{N}$, let $\tilde{\beta}_{t}^{1}$ and $\tilde{\alpha}_{t}^{*,i}$, $i \in \{2, \ldots, N\}$, be $\textcolor{black}{\mathcal{R}_{K}}$-valued random measures determined by:
\begin{equation*}
\begin{split}
&\textcolor{black}{\tilde{\beta}_{t}^{N,1}(dx)\,dt \doteq \delta_{\beta^{N,1}(t,\bm{X}_t^{N;\beta})}(x)\,dx\,dt,\quad\quad\quad\,\,\,\,(t, x) \in [0,t] \times \overline{B}_{K}(0)}\\
&\tilde{\alpha}_{t}^{*,i}(dx)\,dt \doteq  \delta_{\alpha^{*}(t, X_t^{N, i; \beta})}(x)\,dx\,dt\quad\,\,\,\, (t, x) \in [0,t] \times \textcolor{black}{\overline{B}_{K}(0)},\,\,i\in\{2,\ldots,N\}.\\
\end{split}
\end{equation*}
We rewrite Eq.~\eqref{eq:one_player_deviates} in terms of these relaxed controls:
\begin{equation}\label{eq:one_player_deviates_deviation}
\begin{split}
X_t^{N,1;\beta}= \textcolor{black}{X_0}^{N,1} &+ \int_{[0,t]\times \overline{B}_{\textcolor{black}{K}}(0)}  x \tilde{\beta}_{s}^{N,1}(dx)\,ds\\
&+ \int_{0}^{t} b(X_s^{N,1;\beta},\frac{1}{N}\sum_{j=1}^NV^N(X_s^{N,1;\beta}-X_s^{N,j;\beta}))\,ds + \textcolor{black}{W_t^{N,1;\beta}}\\
X_t^{N,i;\beta}=  \textcolor{black}{X_0}^{N,i} &+ \int_{[0,t]\times \overline{B}_{\textcolor{black}{K}}(0)} x \tilde{\alpha}_{i,s}^{*}(dx)\textcolor{black}{\,ds}\\
&+ \int_{0}^{t} b(X_s^{N,i;\beta},\frac{1}{N}\sum_{j=1}^NV^N(X_s^{N,i;\beta}-X_s^{N,j;\beta})) \,ds + \textcolor{black}{W^{N, i;\beta}_t}\\
&\quad\quad\quad  i\in\left\{2, \ldots, N\right\},\,\,t \in [0, T].
\end{split}
\end{equation}
\noindent We do the following claims. \textbf{Claim a.}: the family $\left(\mathbb{P}^{N} \circ (X^{N, 1; \beta}, \textcolor{black}{\tilde{\beta}^{N,1}}, S^{N;\beta})^{-1}\right)_{N \in \mathbb{N}}$ is tight in $\mathcal{P}(\mathcal{X} \times \textcolor{black}{\mathcal{R}_{K}} \times \mathcal{P}(\mathbb{R}^{d}))$ and thus it admits a convergent subsequence. We denote by $(X^{\beta^{*}}, \tilde{\beta}^{*,1},\,p)$ the limit of the subsequence \textcolor{black}{that can be constructed by means of Skorokhod's representation theorem on a suitable limiting probability space $(\Omega^{\beta^*},\mathcal{F}^{\beta^*},\mathbb{Q}^{\beta^*})$};
\textbf{Claim b.}: the limit $X^{\beta^{*}}$ has the following representation:
\begin{equation}\label{eq:limit_equation}
X_t^{\tilde{\beta}^{*}} = \textcolor{black}{X_0} +  \int_{[0,T]\times \overline{B}_{\textcolor{black}{K}}(0)} x \tilde{\beta}^{*,1}_{s}(dx)\,ds + \int_{0}^{t} b(X_s^{\tilde{\beta}^{*}}, p(s, X_s^{\tilde{\beta}^{*}}))\,ds + \textcolor{black}{W^{\beta^*}_t},\quad t \in [0,T]
\end{equation}
\textcolor{black}{on $(\Omega^{\beta^*},\mathcal{F}^{\beta^*},\mathbb{Q}^{\beta^*})$ where $W^{\beta^*}$ is a Wiener process, i.e. there exist a filtration $(\mathcal{F}^{\beta^*}_t)$ and an $(\mathcal{F}^{\beta^*}_t)$-Wiener process $W^{\beta^*}$ on $(\Omega^{\beta^*},\mathcal{F}^{\beta^*},\mathbb{Q}^{\beta^*})$. such that $X^{\tilde{\beta}^{*}}$ has representation \eqref{eq:limit_equation}.}
If both \textbf{Claim a} and \textbf{Claim b} hold, by setting $\beta_t^{*} \doteq \int_{\overline{B}_{\textcolor{black}{K}}(0)} x \tilde{\beta}^{*,1}_{t}(dx)$, we have that $J^{N}_i([\bm{\alpha}^{N,-i}, \textcolor{black}{\beta^{N,i}}])$ converges to
\begin{equation*}
J(\beta^{*}) = \mathbb{E}\left[\int_{0}^{T} \frac{1}{2} |\beta^{*}_s|^2\,ds + g(X_T^{\beta^{*}})\right]
\end{equation*}
along the selected subsequence with $J(\beta^{*}) \geq J(\alpha^{*})$.  Eq.~\eqref{eq:goal_Step2} follows by taking the \textcolor{black}{limit inferior} of the sequence.\\
\noindent We now prove the two claims.\\
\indent\textit{Proof of \textbf{Claim a.}} Tightness of $(\mathbb{P}^{N} \circ (X^{N,1:\beta})^{-1})$ and of $(\mathbb{P}^{N} \circ (S^{N;\beta})^{-1})$ under $\mathbb{Q}^{N}$ follows from their tightness under $\mathbb{P}^{N}$. On the other hand, $(\mathbb{P}^{N} \circ (\textcolor{black}{\tilde{\beta}^{N,1}})^{-1})$ is tight in $\mathcal{P}(\textcolor{black}{\mathcal{R}_{K}})$ because $\textcolor{black}{\mathcal{R}_{K}}$ is compact. This implies that $\left(\mathbb{P}^{N} \circ (X^{N, 1; \beta}, \tilde{\beta}^{N,1}, S^{N;\beta})^{-1}\right)_{N \in \mathbb{N}}$ is tight in $\mathcal{P}(\mathcal{X} \times \textcolor{black}{\mathcal{R}_{K}} \times \mathcal{P}(\mathbb{R}^{d}))$.\\
\indent\textit{Proof of \textbf{Claim b.}} We use a characterization of solutions to Eq.~ \eqref{eq:limit_equation} with fixed measure variable through a martingale problem in the sense of \cite{stroock2007multidimensional} \textcolor{black}{(see \cite{el1990martingale} for a study of the martingale problems we employ)}. Let $f \in \text{C}_{c}^{2}(\mathbb{R}^{d})$ and let us define the process $M^{f}$ on $(\mathcal{X}\times\textcolor{black}{\mathcal{R}_{K}}, \mathcal{B}(\mathcal{X}\times\textcolor{black}{\mathcal{R}_{K}}))$ by 
\begin{equation}\label{eq:process_martingale}
\begin{split}
M^{f}_t(\varphi, \rho) &\doteq f(\varphi(t))-f(\varphi(0)) -\int_{[0,t] \times \overline{B}_{\textcolor{black}{K}}(0)} x \rho_s(dx)\textcolor{black}{\nabla f(\varphi(s))}\,ds\\
								& - \int_{0}^{t}\left(b(\varphi(s), p(s, \varphi(s)) \nabla f(\varphi(s)) + \frac{1}{2}\Delta f(\varphi(s))\right)\,ds,
\end{split}
\end{equation}
where $t \in [0, T]$. \textcolor{black}{We claim that $\Theta^{*}\doteq \mathbb{P} \circ (X^{\tilde{\beta}^{*}}, \tilde{\beta}^{*})^{-1}\in \mathcal{P}(\mathcal{X} \times \mathcal{R}_K)$ is a solution of the martingale problem associated to Eq.\eqref{eq:process_martingale}, i.e. such that for all $f \in \text{C}_{c}^{2}(\mathbb{R}^{d})$, $M^f$ is a $\Theta^*$-martingale.} The martingale property is intended on $(\mathcal{X} \times \textcolor{black}{\mathcal{R}_{K}}, \mathcal{B}(\mathcal{X} \times \textcolor{black}{\mathcal{R}_{K}}))$ with respect to the $\Theta^{*}$-augmentation of the canonical filtration made right continuous by a standard procedure. However, to conclude it is sufficient to check that the
martingale property holds with respect to the canonical filtration on $\mathcal{X} \times \textcolor{black}{\mathcal{R}_{K}}$ \citep[see, for instance, Problem 5.4.13 in][]{karatzas1998blackian}.
We denote by $(\mathcal{G}_t)_{t \in [0,T]}$ such a filtration show that the process in Eq.~ \eqref{eq:process_martingale}, which is bounded, measurable and $\mathcal{G}_t$-adapted, is a $\Theta^{*} \doteq \mathbb{P} \circ (X^{\tilde{\beta}^{*}}, \tilde{\beta}^{*})^{-1}$ martingale for all $f \in \text{C}_{c}^{2}(\mathbb{R}^{d})$. This is equivalent to having 
$$
\mathbb{E}^{\Theta^{*}}\left[Y\cdot (M^{f}_{t_2}-M^{f}_{t_1})\right] = 0
$$
for every choice of $(t_1, t_2, Y) \in [0,T]^{2} \times \text{C}_b(\mathcal{X} \times \textcolor{black}{\mathcal{R}_{K}})$ such that $t_1 \leq t_2$ and $Y$ is $\mathcal{G}_{t_1}$-measurable. To this aim, we define and compute the following function $\Psi^{p} = \Psi^{p}_{(t_1, t_2, Y, f)}:\mathcal{P}(\mathcal{X} \times \textcolor{black}{\mathcal{R}_{K}})\rightarrow\mathbb{R}$:
\begin{equation}\label{eq:second_revisione}
\begin{split}
\Psi^{p}(\Theta^{*}) &= \Psi^{p}_{(t_1, t_2, Y, f)}(\Theta^{*}) \doteq \mathbb{E}^{\Theta^{*}}\left[Y\cdot (M^{f}_{t_2}-M^{f}_{t_1})\right]\\
&= \int_{\mathcal{X} \times \mathcal{R}_C}Y(\varphi,\rho)\left(f(\varphi(t_2))-f(\varphi(t_1)) \right)\,\Theta^*(d\varphi,d\rho)\\
&-\int_{\mathcal{X} \times \mathcal{R}_C}Y(\varphi,\rho)\int_{\bar B_C\times[t_1,t_2]}x\rho_t(dx)\nabla f(\varphi(t)) dt\,\Theta^*(d\varphi,d\rho)\\
&-\int_{\mathcal{X} \times \mathcal{R}_C}Y(\varphi,\rho)\int_{t_1}^{t_2} b(\varphi(t),p(t,\varphi(t)) \nabla f(\varphi(t))dt\,\Theta^*(d\varphi,d\rho)\\
&-\frac{1}{2} \int_{\mathcal{X} \times \mathcal{R}_C} Y(\varphi,\rho)\int_{t_1}^{t_2}\Delta f(\varphi(t))dt\,\Theta^*(d\varphi,d\rho).
\end{split}
\end{equation}
The previous function, in particular, is continuous with respect to the weak convergence of measure since the integrands are bounded and continuous on $\mathcal{X} \times \textcolor{black}{\mathcal{R}_{K}}$. Also, we define:
\begin{equation*}
\begin{split}
&\overline{M}^{f,i}_t(\varphi^N,\rho^N) \doteq  f(\varphi^{N,i}(t))-f(\varphi^{N,i}(0))\\
&-\int_0^t\left[\int_{\bar B_C}x\rho^{N,i}_s(dx)+b(\varphi^{N,i}(s), v(\varphi^N(s)))\right]\nabla f(\varphi^{N,i}(s))+\frac{1}{2}\Delta f(\varphi^{N,i}(s)) ds\\
& v(\varphi^N(t)) \doteq  \frac{1}{N}\sum_{j=1}^NV^N(\varphi^{N,i}(t)-\varphi^{N,j}(t)),\quad t\in[0,T],\quad \textcolor{black}{i\in [[N]]},
\end{split}
\end{equation*}
\textcolor{black}{for $(\varphi^{N}, \rho^{N}) \in \mathcal{X}^{\times N} \times \mathcal{R}_{K}^{\times N}$, where $\rho^{N,i}$ and $\varphi^{N,i}$ are respectively the $i^{th}$ component of $\rho$ and $\varphi$, and} the extended empirical measure $\overline{S}^{N;\beta}$ as
\begin{equation*}
\overline{S}^{N;\beta} \doteq \frac{1}{N} \sum_{i = 1}^{N} \delta_{(X^{N,i;\beta}, \rho^{N, i; \beta})}.
\end{equation*}
\textcolor{black}{Here, $X^{N, i; \beta}$ denotes the dynamics of player $i$ in the system where the first player only deviates from the Nash equilibrium written in terms of relaxed controls $\rho^{N, i;\beta}$.}

\noindent Now, \textcolor{black}{by construction,} it holds that
\begin{equation}\label{eq:sum_N}
\frac{1}{N}\sum_{i=1}^{N} \mathbb{E}^{\overline{\Theta}^*_N}\left[\overline{Y}^{i}\cdot\left(\overline{M}^{f,i}_{t_2}-\overline{M}_{t_1}^{f,i} \right) \right] = 0,
\end{equation}
where $\overline{\Theta}^*_N \doteq \mathbb{P}^{N} \circ (X^{N, i; \beta}, \rho^{N,i;\beta})^{-1}$ \textcolor{black}{and for every choice of $(t_1, t_2, \overline{Y}^i) \in [0,T]^{2} \times \text{C}_b(\mathcal{X}^{\times N} \times \textcolor{black}{\mathcal{R}_{K}^{\times N}})$ such that $t_1 \leq t_2$ and $Y$ is $\mathcal{G}^N_{t_1}$-measurable, with $(\mathcal{G}^N_t)$ being the canonical filtration on $\mathcal{B}(\mathcal{X}^{\times N} \times \textcolor{black}{\mathcal{R}_{K}^{\times N}})$}. \textcolor{black}{ To conclude, it then suffices to show} that the previous term converges to the expected value of $\Psi^{p}(\Theta^{*})$ in the limit for $N \rightarrow \infty$. Let us set the sequence $\overline{Y}^{i}$ as $\overline{Y}^{i}(\varphi^{N}) \doteq Y(\varphi^{N,i})$ and show that the following decomposition for the term in Eq.~\eqref{eq:sum_N} holds:
\begin{equation}\label{eq:sum_N_decomposed}
\frac{1}{N}\sum_{i=1}^{N} \mathbb{E}^{\overline{\Theta}^*_N}\left[Y\cdot\left(\overline{M}^{f,i}_{t_2}-\overline{M}_{t_1}^{f,i} \right) \right] =  \mathbb{E}\left[\Psi_{(t_1, t_2, Y, f)}(\overline{S}^{N;\beta})\right] - \Delta^{p}_{(t_1, t_2, Y, f)}(\overline{S}^{N;\beta}).
\end{equation}
Indeed, the first term is equal to: 
\begin{equation}\label{eq:sum_N_decomposed_first_term}
\begin{split}
\mathbb{E}\left[\Psi_{(t_1, t_2, Y, f)}(\overline{S}^{N;\beta})\right]& = \frac{1}{N}\sum_{i=1}^N\mathbb{E} \left[Y(X^{N,i;\beta},\rho^{N,i;\beta})\left(f(X^{N,i;\beta}_{t_2})-f(X^{N,i;\beta}_{t_1}) \right)\right]\\
&-\frac{1}{N}\sum_{i=1}^N\mathbb{E}\left[Y(X^{N,i;\beta},\rho^{N,i;\beta})\int_{ \overline{B}_{\textcolor{black}{K}}(0) \times[t_1,t_2]}x\rho^{N,i;\beta}_t(dx)\nabla f(X^{N,i;\beta}_t) dt\right]\\
&-\frac{1}{N}\sum_{i=1}^N\mathbb{E}\left[Y(X^{N,i;\beta},\rho^{N,i;\beta})\int_{t_1}^{t_2} b(X^{N,i;\beta}_t,p(t,X^{N,i;\beta}_t)) \nabla f(X^{N,i;\beta}_t)dt \right]\\
&-\frac{1}{N}\sum_{i=1}^N\mathbb{E}\left[Y(X^{N,i;\beta},\rho^{N,i;\beta})\int_{t_1}^{t_2}\frac{1}{2}\Delta f(X^{N,i;\beta}_t)dt \right],\\
\end{split}
\end{equation}
whereas the second reads as:
\begin{equation}\label{eq:sum_N_decomposed_second_term}
\begin{split}
 \Delta^{p}_{(t_1, t_2, Y, f)}(\overline{S}^{N;\beta})
&= \frac{1}{N}\sum_{i=1}^N\mathbb{E}\left[Y(X^{N,i;\beta},\rho^{N,i;\beta})\int_{t_1}^{t_2} b(X^{N,i;\beta}_t,v(X^{N;\beta}_t) ) \nabla f(X^{N,i;\beta}_t)dt \right]\\
&- \frac{1}{N}\sum_{i=1}^N\mathbb{E}\left[Y(X^{N,i;\beta},\rho^{N,i;\beta})\int_{t_1}^{t_2} b(X^{N,i;\beta}_t,p(t,X^{N,i;\beta}_t)) \nabla f(X^{N,i;\beta}_t)dt \right]
\end{split}
\end{equation}
\textcolor{black}{In particular, $\Psi_{(t_1, t_2, Y, f)}(\overline{S}^{N; \beta})$ corresponds to the integrals in Eq.~\eqref{eq:second_revisione} computed w.r.t. the extended empirical measure $\overline{S}^{N; \beta}$.} The term in Eq.~\eqref{eq:sum_N_decomposed_first_term} converges to $\Psi^{p}_{(t_1, t_2, Y, f)}(p)$  in the limit for $N \rightarrow \infty$ thanks to the weak continuity of the involved functional and weak convergence of measures. Term in Eq.~\eqref{eq:sum_N_decomposed_second_term}, instead, vanishes in the limit as $N \rightarrow \infty$ thanks to Lemma \ref{lem:upgrade}, since it can be bounded by: :
\begin{small}
\begin{equation}
\begin{split}
& \left|  \Delta^{p}_{(t_1, t_2, Y, f)}(\overline{S}^{N;\beta}) \right|\\
&\leq  \frac{1}{N}\sum_{i=1}^N\mathbb{E}\left[ \left|Y(X^{N,i;\beta})\right|\int_{t_1}^{t_2}\left|b(X_s^{N,i;\beta},p^N(s,X^{N,i;\beta}_s))-b(X_s^{N,i;\beta},p(s,X^{N,i;\beta}_s) )\right|\left|\nabla f(X^{N,i;\beta}_s)\right|ds \right]\\
&\leq  \frac{1}{N}\sum_{i=1}^N \|Y\|_{\infty}\|\nabla f\|_{\infty}L\int_{t_1}^{t_2} \|p^N(s,\cdot)-p(s,\cdot)\|_{\infty}ds.
\end{split}
\end{equation}
\end{small}

\noindent\textcolor{black}{We conclude that $\Theta^{*}\in \mathcal{P}(\mathcal{X} \times \mathcal{R}_K)$ solves the martingale problem associated to Eq.\eqref{eq:process_martingale}. By an argument analogous to that in the proofs of Proposition 5.4.6 and Corollary 5.4.8 in \cite{karatzas1998blackian}, we finally conclude that there exists a weak solution $((\Omega^{\beta^*},\mathcal{F}^{\beta^*},\mathbb{Q}^{\beta^*}), X^{\tilde{\beta}^{*}},W^{\beta^*})$ of Eq.\eqref{eq:limit_equation}.}\\
\indent\textbf{Step 3.} For every $N \in \mathbb{N}, $
\begin{equation}
\begin{split}
&J^{N}_{i}(\bm{\alpha}^N)-\inf_{\beta} J^{N}_{i}([\bm{\alpha}^{N,-i},\beta]) \\
& \leq J^{N}_{i}(\bm{\alpha}^N) -J(\alpha^{*})+J(\alpha^{*})- J^{N}_{i}([\bm{\alpha}^{N,-i},\beta^N_1]) + \varepsilon/2.
\end{split}
\end{equation}
By \textbf{Step 1} and \textbf{Step 2} there exists $N_0(\varepsilon)$ such that
\begin{equation*}
J^{N}_i(\bm{\alpha}^N) -J(\alpha^{*})\leq \varepsilon/4 \qquad J(\alpha^{*})- J^{N}_{i}([\bm{\alpha}^{N,-i},\beta^N_1])\leq  \varepsilon/4.
\nonumber
\end{equation*}
\noindent for all $N\geq N_0(\varepsilon)$. This concludes the proof.
\end{proof}

\newpage
\appendix
\begin{center}
\textsc{APPENDIX}
\end{center}

\section{Some well known results}
\label{app:usefulLemmas}
For the reader convenience, we collect here some (well-known) results on convolutions, regularizations and mollifiers that have been used through the paper.\\
\indent First, we remind some properties on convolution and regularization.  
\begin{proposition}[Convolution and regularization]
\label{prop:brezisConv}\citep[Propositions 4.4.15, 4.4.19 and 4.4.20]{brezis2010functional}
The following statements on convolution hold true:
\begin{itemize}
\item[\textit{(i)}] Let $f\in L^1(\mathbb{R}^d)$ and $g\in L^p(\mathbb{R}^d)$, $1\leq
p\leq \infty$. Then $f\ast g$ is well defined in $L^p(\mathbb{R}^d)$.

\item[\textit{(ii)}] Let $\theta\in \text{C}_c(\mathbb{R}^d)$ and $\varphi \in L^1_{loc}(\mathbb{R%
}^d)$. Then $\Theta*\varphi$ is well defined in $\text{C}(\mathbb{R}^d)$.

\item[\textit{(iii)}] Let $\theta\in \text{C}_c^k(\mathbb{R}^d)$ and $\varphi\in L^1_{loc}(
\mathbb{R}^d)$. Then $\Theta*\varphi$ is well defined in $\text{C}^k(\mathbb{R}^d)$, $k\geq 1$, also $k=\infty$.
\end{itemize}
\end{proposition}
\noindent In particular, in our work we used convolution of the type $\theta \ast \mu$, where $\theta\in \text{C}_c^{\infty}(\mathbb{R}^d)$ and $\mu\in\mathcal{P}(\mathbb{R}^d)$. Therefore, since $\mu\in L^1(\mathbb{R}^d)$ and $\theta\in L^p(\mathbb{R}^d)$ for
any $1\leq p\leq \infty$, by item \textit{(i)} of Proposition \ref{prop:brezisConv} the convolution $\theta*\mu$ is well defined in $L^p(\mathbb{R}^d)$. Moreover, by  items \textit{(ii)} and \textit{(iii)} of Proposition \ref{prop:brezisConv}, $%
\theta*\mu\in \text{C}^k(\mathbb{R}^d)$ for any $k\geq 1$, also $k=\infty$. Also, we use scalar product of the type $\langle\theta*\mu,\varphi \rangle$, where $\varphi\in L^2(\mathbb{R}^d)$.  In particular, for any function $g:\mathbb{R}^d\rightarrow\mathbb{R}$ if we denote $g^-\doteq g(-\cdot)$, then
\begin{eqnarray}
\langle \theta*\mu,\varphi \rangle&=&\int_{\mathbb{R}^d}\int_{\mathbb{R}%
^d}\theta(x-y)\mu(dy)\varphi(x)dx  \notag \\
&=&\int_{\mathbb{R}^d}\int_{\mathbb{R}^d}\theta(x-y)\varphi(x)dx\mu(dy)  \notag
\\
&=&\int_{\mathbb{R}^d}\theta(-\cdot)*\varphi(y)\mu(dy)=\langle
\mu,\theta^-*\varphi \rangle.  \notag
\end{eqnarray}
\indent Second, we give the following definition and proposition.

\begin{definition}[Mollifiers]\label{prop:brezisConv2}
\citep[Chapter 4.4]{brezis2010functional} A sequence of mollifiers is any sequence of
functions $(\theta_N)_{N\in\mathbb{N}}$ from $\mathbb{R}^d$ to $\mathbb{R}$
such that for each $N \in\mathbb{N}$: $\theta_N\in \text{C}^{\infty}_c(\mathbb{R}^d)$ with support in $\overline{B}_{1/N}(0)$, $\theta_N \geq 0$
and $\int_{\mathbb{R}^d}\theta^N(dx)=1$.
\end{definition}

\begin{proposition}[Mollification]
\citep[Proposition 4.4.21]{brezis2010functional} Let $f\in \text{C}(\mathbb{R}^d)$. Then $\theta_N*f\rightarrow f$ uniformly on compact sets.
\end{proposition}

\indent Third, we give the following results on weak convergence.
\begin{lemma}[Weak convergence and the double index problem]
\label{lem:doubleInd} Let $(\mu_N)_{N \in\mathbb{N}} \subset \mathcal{P}(\mathbb{R}^{d})$ a sequence converging weakly to $\mu \in \mathcal{P}(\mathbb{R}^{d})$. Let $(f_N)_{N \in\mathbb{N}} \in \text{C}_b(\mathbb{R}^d)$ be a sequence converging to $f \in  \text{C}_b(\mathbb{R}^d)$ uniformly on compact sets and such that $\sup_{n\in\mathbb{N}}\|f_N\|_{\infty}\leq C<\infty$, $\|f\|_{\infty}\leq C<\infty$ for some $C>0$. 
Then 
\begin{eqnarray}
\int_{\mathbb{R}^d}f_N(x)\mu_N(dx)\underset{N\rightarrow\infty}{
\longrightarrow} \int_{\mathbb{R}^d}f(x)\mu(dx).  \notag
\end{eqnarray}
\end{lemma}
\begin{proof}
The proof is based on the following decomposition, holding for any $R>0$:
\begin{eqnarray}
\int_{\mathbb{R}^d}f_N(x)\mu_N(dx)-\int_{\mathbb{R}^d}f(x)\mu(dx) &=& \int_{\overline{B}_R(0)}(f_N(x)-f(x))\mu_N(dx)\nonumber\\
&+&\int_{\overline{B}_R(0)}f(x)(\mu_N-\mu)(dx)\nonumber\\
&+&\int_{\mathbb{R}^d\setminus \overline{B}_R(0)}f_N(x)\mu_N(dx)-\int_{\mathbb{R}^d\setminus \overline{B}_R(0)}f(x)\mu(dx)
\nonumber
\end{eqnarray}
\noindent where $\overline{B}_R(0)\subset\mathbb{R}^d$ is the closed ball of radius $R$ centered at the origin. Hence
\begin{eqnarray}
\left|\int_{\mathbb{R}^d}f_N(x)\mu_N(dx)-\int_{\mathbb{R}^d}f(x)\mu(dx)\right| &\leq & \|f_N-f \|_{\infty,\overline{B}_R(0)}\nonumber\\
&+&\left|\int_{\overline{B}_R(0)}f(x)(\mu_N-\mu)(dx)\right|\nonumber\\
&+&C\left( \mu_N(\mathbb{R}^d\setminus \overline{B}_R(0))+\mu(\mathbb{R}^d\setminus \overline{B}_R(0))\right)
\nonumber
\end{eqnarray}
\nonumber
\noindent where $ \|\cdot \|_{\infty,\overline{B}_R(0)}$ is the infinity norm on $\overline{B}_R(0)$. Now let $\varepsilon>0$ and choose $R>0$ be such that
\begin{eqnarray}
\sup_{N\in\mathbb{N}}\mu_N(\mathbb{R}^d\setminus \overline{B}_R(0))<\frac{\varepsilon}{4C}\quad \text{and}\quad \mu(\mathbb{R}^d\setminus \overline{B}_R(0))<\frac{\varepsilon}{4C}
\nonumber
\end{eqnarray} 
\noindent by the tightness of the family $(\mu_N)_{N \in\mathbb{N}}$. Then, by uniform convergence on compact sets of the sequence $(f_N)_{N \in\mathbb{N}}$ to $f$ and by weak convergence of the $(\mu_N)_{N \in\mathbb{N}}$ to $\mu$ there exists $N_0\in\mathbb{N}$ such that the first and second terms are lower than $\frac{\varepsilon}{4}$ for all $N\geq N_0$. We conclude that for all $\varepsilon>0 $ there exists $N_0\in\mathbb{N}$ such that
\begin{eqnarray}
\left|\int_{\mathbb{R}^d}f_N(x)\mu_N(dx)-\int_{\mathbb{R}^d}f(x)\mu(dx)\right|< \varepsilon
\nonumber
\end{eqnarray}
\noindent  for all $N\geq N_0$.
\end{proof}

\begin{lemma}
\label{lem:L2convAndProb}
Let $(\mu_N)_{N \in\mathbb{N}} \subset \mathcal{P}(\mathbb{R}^{d})$ a sequence converging weakly to $\mu \in \mathcal{P}(\mathbb{R}^{d})$. Set $f_N \doteq \theta_N*\mu_N$ for some mollifiers $\theta_N$ and assume $ \lim_{N \rightarrow \infty} f_N = f$ in $L^2(\mathbb{%
R}^d)$ for some $f\in L^2(\mathbb{R}^d)$. Then $\mu$ has density $f$ with
respect to the Lebesgue measure on $\mathbb{R}^d$.
\end{lemma}
\begin{proof}
\noindent First, notice that $\langle f_N, \varphi \rangle =\langle \theta_N \ast \mu_N\varphi \rangle = \langle \mu_N\theta_N^{-}*\varphi \rangle$ for any $\varphi \in L^2(\R^d)\cap C(\R^d)$ and for each $N \in\N$. Set $\varphi_N \doteq \theta_N^-*\varphi$ for each $N \in \N$. Now \textcolor{black}{$\langle f_N, \varphi \rangle \rightarrow \langle f, \varphi \rangle$} for any $\varphi\in L^2(\R^d)$, by strong convergence in $L^2(\R^d)$ of the $f^N$, but also
\begin{equation*}
\int_{\R^d}f_N(x)\varphi(x) dx= \int_{\R^d}\varphi_N^-(x) \mu_N(dx) \rightarrow \int_{\R^d}\varphi(x) \mu(dx) 
\end{equation*}
\noindent by weak convergence of the $\mu^N$ and uniform convergence on compact sets of the $\phi_N$ to $\phi$ (Lemma \ref{lem:doubleInd}). Hence
$$
\int_{\R^d}\varphi(x) \mu(dx)=\int_{\R^d}\varphi(x)f(x)dx
$$
\noindent for any $\varphi\in L^2(\R^d)\cap \text{C}(\R^d)$. The same reasoning holds for any $\varphi\in \text{C}_b(\R^d)$ hence we conclude.
\end{proof}

\section{Hamilton-Jacobi Equation, Kolmogorov equation Equations and Mild Solutions}\label{app:hjandfpmild}
\noindent In Subsection \ref{subsec:HJ_and_FP_mild} we study the decoupled Hamilton-Jacobi Bellman equation and Kolmogorov equation equations defining the PDE system in Eq.~\eqref{eq:pdesystem} via the mild formulation; see Theorem \ref{teo:uniqueness_density}, Theorem \ref{teo:uniqueness_value_function}. This enables us to prove the equivalence between the mild and weak formulations; see proof of Lemma \ref{lem:equiv_mild} in Subsection \ref{subsec:equiv_mild}. \textcolor{black}{In Subsection \ref{subsec:global_existence} we prove Theorem \ref{th:globalExistence}, i.e. the existence of a global solution of the PDE system (see Theorem \ref{th:globalExistence} in Section \ref{sec:meanfieldgame}).} On the other hand, in Subsection \ref{subsec:local_existence_uniqueness} we prove Theorem \ref{th:localExistenceUniqueness}, i.e. the local uniqueness of a solution of the  PDE system (see Theorem \ref{th:localExistenceUniqueness} in Section \ref{sec:meanfieldgame}).  Finally, in Subsection \ref{subsec:proofVerificationTheorem} we give the proof of Theorem \ref{th:VerificationTheorem}.

\subsection{The Hamilton-Jacobi and the Kolmogorov equation equation in mild form}\label{subsec:HJ_and_FP_mild}
\noindent Throughout this section, we assume that $p_0, b, f, g$ satisfy the hypotheses (H1)-(H2) and (H4) in Section \ref{sec:preliminaries}. 
\begin{theorem}
\label{teo:uniqueness_density} Given $p_{0}\in \text{C}_{b}\left(  \mathbb{R}%
^{d}\right)  $, given $\alpha\in \text{C}_{b}\left(  \left[  0,T\right]
\times\mathbb{R}^{d};\mathbb{R}^{d}\right)  $, there exists at most one
solution of equation
\begin{equation}
p\left(  t\right)  =\mathcal{P}_{t}p_{0}-\int_{0}^{t}\nabla\mathcal{P}%
_{t-s}\left(  p(s)\left(  \alpha(s)-b(\,\cdot\,,p(s))\right)  \right)
\,ds.\label{eq:equation_mild}%
\end{equation}
in the class $\text{C}_{b}\left(  \left[  0,T\right]  \times\mathbb{R}^{d}\right)  $.
\end{theorem}
\begin{proof}
Assume by contradiction that $p^{(i)}(t)$, $i=1, 2$, are two solutions of Eq.~ \eqref{eq:equation_mild} of class $\text{C}_b([0,T]\times\mathbb{R}^{d})$ and set $q(t)$ as their difference. By a  generalized form of Gronwall's lemma one has that $\left\Vert q\left(t\right)  \right\Vert _{\infty}=0$ for every $t\in\left[  0,T\right]$, from which the conclusion readily follows. The precise estimates can be found in the proof of  Theorem \ref{th:localExistenceUniqueness} in Subsection \eqref{subsec:local_existence_uniqueness}. For the sake of space, we refer the reader to that proof; in particular one has to use the estimate for the map $\Gamma_1$, first component of the map $\Gamma$ defined in \eqref{eq:Gamma_map}.
\end{proof}
\begin{theorem}
\label{teo:uniqueness_value_function} \noindent \textcolor{black}{Given $p \in \text{C}_{b}([0,T] \times \mathbb{R}^{d})$}, If
$\alpha\in \text{C}_{b}([0,T]\times\mathbb{R}^{d};\mathbb{R}^{d})$, then there exists
at most one solution $u$, in the class $\text{C}_{b}\left(  \left[  0,T\right]  \times\mathbb{R}%
^{d}\right)$ and such that its partial derivatives are also of class $\text{C}_{b}\left(  \left[  0,T\right]  \times\mathbb{R}^{d}\right)$, of the following equation
\begin{equation}
u\left(  t\right)  =\mathcal{P}_{T-t}g-\int_{t}^{T}\mathcal{P}_{s-t}\left(
b\left(  \,\cdot\,,p\left(  s\right)  \right)  \cdot\alpha\left(  s\right)
-\frac{1}{2}\left\vert \nabla u\left(  s\right)  \right\vert ^{2}%
+f(\,\cdot\,,p(s))\right)  ds\label{eq:equation_mild_2}
\end{equation}
\end{theorem}
\begin{proof}
Assume by contradiction that $u^{(i)}(t)$, $i = 1, 2$, are two solutions of Eq.~\eqref{eq:equation_mild_2} of class $\text{C}_{b}\left(  \left[  0,T\right]  \times\mathbb{R}%
^{d}\right)$ and such that their partial derivatives are of class $\text{C}_{b}\left(  \left[  0,T\right]  \times\mathbb{R}^{d}\right)$. Set $\theta^{(i)} = \nabla u^{(i)}$, $i = 1, 2$ and $q(t)$ their difference. Using the estimates for the map $\Gamma_2$, second component of the map $\Gamma$ defined in \eqref{eq:Gamma_map}, one has that $\| q(t) \|_{\infty} = 0$ for every $t \in [0, T]$, from which $\theta^{(1)}(t) = \theta^{(2)}(t)$ for every $t \in [0, T]$. Therefore, $u^{(1)}(t) = u^{(2)}(t)$ because of Eq.~\eqref{eq:equation_mild_2}.
\end{proof}
\subsection{Proof of Lemma \protect\ref{lem:equiv_mild}}\label{subsec:equiv_mild}
\begin{proof}
Let $(u, p)$ be a weak solution of the PDE system in Eqs.~\eqref{eq:mfgsystem_eq1}--\eqref{eq:mfgsystem_eq2}, and consider Eq.~\eqref{eq:mfgsystem_eq1}. In particular, for a given $t \in \left[0, T \right]$, 
\begin{equation}
\begin{split}
\left\langle u\left( t\right) ,\varphi \left( t\right) \right\rangle
&-\left\langle g,\varphi\left( T\right) \right\rangle +\int_{t}^{T}\left\langle
u\left( s\right) , \mathcal{A}\varphi(s)  \right\rangle ds \\
&= \int_{t}^{T}\left\langle b(\,\cdot\,,p(s))\cdot \nabla u\left( s\right) -\frac{1}{%
2}\left\vert \nabla u\left( s\right) \right\vert ^{2}+f(\,\cdot\,,p(s)),\varphi \left(
s\right) \right\rangle ds. \\
\end{split}
\end{equation}
Using on $\left[ t,T\right] $ the following test function
\begin{equation*}
\varphi^{\left( t\right) }\left( s,x\right) = \left( \mathcal{P}_{s-t}\phi \right) \left( x\right) ,\qquad s\in \left[ t,T\right],
\end{equation*}
with $\phi \in \text{C}^{1}([0,T] \times \text{C}^2_{b}(\mathbb{R}^{d}) \cap W^{2,2}(\mathbb{R}^{d}))$, we get
\begin{equation*}
\begin{split}
&\left\langle u\left( t\right) ,\phi \right\rangle -\left\langle g,\mathcal{%
P}_{T-t}\phi \right\rangle +\int_{t}^{T}\left\langle u\left( s\right)
,\mathcal{A}\mathcal{P}_{s-t}\phi
\right\rangle ds \\
&=\int_{t}^{T}\left\langle b(\,\cdot\,,p(s))\cdot \nabla u\left( s\right) -\frac{1}{
2}\left\vert \nabla u\left( s\right) \right\vert ^{2}+f(\,\cdot\,,p(s)),\mathcal{P}
_{s-t}\phi \right\rangle ds.\\
\end{split}
\end{equation*}
\noindent Notice that $\mathcal{A} \mathcal{P}
_{s-t}\phi =0$ and that $\left\langle a,\mathcal{P}_{t}b\right\rangle
=\left\langle \mathcal{P}_{t}a,b\right\rangle $ for every pair of functions $
a,b\in \text{C}_{b}\left( \mathbb{R}^{d}\right) $. Then
\begin{equation*}
\begin{split}
&\left\langle u\left( t\right) ,\phi \right\rangle -\left\langle \mathcal{P}
_{T-t}g,\phi \right\rangle \\
&\int_{t}^{T}\left\langle \mathcal{P}_{s-t}\left( b(\,\cdot\,,p(s))\cdot \nabla
u\left( s\right) -\frac{1}{2}\left\vert \nabla u\left( s\right) \right\vert
^{2}+f(\,\cdot\,,p(s))\right) ,\phi \right\rangle ds.\\
\end{split}
\end{equation*}
Because $\phi$ can be chosen in an arbitrary way, we deduce the mild formulation of Eq.~\eqref{eq:mfgsystemmild_eq1}. The equation for $p$ is similar, as well as the other direction.
\end{proof}

\textcolor{black}{
\subsection{Proof of Theorem \ref{th:globalExistence}}
\label{subsec:global_existence}
Throughout this section, we assume that $p_0, b, f,$ and $g$ satisfy the hypotheses (H1)-(H2) and (H4) in Section \ref{sec:preliminaries} and (H5) in Section \ref{sec:meanfieldgame}. In addition, we shall repeatedly use the following well-known inequality:
\begin{equation}\label{gradient estimate infinity norm}
\left\Vert \nabla\mathcal{P}_{t}f\right\Vert _{\infty}\leq C_{d}%
t^{-1/2}\left\Vert f\right\Vert _{\infty}
\end{equation}
for all $f\in L^{\infty}\left(  \mathbb{R}^{d}\right)  $, with $C_{d}=d^{1/2}$, which follows for instance from the formula $\nabla\mathcal{P}_{t}f\left(
x\right)  =t^{-1}\mathbb{E}\left[ W_{t}f\left(  x+W_{t}\right)  \right]  $
(elementary proved by differentiating the heat kernel):
\[
\left\vert \nabla\mathcal{P}_{t}f\left(  x\right)  \right\vert \leq
t^{-1}\left\Vert f\right\Vert _{\infty}\mathbb{E}\left[  \left\vert
W_{t}\right\vert \right]  \leq t^{-1}\left\Vert f\right\Vert _{\infty
}\mathbb{E}\left[  \left\vert W_{t}\right\vert ^{2}\right]  ^{1/2}\leq
C_{d}t^{-1/2}\left\Vert f\right\Vert _{\infty}.
\]
We use the Brouwer-Schauder fixed point theorem to prove Theorem \ref{th:globalExistence}. Brouwer-Schauder fixed point theorem says that if $K$ is a non empty, closed, bounded  and convex subset of a Banach space $V$ and $\Phi\,:\,K \rightarrow K$ is a continuous map such that $\Phi \left(K\right)$ is relatively compact in $V$, then $\Phi$ has a fixed point in $K$.\\
\indent We will apply this theorem to the space $V=\text{C}_{b}(\left[ 0,T\right] \times \mathbb{R}^{d})$. Instead, in order to define the map $\Phi$, let $p \in V$ be given and let $w = w_{p}$ be a weak solution of the first equation of the PDE system \eqref{eq:pdesystem_hopf_cole}. Existence and uniqueness of such a solution is given by classical parabolic results; e.g., one proof can be done by contraction principle applied to the mild formulation in Eq.~\eqref{eq:mild_formulation_app} below. In particular, $w_p$ satisfies the following properties:
\begin{equation*}
w_{p}\left( t,x\right) \geq e^{-\left( \left\Vert g\right\Vert _{\infty
}+T\left\Vert f\right\Vert _{\infty }\right)}\quad\text{and}\quad\left\Vert w_{p}\right\Vert _{\infty }+\left\Vert \nabla w_{p}\right\Vert
_{\infty }\leq C_{1}\left( b,f,g,T\right)
\end{equation*}
independently of $p\in V$, with $C_{1}\left( b,f,g,T\right) >0$ depending only on $\| b \|_{\infty}, \| f \|_{\infty}$ and $\| g \|_{\infty}$. One way to prove this fact is by using the following identity 
\begin{equation}\label{eq:mild_formulation_app}
w_{p}\left( t\right) =\mathcal{P}_{T-t}\exp \left( -g\right) -\int_{t}^{T}%
\mathcal{P}_{s-t}\left( b\left( \cdot ,p\left( s\right) \right) \cdot \nabla
w_{p}\left( s\right) -w_{p}\left( s\right) f\left( \cdot ,p\left( s\right)
\right) \right) ds  
\end{equation}
and estimate \ref{gradient estimate infinity norm} of the heat semi-group's gradient. At this point, we call $\Phi \left( p\right) $ the solution of the following equation 
\begin{equation} \label{eq for Phi p}
\Phi \left( p\right) \left( t\right) =\mathcal{P}_{t}p_{0}+\int_{0}^{t}%
\nabla \mathcal{P}_{t-s}\left( \Phi \left( p\right) \left( s\right) \left( 
\frac{\nabla w_{p}\left( s\right) }{w_{p}\left( s\right) }+b\left( \cdot
,p\left( s\right) \right) \right) \right) ds. 
\end{equation}
\noindent Notice that this is not the second equation of the PDE system \eqref{eq:pdesystem_hopf_cole} with $w=w_{p}$ because we keep the original $p$ in $b\left( \cdot ,p\left(s\right) \right)$. Existence of a global solution $\Phi \left( p\right) \in
V$ can be proved by iteration, using \ref{gradient estimate infinity norm} and 
$\left\Vert \frac{\nabla w_{p}}{w_{p}}\right\Vert _{\infty }\leq C_{w}\left( g,f,b,T\right)$. In addition, one gets 
\begin{equation*}
\left\Vert \Phi \left( p\right) \right\Vert _{\infty }\leq C_{2}\left(
b,f,g,p_{0},T\right) 
\end{equation*}
for a suitable constant $C_{2}\left( b,f,g,p_{0},T\right) >0$ depending, again, only on $\| b \|_{\infty}, \| f \|_{\infty}$ and $\| g \|_{\infty}$. Therefore, the set
\begin{equation*}
K \doteq \left\{ p\in V:\left\Vert p\right\Vert _{\infty }\leq C_{2}\left(
b,f,g,p_{0},T\right) \right\} 
\end{equation*}
is bounded, closed, convex and invariant.\\
\noindent We prove now that the map $\Phi$ satisfies the assumptions in the Brouwer-Schauder fixed point theorem. It is not difficult to prove that the map $\Phi$ is continuous by using \ref{gradient estimate infinity norm} again. Instead, it is non straightforward to prove that $\Phi\left(K\right)$ is relatively compact, due to the unboundedness of the space domain. In order to do so, we use the following compactness result, which is an easy variant of the Ascoli-Arzel\`{a} theorem.
\begin{theorem}\label{th:AscoliArzela_1}
Let $\alpha(\cdot), \beta(\cdot)$ and $C_{1}(\cdot), C_{2}(\cdot)$ be four positive and non-decreasing functions and
$\rho $ as in (H5); see Section \ref{sec:meanfieldgame}. Let $C_{3}>0$ a constant. Then the set $\Xi
_{\alpha ,C_{1},\beta ,C_{2},\rho ,C_{3}}$ of all functions $f\in
\text{C}_{b}\left( \left[ 0,T\right] \times \mathbb{R}^{d}\right) $ such that 
\begin{eqnarray*}
\sup_{t\in \left[ 0,T\right] }\left\Vert f\left( t\right) \right\Vert
_{\alpha ,R} &\leq &C_{1}\left( R\right) \text{ for every }R>0\quad\quad\quad\quad\quad\quad\,\,\,\,\text{(H1.1)}\\
\sup_{\substack{ t,s\in \left[ 0,T\right]  \\ t\neq s}}\frac{\left\Vert
f\left( t\right) -f\left( s\right) \right\Vert _{\infty ,R}}{\left\vert
t-s\right\vert ^{\beta }} &\leq &C_{2}\left( R\right) \text{ for every }R>0
\quad\quad\quad\quad\quad\quad\,\,\,\,\text{(H2.1)}\\
\left\vert f\left( t,x\right) \right\vert  &\leq &C_{3}\rho \left( x\right) 
\text{ for all }\left( t,x\right) \in \left[ 0,T\right] \times \mathbb{R}^{d}\quad\quad\text{(H3.1)}\\
\end{eqnarray*}%
is relatively compact in $\text{C}_{b}\left( \left[ 0,T\right] \times \mathbb{R}%
^{d}\right) $.
\end{theorem}
\noindent Before proceeding with the proof of Theorem \ref{th:AscoliArzela_1}, we recall the following version of the Ascoli-Arzel\`{a} theorem.
\begin{theorem}\label{th:AscoliArzela_2}
Assume that that a family of functions $F\subset \text{C}\left( \left[ 0,T\right] ;  \text{C}_{b}\left( B_{M}\right) \right) $ satisfies the following two properties:
\begin{itemize}
\item[\textit{(i)}] $\left\{ f\left( t\right) ;f\in F,t\in \left[ 0,T\right] \right\} \subset K_{M}$ for some compact set $K_{M} \subset \text{C}_{b}\left( B_{M}\right) $
\item[\textit{(ii)}] $F$ is uniformly equicontinuous in $\text{C} \left( \left[ 0,T\right] ; \text{C}_{b}\left( B_{M}\right) \right) $, namely for every $\epsilon >0$ there exists a $\delta >0$ such that $\left\Vert f\left( t\right) -f\left( s\right) \right\Vert _{\text{C}_{b}\left( B_{M}\right) }\leq \epsilon $ for every $f\in F$ and $t, s\in \left[ 0,T\right] $ such that $\left\vert t-s\right\vert \leq
\delta $.
\end{itemize}
Then $F$ is relatively compact in $\text{C}\left( \left[ 0,T\right] ; \text{C}_{b}\left(B_{M}\right) \right) $.
\end{theorem}
\begin{proof}[Proof of Theorem \ref{th:AscoliArzela_1}]
Notice that given any closed ball $B_M \doteq \overline{B}_{M}(0) \subset \mathbb{R}^{d}$ of radius  $M$ around the origin, the space $\text{C}_{b}\left( 
\left[ 0,T\right] \times B_{M}\right) $ and the space $\text{C}\left( \left[ 0,T\right]
;\text{C}_{b}\left( B_{M}\right) \right) $ are equivalent. This is not longer true for $\text{C}_{b}\left( \left[ 0,T\right] \times \mathbb{R}^{d}\right) $ and $\text{C}\left( \left[ 0,T\right] ; \text{C}_{b}\left( \mathbb{R}^{d}\right) \right)$. Indeed, it holds that $\text{C}\left( \left[ 0,T\right] ; \text{C}_{b}\left( \mathbb{R}^{d}\right) \right) \subset \text{C}_{b}\left( \left[ 0,T\right] \times \mathbb{R}^{d}\right)$. On any $B_{M}$ we use Theorem \ref{th:AscoliArzela_2}. Now, consider a sequence $\left( p_{n}\right) _{n\in \mathbb{N}}\subset \Xi
_{\alpha ,C_{1},\beta ,C_{2},\rho ,C_{3}}$. For every $B_{M}$, denote by $
p_{n}^{M}$ the restriction of $p_{n}$ to $\left[ 0,T\right] \times B_{M}$.
They belong to $\text{C}_{b}\left( \left[ 0,T\right] \times B_{M}\right) $ which is equivalent to $\text{C}\left( \left[ 0,T\right] ; \text{C}_{b}\left( B_{M}\right) \right)$. The space $\text{C}_{b}^{\alpha}\left(B_{M}\right) $ has compact embedding into $\text{C}_{b}\left( B_{M}\right) $ by Ascoli-Arzel\`{a} theorem. By $(H1.1)$ in Theorem \ref{th:AscoliArzela_1}, the set $\left\{ p_{n}^{M}\left( t\right) ,n\in \mathbb{N},t\in \left[ 0,T\right]\right\}$ is bounded in $\text{C}_{b}^{\alpha }\left( B_{M}\right)$, hence assumption \textit{(i)} of Theorem \ref{th:AscoliArzela_2} is satisfied. On the other hand, by $(H2.1)$  in Theorem \ref{th:AscoliArzela_1} the sequence $\left(p_{n}^{M}\right) _{n\in \mathbb{N}}$ is uniformly equicontinuous in $\text{C}\left( \left[ 0,T\right] ; \text{C}_{b}\left( B_{M}\right) \right)$. Hence, by 
Theorem \ref{th:AscoliArzela_2} we may extract a subsequence which converges in $\text{C}\left( \left[ 0,T\right] ; \text{C}_{b}\left( B_{M}\right)
\right) $. By a diagonal argument, we can find a function $p\in \text{C}\left( \left[ 0,T\right] \times \mathbb{R}^{d}\right) $ and a subsequence $\left( p_{n_{k}}\right) $ such that $\left\Vert (p_{n}^{M}-p)_{|_{\left[ 0,T\right]
\times B_{M}}} \right\Vert _{\infty }\rightarrow 0$ as $k\rightarrow \infty $,
for every $M$. Given $\epsilon >0$, let $M_{\epsilon }$ be such that 
\begin{equation*}
\left\Vert \rho_{|_{B_{M}^{c}}}\right\Vert _{\text{C}_{b}\left( B_{M}^{c}\right)
}\leq \frac{\epsilon }{4C_{3}}.
\end{equation*}%
Since $\left\vert p_{n_{k}}\left( t,x\right) \right\vert \leq C_{3}\rho
\left( x\right) $, we also have
\begin{equation*}
\left\Vert p_{n_{k}}|_{\left[ 0,T\right] \times B_{M}^{c}}\right\Vert
_{\text{C}_{b}\left( \left[ 0,T\right] \times B_{M}^{c}\right) }\leq \frac{\epsilon 
}{4}.
\end{equation*}%
In addition, since $p_{n_{k}}\rightarrow p$ point-wise, we also have $\left\vert p\left( t,x\right) \right\vert \leq C_{3}\rho \left( x\right)$
and thus 
\begin{equation*}
\left\Vert p_{|\left[ 0,T\right] \times B_{M}^{c}}\right\Vert _{\text{C}_{b}\left( 
\left[ 0,T\right] \times B_{M}^{c}\right) }\leq \frac{\epsilon }{4}.
\end{equation*}%
Then $p\in \text{C}_{b}\left( \left[ 0,T\right] \times \mathbb{R}^{d}\right) $ and $\left\Vert p_{n_{k}}-p\right\Vert _{\infty } \leq \epsilon$. Now, if corresponding to $M_{\epsilon }$, we choose $k_{0}$ such that for all $%
k\geq k_{0}$ we have 
\begin{equation*}
\left\Vert (p_{n_{k}}^{N}-p)_{|_{\left[ 0,T\right] \times B_{N}}} \right\Vert
_{\text{C}_{b}\left( \left[ 0,T\right] \times B_{N}\right) }\leq \frac{\epsilon }{2}.
\end{equation*}
Whence, we have proved uniform convergence on the full space $\mathbb{R}^{d}$.
\end{proof}
\noindent The following proposition allows us to conclude the proof of Theorem \ref{th:globalExistence}.
\begin{proposition}
There exist four positive and non decreasing functions  $\alpha(\,\cdot\,) ,C_{1}(\,\cdot\,),$ $\beta(\,\cdot\,), C_{2}(\,\cdot\,)$, $\rho$ as in (H5) of Section \ref{sec:meanfieldgame} and a constant $C_{3} > 0$ such that $\Phi \left( K\right) \subset \Xi _{\alpha ,C_{1},\beta ,C_{2},\rho ,C_{3}}$.
\end{proposition}
\begin{proof}
Without loss of generality, we may assume $\alpha <\frac{1}{2}$. To shorten notations, set 
\begin{equation*}
h\left( s\right) \doteq \frac{\nabla w_{p}\left( s\right) }{w_{p}\left( s\right) }+b\left( \cdot ,p\left( s\right) \right).
\end{equation*}
Notice that the following inequalities hold
\begin{equation*}
\begin{split}
\left\Vert h\left( s\right) \right\Vert _{\infty } &\leq C_{w}\left(
g,f,b,T\right) +\left\Vert b\right\Vert _{\infty } \doteq C_{h}\left(
g,f,b,T\right)\\
\left\Vert \Phi \left( p\right) \left( s\right) h\left( s\right) \right\Vert
_{\infty } &\leq C_{2}\left( b,f,g,p_{0},T\right) C_{h}\left( g,f,b,T\right) \doteq C_{3}\left( b,f,g,p_{0},T\right) 
\end{split}
\end{equation*}
From Eq.~\eqref{eq for Phi p} we have
\begin{eqnarray*}
\left\Vert \Phi \left( p\right) \left( t\right) \right\Vert _{\alpha } &\leq
&C_{T,\alpha }\left\Vert p_{0}\right\Vert _{\alpha }+\int_{0}^{t}\frac{%
C_{T,\alpha }}{\left( t-s\right) ^{\frac{1}{2}+\alpha }}\left\Vert \Phi
\left( p\right) \left( s\right) h\left( s\right) \right\Vert _{\infty }ds \\
&\leq &C_{T,\alpha }\left\Vert p_{0}\right\Vert _{\alpha }+C_{T,\alpha }T^{%
\frac{1}{2}-\alpha }C_{3}\left( b,f,g,p_{0},T\right),
\end{eqnarray*}%
were we have used a gradient estimate in H\"older norm similar to those of Lemma C.3 below, but easier. Therefore, $(H1.1)$ in Theorem \ref{th:AscoliArzela_1} is satisfied, even uniformly with respect to $R$. Let us see $(H2.1)$ in Theorem \ref{th:AscoliArzela_1} with $t>t^{\prime }$:
\begin{equation*}
\left\Vert \Phi \left( p\right) \left( t\right) -\Phi \left( p\right) \left(
t^{\prime }\right) \right\Vert _{\infty ,R}\leq I_{1}+I_{2}+I_{3}
\end{equation*}%
\begin{equation*}
I_{1}:=\left\Vert \mathcal{P}_{t}p_{0}-\mathcal{P}_{t^{\prime
}}p_{0}\right\Vert _{\infty ,R}
\end{equation*}%
\begin{equation*}
I_{2}:=\int_{t^{\prime }}^{t}\left\Vert \nabla \mathcal{P}_{t-s}\left( \Phi
\left( p\right) \left( s\right) h\left( s\right) \right) \right\Vert
_{\infty ,R}ds
\end{equation*}%
\begin{equation*}
I_{3}:=\int_{0}^{t^{\prime }}\left\Vert \left( \nabla \mathcal{P}%
_{t-s}-\nabla \mathcal{P}_{t^{\prime }-s}\right) \left( \Phi \left( p\right)
\left( s\right) h\left( s\right) \right) \right\Vert _{\infty ,R}ds.
\end{equation*}%
We use the following property: for small $t$,%
\begin{equation*}
k\in \text{C}_{b}^{\alpha }\left( B_{N}\right) \Rightarrow \left\Vert \mathcal{P}%
_{t}k-k\right\Vert _{\infty }\leq C_{\alpha }t^{\alpha }\left\Vert
k\right\Vert _{\alpha }.
\end{equation*}%
Hence%
\begin{eqnarray*}
I_{1} &=&\left\Vert \mathcal{P}_{t-t^{\prime }}\mathcal{P}_{t^{\prime
}}p_{0}-\mathcal{P}_{t^{\prime }}p_{0}\right\Vert _{\infty } \\
&\leq &C_{\alpha }\left( t-t^{\prime }\right) ^{\alpha }\left\Vert \mathcal{P%
}_{t^{\prime }}h\right\Vert _{\alpha }\leq C\left( t-t^{\prime }\right)
^{\alpha }\left\Vert h\right\Vert _{\alpha }
\end{eqnarray*}%
\begin{equation*}
I_{2}\leq \int_{t^{\prime }}^{t}\frac{C_{T}}{\left( t-s\right) ^{\frac{1}{2}}%
}\left\Vert \Phi \left( p\right) \left( s\right) h\left( s\right)
\right\Vert _{\infty }ds\leq 2C_{T}\sqrt{t-t^{\prime }}C_{3}\left(
g,f,b,p_{0},T\right) 
\end{equation*}
\begin{eqnarray*}
I_{3} &=&\int_{0}^{t^{\prime }}\left\Vert \left( \mathcal{P}_{t-t^{\prime
}}-Id\right) \nabla \mathcal{P}_{t^{\prime }-s}\left( \Phi \left( p\right)
\left( s\right) h\left( s\right) \right) \right\Vert _{\infty }ds \\
&&\leq \int_{0}^{t^{\prime }}C_{\alpha }\left( t-t^{\prime }\right) ^{\alpha
}\left\Vert \nabla \mathcal{P}_{t^{\prime }-s}\left( \Phi \left( p\right)
\left( s\right) h\left( s\right) \right) \right\Vert _{\alpha }ds \\
&\leq &C_{\alpha }\left( t-t^{\prime }\right) ^{\alpha }\int_{0}^{t^{\prime
}}\frac{C_{T,\alpha }C_{3}\left( g,f,b,p_{0},T\right) }{\left( t^{\prime
}-s\right) ^{\frac{1}{2}+\alpha }}ds\leq C\left( t-t^{\prime }\right)
^{\alpha }
\end{eqnarray*}%
Therefore, also the second
condition in the definition of $\Xi _{\alpha ,C_{1},\beta ,C_{2},\rho ,C_{3}}
$ is satisfied, even uniformly with respect to $R$.
The difficult property is 
\begin{equation*}
\left\vert \Phi \left( p\right) \left( t,x\right) \right\vert \leq C_{3}\rho
\left( x\right) 
\end{equation*}%
for every $x\in \mathbb{R}^{d}$, $t\in \left[ 0,T\right] ,p\in K$, for a
suitable constant $C_{3}>0$. The idea is to write an equation for $\pi _{p}\left( t,x\right) :=\rho ^{-1}\left( x\right) \Phi \left( p\right)\left( t,x\right)$ and deduce that $\left\Vert \pi _{p}\left( t\right) \right\Vert _{\infty }\leq C_{3}$ for every $t\in \left[ 0,T\right] ,p\in K$.
We use the weak formulation
\begin{equation*}
\left\langle \Phi \left( p\right) \left( t\right) ,\varphi \right\rangle 
 =\left\langle p_{0}, \varphi  \right\rangle +\frac{1}{2}\int_{0}^{t}\left%
\langle \Phi \left( p\right) \left( s\right) ,\Delta \varphi  \right\rangle ds +\int_{0}^{t}\left\langle \Phi \left( p\right) \left( s\right) h\left(
s\right) ,\nabla \varphi  \right\rangle ds
\end{equation*}
with a test function $\varphi $ of the form $\rho ^{-1}\psi $ with $\psi \in
C_{c}^{\infty }\left( \mathbb{R}^{d}\right) $. Then%
\begin{eqnarray*}
\left\langle \pi _{p}\left( t\right) ,\psi \right\rangle  &=&\left\langle
\rho ^{-1}p_{0},\psi \right\rangle +\frac{1}{2}\int_{0}^{t}\left\langle \pi
_{p}\left( s\right) ,\Delta \psi \right\rangle ds \\
&&+\frac{1}{2}\int_{0}^{t}\left\langle \Phi \left( p\right) \left( s\right)
,\psi \Delta \rho ^{-1}+2\nabla \psi \cdot \nabla \rho ^{-1}\right\rangle ds
\\
&&+\int_{0}^{t}\left\langle \Phi \left( p\right) \left( s\right) h\left(
s\right) ,\nabla \left( \rho ^{-1}\varphi  \right) \right\rangle ds
\end{eqnarray*}%
namely, formally speaking, 
\begin{eqnarray*}
\pi _{p}\left( t\right)  &=&\rho ^{-1}p_{0}+\frac{1}{2}\int_{0}^{t}\Delta
\left( \rho ^{-1}\Phi \left( p\right) \left( s\right) \right) ds \\
&&+\frac{1}{2}\int_{0}^{t}\left( \Delta \rho ^{-1}\right) \Phi \left(
p\right) \left( s\right) ds-\int_{0}^{t}\text{div}\left( \Phi \left(
p\right) \left( s\right) \nabla \rho ^{-1}\right) ds \\
&&-\int_{0}^{t}\rho ^{-1}\text{div}\left( \Phi \left( p\right) \left(
s\right) h\left( s\right) \right) ds.
\end{eqnarray*}
Using 
\begin{equation*}
\rho ^{-1}\text{div}\left( \Phi \left( p\right) \left( s\right) h\left(
s\right) \right) =\text{div}\left( \pi _{p}\left( s\right) h\left( s\right)
\right) -\Phi \left( p\right) \left( s\right) h\left( s\right) \cdot \nabla
\rho ^{-1}
\end{equation*}%
this leads to  
\begin{eqnarray*}
\pi _{p}\left( t\right)  &=&\mathcal{P}_{t}\left( \rho ^{-1}p_{0}\right) +%
\frac{1}{2}\int_{0}^{t}\mathcal{P}_{t-s}\left( \left( \Delta \rho
^{-1}\right) \Phi \left( p\right) \left( s\right) \right) ds \\
&&+\int_{0}^{t}\nabla \mathcal{P}_{t-s}\left( \Phi \left( p\right) \left(
s\right) \nabla \rho ^{-1}\right) ds \\
&&+\int_{0}^{t}\nabla \mathcal{P}_{t-s}\left( \pi _{p}\left( s\right)
h\left( s\right) \right) ds \\
&&+\int_{0}^{t}\mathcal{P}_{t-s}\left( \nabla \rho ^{-1}\cdot \Phi \left(
p\right) \left( s\right) h\left( s\right) \right) ds.
\end{eqnarray*}
Therefore%
\begin{eqnarray*}
\left\Vert \pi _{p}\left( t\right) \right\Vert _{\infty } &\leq &\left\Vert
\rho ^{-1}p_{0}\right\Vert _{\infty }+\frac{1}{2}\int_{0}^{t}\left\Vert
\left( \Delta \rho ^{-1}\right) \Phi \left( p\right) \left( s\right)
\right\Vert _{\infty }ds \\
&&+\int_{0}^{t}\frac{C_{T}}{\sqrt{t-s}}\left\Vert \Phi \left( p\right)
\left( s\right) \nabla \rho ^{-1}\right\Vert _{\infty }ds \\
&&+\int_{0}^{t}\frac{C_{T}}{\sqrt{t-s}}\left\Vert \pi _{p}\left( s\right)
h\left( s\right) \right\Vert _{\infty }ds \\
&&+\int_{0}^{t}\left\Vert \nabla \rho ^{-1}\cdot \Phi \left( p\right) \left(
s\right) h\left( s\right) \right\Vert _{\infty }ds
\end{eqnarray*}%
\begin{eqnarray*}
&\leq &\left\Vert \rho ^{-1}p_{0}\right\Vert _{\infty }+\frac{\left\Vert
\Delta \rho ^{-1}\right\Vert _{\infty }\left\Vert \Phi \left( p\right)
\right\Vert _{\infty }}{2}T \\
&&+2C_{T}T^{1/2}\left\Vert \nabla \rho ^{-1}\right\Vert _{\infty }\left\Vert
\Phi \left( p\right) \right\Vert _{\infty }+C_{h}\left( g,f,b,T\right)
\int_{0}^{t}\frac{C_{T}}{\sqrt{t-s}}\left\Vert \pi _{p}\left( s\right)
\right\Vert _{\infty }ds \\
&&+T\left\Vert \nabla \rho ^{-1}\right\Vert _{\infty }C_{3}\left(
g,f,b,p_{0},T\right) .
\end{eqnarray*}%
Recall that $\left\Vert \Phi \left( p\right) \right\Vert _{\infty }\leq
C_{2}\left( g,f,b,p_{0},T\right) $ independently of $p\in K$. Moreover
recall that $\left\Vert \Delta \rho ^{-1}\right\Vert _{\infty }+\left\Vert
\nabla \rho ^{-1}\right\Vert _{\infty }<\infty $. From a generalized form of
Gronwall lemma we deduce a uniform bound for $\left\Vert \pi _{p}\left(
t\right) \right\Vert _{\infty }$.
\end{proof} 
\noindent At this point, we can apply Brouwer-Schauder fixed point theorem and have existence of a weak solution $\left( w,p\right) $. The proof that $\left(u, p\right) :=\left( -\log w,p\right) $ satisfies the original system can then be done by means of mollifiers.}
\subsection{Proof of Theorem \ref{th:localExistenceUniqueness}}
\label{subsec:local_existence_uniqueness}
Throughout this section, we assume that $p_0, b, f, g$ satisfy the hypotheses (H1)-(H2) and (H4) in Section \ref{sec:preliminaries}. 
\begin{proof}
We are going to apply the contraction principle to the system in Eqs.~ \eqref{eq:mfgsystemmild_eq1}-\eqref{eq:mfgsystemmild_eq2}. Setting $\theta \doteq \nabla u$, for $T$ small enough, it reads as 
\begin{equation*}
\begin{split}
p\left( t\right) & =\mathcal{P}_{t}p_{0}-\int_{0}^{t}\nabla \mathcal{P}%
_{t-s}\left( p\left( s\right) \left( \theta \left( s\right) -b\left(\,\cdot\,,p(s)\right) \right) \right)\,ds \\
\theta \left( t\right) & =\nabla \mathcal{P}_{T-t}g-\int_{t}^{T}\nabla 
\mathcal{P}_{r-t}\left( b\left(\,\cdot\,,p(s)\right) \cdot \theta \left(
r\right) -\frac{1}{2}\left\vert \theta \left( r\right) \right\vert
^{2}+f\left(\,\cdot\,, p(r) \right) \right)\,dr.\\
\end{split}
\end{equation*}
\indent Now, consider the following Banach space: 
\begin{equation*}
\textcolor{black}{X_{T}=\text{C}_b( \left[ 0,T\right] \times \mathbb{R}^{d})
\times \text{C}_b( \left[ 0,T\right] \times \mathbb{R}^{d})},
\end{equation*}
and  by $\left\Vert\,\cdot\,\right\Vert _{T,\infty}$ the norm in each space $\textcolor{black}{\text{C}_b( \left[ 0,T\right] \times \mathbb{R}^{d})}$. On the product space $X_{T}$ consider the norm
\begin{equation*}
\left\Vert \left( a,b\right) \right\Vert _{T,\infty } \doteq \left\Vert
 a \right\Vert _{T,\infty }+\left\Vert  b \right\Vert _{T,\infty }.
\end{equation*}
Define the map $\Gamma : X_{T}\rightarrow X_{T}$ as
\begin{equation}\label{eq:Gamma_map}
\Gamma \left( p,\theta \right)  = \left( \Gamma _{1}\left( p,\theta \right)
,\Gamma _{2}\left( p,\theta \right) \right)
\end{equation}
whose marginals are given by
\begin{equation*}
\begin{split}
\Gamma _{1} \left( p,\theta \right) \left( t\right) & \doteq\mathcal{P}%
_{t}p_{0}-\int_{0}^{t}\nabla \mathcal{P}_{t-s}\left( p\left( s\right) \left(
\theta \left( s\right) - b\left(\,\cdot\,,p(s)\right) \right)\right)\\
\Gamma _{2}\left( p,\theta \right) \left( t\right) & \doteq\nabla \mathcal{P}%
_{T-t}g-\int_{t}^{T}\nabla \mathcal{P}_{r-t}\left( b\left(\,\cdot\,,p(r)\right)\cdot \theta \left( r\right) -\frac{1}{2}\left\vert \theta \left(
r\right) \right\vert ^{2}+f\left(\,\cdot\, p(r) \right) \right) dr.
\end{split}
\end{equation*}
Notice that the fact that $\Gamma \left( p,\theta \right) \in X_{T}$ when $\left(
p,\theta \right) \in X_{T}$ is implicit in the following computations and
thus it will not be explained a priori. It is based on the following estimates of the heat semi-group's gradient (cfr. also the proof of Theorem \ref{th:globalExistence} and the reference therein): $\left\Vert \nabla \mathcal{P}_{t}F\right\Vert _{\infty }\leq  C_{0}t^{-1/2}\left\Vert F\right\Vert _{\infty }$ for some constant $C_0$ and every $F\in \text{C}_{b}\left( \mathbb{R}^{d}\right) $ and $\left\Vert \nabla \mathcal{P}_{t}F\right\Vert _{\infty }\leq C_{0}\left\Vert \nabla F\right\Vert _{\infty }$ for every $F\in \text{C}_{b}\left( \mathbb{R}^{d}\right) $ such that $\nabla F\in \text{C}_{b}\left( \mathbb{R}^{d}\right) $.\\
\indent Now, let us investigate when $\Gamma $ is a contraction. We have
\begin{equation*}
\begin{split}
& \left\Vert \Gamma _{1}\left( p,\theta \right) \left( t\right) -\Gamma
_{1}\left( p^{\prime },\theta ^{\prime }\right) \left( t\right) \right\Vert
_{\infty } \\
& \leq \int_{0}^{t}\frac{C_{0}}{\left( t-s\right) ^{1/2}}\left( \left\Vert
p\left( s\right) \right\Vert _{\infty }\left( \left\Vert \theta \left(
s\right) -\theta ^{\prime }\left( s\right) \right\Vert _{\infty }+\left\Vert
b\left(\,\cdot\,, p\left( s\right)\right) -b\left(\,\cdot\,, p^{\prime}\left( s\right)\right)
\right\Vert _{\infty }\right) \right) ds \\
&+\int_{0}^{t}\frac{C_{0}}{\left( t-s\right) ^{1/2}}\left\Vert p\left(
s\right) -p^{\prime }\left( s\right) \right\Vert _{\infty }\left( \left\Vert
\theta ^{\prime }\left( s\right) \right\Vert _{\infty }+ \left\Vert
b\left(\,\cdot\,, p^{\prime}\left( s\right)\right)
\right\Vert _{\infty }\right) ds\\
& \leq \int_{0}^{t}\frac{C_{0}}{\left( t-s\right) ^{1/2}}\left( \left\Vert
p\left( s\right) \right\Vert _{\infty }\left( \left\Vert \theta \left(
s\right) -\theta ^{\prime }\left( s\right) \right\Vert _{\infty
}+L_{b}\left\Vert p\left( s\right) -p^{\prime }\left( s\right) \right\Vert
_{\infty }\right) \right) ds \\
&+\int_{0}^{t}\frac{C_{0}}{\left( t-s\right) ^{1/2}}\left\Vert p\left(
s\right) -p^{\prime }\left( s\right) \right\Vert _{\infty }\left( \left\Vert
\theta ^{\prime }\left( s\right) \right\Vert _{\infty }+\text{C}_{b}\left(
1+\left\Vert p^{\prime }\left( s\right) \right\Vert _{\infty }\right)
\right) ds\\
& \leq \int_{0}^{t}\frac{C_{0}}{\left( t-s\right) ^{1/2}}ds\cdot \left(
\left\Vert p\right\Vert _{T,\infty }\left( \left\Vert \theta -\theta
^{\prime }\right\Vert _{T,\infty }+L\left\Vert p-p^{\prime }\right\Vert
_{T,\infty }\right) \right) \\
& +\int_{0}^{t}\frac{C_{0}}{\left( t-s\right) ^{1/2}}ds\cdot \left\Vert
p-p^{\prime }\right\Vert _{T,\infty }\left( \left\Vert \theta ^{\prime
}\right\Vert _{T,\infty }+C\left( 1+\left\Vert p^{\prime }\right\Vert
_{T,\infty }\right) \right)\\
&\leq 2C_{0}\sqrt{T}\left[ \left\Vert p\right\Vert _{T,\infty }\left\Vert
\theta -\theta ^{\prime }\right\Vert _{T,\infty }+\left( C+\left\Vert \theta
^{\prime }\right\Vert _{T,\infty }+C\left\Vert p^{\prime }\right\Vert
_{T,\infty }+L\left\Vert p\right\Vert _{T,\infty }\right) \left\Vert
p-p^{\prime }\right\Vert _{T,\infty }\right] 
\end{split}
\end{equation*}
and 
\begin{equation*}
\resizebox{1.0 \textwidth}{!}{$
\begin{split}
&\left\Vert \Gamma_{2}\left( p,\theta\right) \left( t\right) -\Gamma
_{2}\left( p^{\prime},\theta^{\prime}\right) \left( t\right) \right\Vert
_{\infty}\\
& \leq2C_{0}\sqrt{T}\left[ \left( C+C\left\Vert p\right\Vert _{T,\infty}+\frac{%
\left\Vert \theta\right\Vert _{T,\infty}+\left\Vert \theta^{\prime
}\right\Vert _{T,\infty}}{2}\right) \left\Vert \theta-\theta^{\prime
}\right\Vert _{T,\infty}+\left( L+L\left\Vert \theta^{\prime}\right\Vert
_{T,\infty}\right) \left\Vert p-p^{\prime}\right\Vert _{T,\infty}\right],
\end{split}$}
\end{equation*}
respectively.\\
\noindent Summarizing, there exists a constant $\widetilde{C}>0$, depending only on $%
C_{0}$, $C$, $L$, such that 
\begin{equation*}
\left\Vert \Gamma\left( p,\theta\right) -\Gamma\left(
p^{\prime},\theta^{\prime}\right) \right\Vert _{T,\infty}\leq\widetilde{C}%
\sqrt {T}\left\Vert \left( p,\theta\right) -\left( p^{\prime},\theta^{\prime
}\right) \right\Vert _{T,\infty}\left( \left\Vert \left( p,\theta\right)
\right\Vert _{T,\infty}+\left\Vert \left( p^{\prime},\theta^{\prime}\right)
\right\Vert _{T,\infty}\right) .
\end{equation*}
Therefore, to have a contraction we need a bound on $\left\Vert \left(
p,\theta\right) \right\Vert _{T,\infty}+\left\Vert \left(
p^{\prime},\theta^{\prime}\right) \right\Vert _{T,\infty}$.
Proceeding as above we have
\begin{equation*}
\begin{split}
&\left\Vert \Gamma _{1}\left( p,\theta \right) \left( t\right) \right\Vert
_{\infty } \textcolor{black}{\leq} \left\Vert p_{0}\right\Vert _{\infty }+\int_{0}^{t}\frac{C_{0}}{
\left( t-s\right) ^{1/2}}\left( \left\Vert p\left( s\right) \right\Vert
_{\infty }\left( \left\Vert \theta \left( s\right) \right\Vert _{\infty
}+\left\Vert b(\,\cdot\,,p(s)) \right\Vert _{\infty }\right)
\right) ds\\
&\left\Vert \Gamma _{2} \left( p,\theta \right) \left( t\right) \right\Vert
_{\infty } \leq C_{0}\left\Vert \nabla g\right\Vert _{\infty }\\
&+\int_{t}^{T}%
\frac{C_{0}}{\left( r-t\right) ^{1/2}}\left( \left\Vert b(\,\cdot\,,p(r)) \right\Vert _{\infty }\left\Vert \theta \left( r\right)
\right\Vert _{\infty }+\frac{1}{2}\left\Vert \theta \left( r\right)
\right\Vert _{\infty }^{2}+\left\Vert f(\,\cdot\,,p(r))
\right\Vert _{\infty }\right) dr 
\end{split}
\end{equation*}
and 
\begin{equation*}
\begin{split}
\left\Vert \Gamma _{1}\left( p,\theta \right) \left( t\right) \right\Vert
_{\infty }& =\left\Vert p_{0}\right\Vert _{\infty }+2C_{0}\sqrt{T}\left(
\left\Vert p\right\Vert _{T,\infty }\left( \left\Vert \theta \right\Vert
_{T,\infty }+\left\Vert b\left( p\right) \right\Vert _{T,\infty }\right)
\right) \\
\left\Vert \Gamma _{2}\left( p,\theta \right) \left( t\right) \right\Vert
_{\infty }& \leq C_{0}\left\Vert \nabla g\right\Vert _{\infty }+2C_{0}\sqrt{T%
}\left( \left\Vert b\left( p\right) \right\Vert _{T,\infty }\left\Vert
\theta \right\Vert _{T,\infty }+\frac{1}{2}\left\Vert \theta \right\Vert
_{T,\infty }^{2}+\left\Vert f\left( p\right) \right\Vert _{T,\infty }\right)
\end{split}
\end{equation*}
\noindent Using the bound on $b$ and $f$, we get
\begin{equation*}
\begin{split}
\left\Vert \Gamma _{1}\left( p,\theta \right) \left( t\right) \right\Vert
_{T,\infty }& \textcolor{black}{\leq}\left\Vert p_{0}\right\Vert _{\infty }+2C_{0}\sqrt{T}\cdot
\left( \left\Vert p\right\Vert _{T,\infty }\left( \left\Vert \theta
\right\Vert _{T,\infty }+C\left( 1+\left\Vert p\right\Vert _{T,\infty
}\right) \right) \right) \\
\left\Vert \Gamma _{2}\left( p,\theta \right) \left( t\right) \right\Vert
_{T,\infty }& \leq C_{0}\left\Vert \nabla g\right\Vert _{\infty }\\
				&+2C_{0}%
\sqrt{T}\cdot \left( C\left( 1+\left\Vert p\right\Vert _{T,\infty }\right)
\left\Vert \theta \right\Vert _{T,\infty }+\frac{1}{2}\left\Vert \theta
\right\Vert _{T,\infty }^{2}+C\left( 1+\left\Vert p\right\Vert _{T,\infty
}\right) \right)
\end{split}
\end{equation*}
Therefore, we have proved:
\begin{equation*}
\left\Vert \Gamma \left( p,\theta \right) \right\Vert _{T,\infty }\leq
\left( C_{0}\left\Vert \nabla g\right\Vert _{\infty }+\left\Vert
p_{0}\right\Vert _{\infty }\right) +2C_{0}\sqrt{T}K\cdot \left( \left\Vert
\left( p,\theta \right) \right\Vert _{T,\infty }+\left\Vert \left( p,\theta
\right) \right\Vert _{T,\infty }^{2}\right)
\end{equation*}%
for some constant $K>0$. Hence setting%
\begin{equation*}
\Lambda _{T,R}=\left\{ \left( p,\theta \right) \in X_{T}:\left\Vert \left(
p,\theta \right) \right\Vert _{T,\infty }\leq R\right\}
\end{equation*}%
if we take $\left( p,\theta \right) \in \Lambda _{T,R}$ we get%
\begin{equation*}
\left\Vert \Gamma \left( p,\theta \right) \right\Vert _{T,\infty }\leq
\left( C_{0}\left\Vert \nabla g\right\Vert _{\infty }+\left\Vert
p_{0}\right\Vert _{\infty }\right) +2C_{0}\sqrt{T}K\left( R+R^{2}\right) .
\end{equation*}%
In particular, there exist $T_{0},R_{0}>0$ such that for every $0<T\leq T_{0}$ and $0<R\leq
R_{0}$ we have 
\begin{equation*}
\left( C_{0}\left\Vert \nabla g\right\Vert _{\infty }+\left\Vert
p_{0}\right\Vert _{\infty }\right) +2C_{0}\sqrt{T}K\left( R+R^{2}\right)
\leq R.
\end{equation*}%
With any such choice of $T,R>0$ we have 
\begin{equation*}
\Gamma \left( \Lambda _{T,R}\right) \subset \Lambda _{T,R}.
\end{equation*}%
If $\left( p,\theta \right) ,\left( p^{\prime },\theta ^{\prime }\right) \in
\Lambda _{T,R}$ we have proved above%
\begin{equation*}
\left\Vert \Gamma \left( p,\theta \right) -\Gamma \left( p^{\prime },\theta
^{\prime }\right) \right\Vert _{T,\infty }\leq 2 R \widetilde{C} \sqrt{T}%
\left\Vert \left( p,\theta \right) -\left( p^{\prime },\theta ^{\prime
}\right) \right\Vert _{T,\infty }.
\end{equation*}%
Hence, reducing $T$ if necessary, we see that $\Gamma $, as a map from the
metric space $\Lambda _{T,R}$ into itself, is a contraction. 
\end{proof}

\subsection{Proof of Theorem \ref{th:VerificationTheorem}-\textit{(i)}}\label{subsec:proofVerificationTheorem}
\begin{proof}
Let $\epsilon>0$ and let be $(\theta_{\epsilon})_{\epsilon > 0}$ be a family of mollifiers. Now, 
define the function $u_{\epsilon}:[0,T] \times \mathbb{R}^{d} \rightarrow \mathbb{R}$ by setting
\begin{equation*}
u_{\epsilon}(t, x) \doteq (\theta_{\epsilon} * u(t,\,\cdot\,))(x) = \int_{\mathbb{R}^{d}} \theta_{\epsilon}(x-y)\,u(t, y)\,dy.
\end{equation*}
In particular, taking the convolution of the Hamilton-Jacobi Bellman equation \eqref{eq:pdesystem} with $\theta_{\epsilon}$ it is not difficult to see that $u_{\epsilon}$ satisfies the following equation
\begin{equation*}
-\partial_{t} u_{\epsilon} - \frac{1}{2}\Delta u_{\epsilon} - \theta_{\epsilon} * (b(x, p(t, x)) \cdot \nabla u) + \frac{1}{2}\,\theta_{\epsilon} * |\nabla u|^2 = \theta_{\epsilon} * f(x, p(t, x))
\end{equation*}
on $(0, T) \times \mathbb{R}^{d}$. The smoothing properties of convolution (see Proposition \ref{prop:brezisConv}) guarantees that $D^2 u_{\epsilon}(t, x)$ is continuous; besides, 
from the Hamilton-Jacobi Bellman equation it follows that also $\partial_t u_{\epsilon}$ is continuous, and therefore that $u_{\epsilon} \in \text{C}^{1,2}((0,T) \times \mathbb{R}^{d})$. Applying It\^{o}'s formula we obtain
\begin{equation*}
\begin{split}
d u_{\epsilon}(t, X_t^{\alpha}) &=  \partial_t u_{\epsilon} dt + \nabla u_{\epsilon} \cdot (\alpha_t + b(X_t^{\alpha}, p(t, X_t^{\alpha})))\,dt + \nabla u_{\epsilon} \cdot dW_t + \frac{1}{2} \nabla u_{\epsilon}\,dt\\
											 &= \left(\partial_t u_{\epsilon} + \frac{1}{2} \nabla u_{\epsilon} + \theta_{\epsilon} * (b(\,\cdot\,,p) \cdot \nabla u_{\epsilon})(t, X_t^{\alpha})\right)\,dt \\
											 &+ \left((\alpha_t + b(X_t^{\alpha}, p(t, X_t^{\alpha}))\cdot \nabla u_{\epsilon}(t, X_t^{\alpha}) - \theta_{\epsilon} * (b(\,\cdot\,,p) \cdot \nabla u)(t, X_t^{\alpha})\right)\,dt\\
											 &+ \nabla u_{\epsilon}(t, X_t^{\alpha})\cdot dW_t\\
											 &= \left(\frac{1}{2}(\theta_{\epsilon} * |\nabla u|^2)(t, X_t^{\alpha}) - (\theta_{\epsilon} * f(\,\cdot\,,p))(t, X_t^{\alpha})\right)\,dt\\
											 &+ \left(\alpha_t \nabla u_{\epsilon}(t, X_t^{\alpha}) + r_{\epsilon}\right)\,dt + \nabla u_{\epsilon}(t, X_t^{\alpha}) \cdot dW_t,
\end{split}
\end{equation*}
where we defined
$$r_{\epsilon}(t) \doteq b(X_t^{\alpha}, p(t, X_t^{\alpha})) \cdot \nabla u_{\epsilon}(t, X_t^{\alpha}) - \theta_{\epsilon} * (b(\,\cdot\,,p) \cdot \nabla u)(t, X_t^{\alpha}).$$
Hence, 
\begin{equation*}
\begin{split}
\mathbb{E}[(\theta_{\epsilon}*  g)(X_T^{\alpha})] & - \mathbb{E}[u_{\epsilon}(0, X_0^{\alpha})]\\
											   &= \mathbb{E}\left[\int_{0}^{T}\left(\frac{1}{2}(\theta_{\epsilon} * |\nabla u_{\epsilon}|^{2})(t, X_t^{\alpha}) - (\theta_{\epsilon} * f(\,\cdot\,,p))(t, X_t^{\alpha}\right)\,dt\right]\\
											   &+ \mathbb{E}\left[\int_{0}^{T} (\alpha_t \cdot \nabla u_{\epsilon}(t, X_t^{\alpha}) + r_{\epsilon}(t))\,dt\right]
\end{split}
\end{equation*}
We claim that by taking the limit as $\epsilon \rightarrow 0$ in the previous equation we obtain the identity \eqref{eq:euristic} as in the heuristic argument.\\ 
\indent We first deal with terms that do not explicitly depend on time, then extend the argument to time-dependent terms. To this end, let $v \in \text{C}(\mathbb{R}^{d})$; then, $\theta_{\epsilon} * v \rightarrow v$ as $\epsilon \rightarrow 0$ uniformly on compact sets (see Proposition \ref{prop:brezisConv2}). Set now $v_{\epsilon} \doteq \theta_{\epsilon} * v$. If $v$ is bounded by a constant \textcolor{black}{$K$}, then the same holds for $v_{\epsilon}$ and the constant bounding $v_{\epsilon}$ is independent of $\epsilon$. For all $R > 0$ and for any probability measure $\mu \in \mathcal{P}(\mathbb{R}^{d})$ we have 
\begin{equation}\label{eq:convergence_zero_vepsilon_x}
\begin{split}
\Bigg|\int_{\mathbb{R}^{d}}(v_{\epsilon}(x) - v(x))\,\mu(d x)\Bigg| &\leq |\overline{B}_{R}(0)| \sup_{x \in \overline{B}_{R}(0)} |v_{\epsilon}(x) - v(x)| \\
																									&+ \int_{\mathbb{R}^{d}\setminus \overline{B}_{R}(0)} (\| v_{\epsilon} \|_{\infty} + \| v \|_{\infty})\,\mu(dx)\\
																									&\rightarrow 2 K \int_{\mathbb{R}^{d}\setminus \overline{B}_{R}(0)} \mu(dx)\quad\text{as}\quad\epsilon\rightarrow 0,
\end{split}
\end{equation}
where $\overline{B}_{R}(0) \subset \mathbb{R}^{d}$ denotes the closed ball of radius $R$ around the origin and $|\overline{B}_{R}(0)|$ its measure. In particular, the last term in \eqref{eq:convergence_zero_vepsilon_x} converges to zero as $R \rightarrow \infty$.\\
\noindent

Let now $u(t,\,\cdot\,)\in\text{C}_{b}(\mathbb{R}^{d})$, bounded by a constant $K$, with $u$ the first component of the solution of the PDE system in Eq.~\eqref{eq:pdesystem}. Moreover, let $\mu_t^{\alpha}$ the law of $X_t^{\alpha}$. Then 
\begin{equation*}
\begin{split}
\mathbb{E}[(\theta_{\epsilon} * u(t,\,\cdot\,))(X_{t}^{\alpha})] = \int_{\mathbb{R}^{d}} (\theta_{\epsilon} * u(t,\cdot))(x)\,\mu_t^{\alpha}(d x) \rightarrow \int_{\mathbb{R}^{d}} u(t,x) \mu_t^{\alpha}(t)(dx)\quad\text{as}\quad\epsilon\rightarrow 0
\end{split}
\end{equation*}
for all $t \in [0,T]$, so in particular for $t = T$ and $u(T) = g$.\\
\indent Now, we show that a similar argument holds also for terms that have an explicit, continuous, dependence on the time variable. Let  \textcolor{black}{$v \in \text{C}_{b}([0, T] \times \mathbb{R}^{d})$}; then for each fixed $t \in [0, T]$ we have that $\theta_{\epsilon} * v(t) \rightarrow v(t)$ as $\epsilon \rightarrow 0$ uniformly on compact sets (see, again, Proposition \ref{prop:brezisConv2}). In particular, for all $R > 0$ and for any probability measure $\mu \in \mathcal{P}(\mathbb{R}^{d})$ we have:
\begin{equation*}
\begin{split}
 \Bigg| \int_{\mathbb{R}^{d}}\int_{0}^{T}(v_{\epsilon}(t, x) - v(t, x))\,dt\,\mu(dx) \Bigg| &\leq \int_{\overline{B}_{R}(0)}\int_{0}^{T} |v_{\epsilon}(t, x) - v(t, x)|\,dt\,\mu(dx)\\
 &+\int_{\mathbb{R}^{d}\setminus \overline{B}_{R}(0)} \int_{0}^{T}  (\| v_{\epsilon} \|_{\infty} + \| v \|_{\infty})\,dt\,\mu(dx).
\end{split}
\end{equation*}
\noindent The first term converges to zero as $\epsilon \rightarrow 0$ provided that both $v_{\epsilon}$ and $v$ belongs to $\text{C}([0,T] \times \mathbb{R}^{d})$; indeed, in this case we can compute the maximum over $[0, T]$. The second term converges to zero by an argument similar to that used in Eq.~ \eqref{eq:convergence_zero_vepsilon_x}.\\
\noindent However, if \textcolor{black}{$v \in  \text{C}_{b}([0, T] \times \mathbb{R}^{d})$}, then $v(t,\,\cdot\,) \in \text{C}_b(\mathbb{R}^{d})$ and $v(\,\cdot\,,x) \in \text{C}([0,T])$; therefore, the compactness of $[0,T]$ implies the uniform continuity of $v(\,\cdot\,,x)$. Then, the fact that $v(t,\,\cdot\,) \in \text{C}_b(\mathbb{R}^{d})$ and the uniform continuity of $v(\,\cdot\,,x)$ imply the joint continuity of $v$. Indeed, let $(t, x) \in [0,T] \times \mathbb{R}^{d}$. For all $\epsilon > 0$ there exist $\delta > 0$ and $ \eta > 0$ such that
\begin{equation*}
|v(t',x')-v(t,x)|<\epsilon\quad\forall (t',x')\in [0,T]\times\mathbb{R}^d\quad \text{s.t.}\quad|x-x'|<\eta,\quad |t-t'|<\delta.
\end{equation*}
\noindent More precisely, let $\delta>0$ be the constant related to the uniform continuity in time associated to $\epsilon/2$ and $\eta>0$ be the constant related to the continuity in space associated to $\epsilon/2$. Then:
\begin{equation*}
\begin{split}
|v(t',x')-v(t,x)|&\leq |v(t',x')-v(t,x')|+|v(t,x')-v(t,x)|\\
					 &< \frac{\epsilon}{2}+\frac{\epsilon}{2}\quad\forall (t',x')\in [0,T]\times\mathbb{R}^d\quad \text{s.t.}\quad|x-x'|<\eta,\quad |t-t'|<\delta.
\end{split}
\end{equation*}
\noindent By the fact that all our terms satisfy the required continuity as $v$, by the boundedness of the  admissible controls and by choosing $\mu=\mu^{\alpha}$ law of $X^{\alpha}$ we conclude.
\end{proof}

\section{H\"{o}lder-type seminorm bounds-1}
\label{app:holder_heat_bound}
\noindent This section collects some results for H\"{o}lder-type seminorm (see definition in Eq.~ \eqref{eq:holder_seminorm}) used in the  proof of Theorem \ref{th:oelschlager85New}.\\
\textcolor{black}{We start by fixing the fractional exponent $s \in (0,1)$ and for any $p \in [1, +\infty)$, we define $W^{s,p}(\mathbb{R}^{d})$ as the space:
\begin{equation*}
W^{s,p}(\mathbb{R}^{d}) \doteq \left\{ f \in L^{p}(\mathbb{R}^{d})\,:\, \frac{
\left\vert f\left( x\right) -f\left( y\right) \right\vert}{\left\vert
x-y\right\vert ^{\frac{d}{p}+s}}\in L^{p}(\mathbb{R}^{d} \times \mathbb{R}^{d})\right\}
\end{equation*}
endowed with the following norm:
\begin{equation*}
\|f\|_{W^{s,p}(\mathbb{R}^{d})}^{p} \doteq \int_{\mathbb{R}^{d}} |f(x)|^p\,dx + \int_{\mathbb{R}^{d}}\int_{\mathbb{R}^{d}}\frac{%
\left\vert f\left( x\right) -f\left( y\right) \right\vert ^{p}}{\left\vert
x-y\right\vert ^{d+sp}}dx\,dy \doteq \|f \|_{L^{p}(\mathbb{R}^{d})}^{p} +  \left[ f\right] _{p,sp}^{p}.
\end{equation*}}
\textcolor{black}{Let $p\in[1,+\infty)$ and $s\in(0,1)$ be such that $sp > d$.} Then, there exists a constant $C > 0$, depending on $d, s, p$, such that
\begin{equation}\label{eq:embedding_revision}
\| f \|_{\infty} + [f]_{\gamma} \leq C\left(\|f\|_{L^p} + [f]_{p, sp}\right),
\end{equation}
where $\gamma \doteq (s p - d)/p$ \textcolor{black}{and $sp > d$}. We refer to \cite{di2012hitchhiker}, Theorem 8.2, for a proof of the previous result. We state the following lemma
\begin{lemma}\label{lem:lemma_bound_1}
Let $p \in [1, +\infty)$, $s \in (0,1)$ be such that $s p > d$, $d \in \mathbb{N}$. Then, 
\begin{equation*}
\left[ f\right] _{p,sp}^{p}\leq \int_{\mathbb{R}^{d}}\int_{\left\vert
h\right\vert \leq 1}\frac{\left\vert f\left( y+h\right) -f\left( y\right)
\right\vert ^{p}}{\left\vert h\right\vert ^{d+sp}}\,dh\,dy+2\,C_{p,d,s}\left\Vert
f\right\Vert _{L^{p}}^{p}.
\end{equation*}
\end{lemma}
\begin{proof}
We write $[f]_{p, s p}^{p}$ as
\begin{equation*}
\begin{split}
& \left[ f\right] _{p,sp}^{p}=\int_{\mathbb{R}^{d}}\int_{\mathbb{R}^{d}}\frac{%
\left\vert f\left( y+h\right) -f\left( y\right) \right\vert ^{p}}{\left\vert
h\right\vert ^{d+sp}}dh\,dy = I_1 + I_2\quad\text{where}\\
& I_{1} = \int_{\mathbb{R}^{d}}\int_{\left\vert h\right\vert \leq 1}\frac{\left\vert f\left( y+h\right) -f\left( y\right) \right\vert ^{p}}{\left\vert
h\right\vert ^{d+sp}}dh\,dy\,\,\text{and}\,\,I_{2}=\int_{\mathbb{R}^{d}}\int_{\left\vert h\right\vert >1}\frac{\left\vert f\left( y+h\right) -f\left( y\right) \right\vert ^{p}}{\left\vert
h\right\vert ^{d+sp}}dh\,dy
\end{split}
\end{equation*}
Then, 
\begin{equation*}
\begin{split}
&\int_{\mathbb{R}^{d}}\int_{\left\vert h\right\vert >1}\frac{\left\vert
f\left( y+h\right) -f\left( y\right) \right\vert ^{p}}{\left\vert
h\right\vert ^{d+sp}}dh\,dy \\
&\leq C_{p}\int_{\mathbb{R}^{d}}\int_{\left\vert h\right\vert >1}\frac{%
\left\vert f\left( y+h\right) \right\vert ^{p}}{\left\vert h\right\vert
^{d+sp}}dh\,dy+C_{p}\left( \int_{\left\vert h\right\vert >1}\frac{1}{%
\left\vert h\right\vert ^{d+sp}}dh\right) \left( \int_{\mathbb{R}%
^{d}}\left\vert f\left( y\right) \right\vert ^{p}dy\right) \\
&= C_{p}\int_{\left\vert h\right\vert >1}\int_{\mathbb{R}^{d}}\frac{%
\left\vert f\left( y+h\right) \right\vert ^{p}}{\left\vert h\right\vert
^{d+sp}}dy\,dh+C_{p,d,s}\left\Vert f\right\Vert _{L^{p}}^{p} \\
&=C_{p}\int_{\left\vert h\right\vert >1}\int_{\mathbb{R}^{d}}\frac{
\left\vert f\left( y\right) \right\vert ^{p}}{\left\vert h\right\vert ^{d+sp}
}dy\,dh+C_{p,d,s}\left\Vert f\right\Vert _{L^{p}}^{p}=2\,C_{p,d,s}\left\Vert
f\right\Vert _{L^{p}}^{p},
\end{split}
\end{equation*}
which concludes the proof. 
\end{proof}

\begin{lemma}
\label{lem:suff_hoelder}Assume there exists a number $\epsilon >0$ with the following property. For every $p\geq 2$ there is a function $g_{p}>0$ such that 
\begin{eqnarray}
&& \mathbb{E}\left[ \left\vert M_{t}^{N}\left( x\right) \right\vert ^{p}\right] \leq g_{p}\left( x\right),\label{eq:cond_1}\\
&& \mathbb{E}\left[ \left\vert M_{t}^{N}\left( x\right) -M_{t}^{N}\left(
x+h\right) \right\vert ^{p}\right] \leq g_{p}\left( x\right) \left\vert
h\right\vert ^{\epsilon p},\label{eq:cond_2}\\
&& \int_{\mathbb{R}^{d}}g_{p}\left( x\right) dx  < \infty\label{eq:cond_3}
\end{eqnarray}
for all $\left\vert h\right\vert \leq 1$ and $x\in \mathbb{R}^{d}$. Then,
there is $\gamma >0$ such that, for every $p\geq 2$, there is a constant $%
C_{p}>0$ such that 
\begin{equation*}
\mathbb{E}\left[ \left\Vert M_{t}^{N}\right\Vert _{\gamma }^{p}\right] \leq
C_{p}.
\end{equation*}
\end{lemma}
\begin{proof}
\textcolor{black}{It is sufficient to prove the thesis for arbitrarily large $\bar{p}\geq 2$, since for smaller ones it follows from H\"older inequality. Choose $s\in(0,\varepsilon)$; then take any $\bar{p}\geq 2$ such that $s\bar{p} > d$. We have to find $\gamma> 0$ such that for every such $\bar{p}$ there is a constant $C_{\bar{p}}$ such that $\mathbb{E}\left[ \left\Vert M_{t}^{N}\right\Vert _{\gamma }^{\bar{p}}\right] \leq C_{\bar{p}}$ uniformly in $t\in[0,T]$ and $N\in\mathbb{N}$.}\\
Thanks to the assumptions, 
\begin{equation*}
\mathbb{E}\left[ \int_{\mathbb{R}^{d}}\left\vert M_{t}^{N}\left( x\right)
\right\vert ^{\bar{p}}dx\right] \leq C.
\end{equation*}
Moreover, thanks to Lemma \ref{lem:lemma_bound_1},
\begin{equation}
\begin{split}
\mathbb{E}\left[ \left[ M_{t}^{N}\right] _{\bar{p},s \bar{p}}^{\bar{p}}\right] &\leq  \int_{
\mathbb{R}^{d}}\int_{\left\vert h\right\vert \leq 1}\frac{\mathbb{E}\left[
\left\vert M_{t}^{N}\left( y+h\right) -M_{t}^{N}\left( y\right) \right\vert
^{\bar{p}}\right] }{\left\vert h\right\vert ^{d+s\bar{p}}}dh\,dy+2\,C_{\bar{p},d,s} \mathbb{E}\left[
\left\Vert M_{t}^{N}\right\Vert _{L^{\bar{p}}}^{\bar{p}}\right] \\
& \leq \int_{\mathbb{R}^{d}}\int_{\left\vert h\right\vert \leq 1}\frac{
g_{\bar{p}}\left( y\right) \left\vert h\right\vert ^{\epsilon \bar{p}}}{\left\vert
h\right\vert ^{d+s\bar{p}}}dh\,dy+C \\
& \leq  \left( \int_{\left\vert h\right\vert \leq 1}\frac{1}{\left\vert
h\right\vert ^{d-\left( \epsilon -s\right) \bar{p}}}dh\right) \int_{\mathbb{R}
^{d}}g_{\bar{p}}\left( y\right) dy+C\leq C.
\end{split}
\end{equation}
\textcolor{black}{Now, using again the fact that $\mathbb{E}\left[ \| M^N_t \|^{\bar{p}}_{L^{\bar{p}}} \right]\leq C$, we may apply inequality \eqref{eq:embedding_revision} and deduce the desired bound for $\bar{\gamma}=(s\bar{p}-d)/\bar{p}$. A-priori this value of $\gamma$ depends on the particular $\bar{p}$ chosen above. However, it is sufficient to choose first a value $\bar{p}_0$, such that $s\bar{p}_0>d$ and prove that $\mathbb{E}\left[ \| M^N_t \|^{\bar{p}}_{\bar{\gamma}_0} \right]\leq C_{\bar{p}_0}$; then for all $\bar{p}>\bar{p}_0$, we prove the inequality with $\bar{\gamma}=s-d/\bar{p}$ which is larger than $\bar{\gamma}_0$, hence it holds also with H\"older exponent $\bar{\gamma}_0$, which can be taken as the value of $\gamma$ in the statement of the lemma.}
\end{proof}

\begin{lemma}
\label{lem:heat_semigroup_estim}
Let $N, d \in \mathbb{N}$, let $\mathcal{P}_t$ be the semi-group associated  to the density $G(t, x)$ of $x + W_t$ where $W_t$ is a standard blackian motion, $x \in \mathbb{R}^{d}$ and $t \in \textcolor{black}{(} 0, T]$. Moreover, let $V \in \text{C}_c^1(\mathbb{R}^d) \cap \mathcal{P}(\mathbb{R}^{d})$. Then
\begin{equation*}
\left\Vert \mathcal{P}_{t}h\right\Vert _{\gamma } \leq  C_{\gamma
}\left\Vert h\right\Vert _{\gamma }.
\end{equation*}
Moreover, if $R>0$ denotes a number such that the support of $V$ is contained in $B_{R}(0)$, the open ball of radius $R$ around the origin, and we write $V^{N}\left( x\right) =\epsilon_{N}^{-d}V\left( \epsilon _{N}^{-1}x\right)$, then there exist two constants $C_{T,R,V}>0$ and $\lambda _{T,R,V}>0$ with the following property: for every $\delta
,\gamma \in \left( 0,1\right) $, $x\in \mathbb{R}^{d}$, $\left\vert
h\right\vert \leq 1$ and $t\in \left[ 0,T\right] $ 
\begin{equation}
\left\vert \left( \nabla \mathcal{P}_{t}V^{N}\right) \left( x\right)
\right\vert \leq \frac{C_{T,R,V}}{t^{\frac{1-\delta }{2}}}\epsilon
_{N}^{-d-\delta }e^{-\frac{\left\vert x\right\vert }{8T}} .
\label{eq:bound1_heat_kernel}
\end{equation}
\begin{equation}
\left\vert \left( \nabla \mathcal{P}_{t}V^{N}\right) \left( x\right) -\left(
\nabla \mathcal{P}_{t}V^{N}\right) \left( x+h\right) \right\vert \leq \frac{
C_{T,R,V}}{t^{\frac{1}{2}\left( 1+\gamma \right) -\frac{\delta }{2}\left(
1-\gamma \right) }}\left\vert h\right\vert ^{\gamma }\epsilon
_{N}^{-d-\delta \left( 1-\gamma \right) } e^{-\lambda
_{T,R,V}\left\vert x\right\vert}
\label{eq:bound2_heat_kernel}
\end{equation}
\end{lemma}
\begin{proof}
\textcolor{black}{The first inequality is a well known properties of analytic semi-group (see, for instance, \cite{lunardi2012analytic})}. We give a detailed proof of the last two equalities.\\
\indent\textbf{\textit{Step 1.}}\quad We collect some preliminary fact. We recall that
\begin{equation*}
G_t(x) \doteq G(t, x) = \frac{1}{\left( 2\pi t\right) ^{d/2}}e^{-\frac{1}{2t}\left\vert x\right\vert ^{2}} \quad\text{and}\quad \left( \mathcal{P}_{t}f\right) \left( x\right) \textcolor{black}{=} \int_{\mathbb{R}^{d}}G_{t}\left( x-y\right) f\left( y\right) dy
\end{equation*}
and we find a bound for $\nabla G_{t}\left( x\right)$ and $|D^2 G_t(x)|$. Notice that
\begin{equation*}
\nabla G_{t}\left( x\right) =-\frac{x}{t}\frac{1}{\left( 2\pi t\right) ^{d/2}%
}e^{-\frac{1}{2t}\left\vert x\right\vert ^{2}} =-\frac{1}{%
\sqrt{t}}\frac{1}{\left( 2\pi t\right) ^{d/2}}\sqrt{\frac{\left\vert
x\right\vert ^{2}}{t}}e^{-\frac{1}{2t}\left\vert x\right\vert ^{2}}
\end{equation*}
hence, being $\sqrt{r}\exp \left( -\frac{1}{2}r\right) \leq \exp \left( -%
\frac{1}{4}r\right) $, 
\begin{equation*}
\left\vert \nabla G_{t}\left( x\right) \right\vert \leq \frac{1}{\sqrt{t}}%
\cdot \frac{1}{\left( 2\pi t\right) ^{d/2}} e^{-\frac{1}{4 t}\left\vert x\right\vert ^{2}}
=\frac{2^{d/2}}{\sqrt{t}}\cdot \frac{%
1}{\left( 2\pi \left( 2t\right) \right) ^{d/2}}e^{-\frac{1}{2}\frac{\left\vert x\right\vert ^{2}}{(2t)}}
\end{equation*}
Similarly, for suitable $\lambda ,C>0$,
\begin{equation*}
\left\vert D^{2}G_{t}\left( x\right) \right\vert \leq \frac{C}{t}\cdot \frac{%
1}{\left( 2\pi \left( \lambda t\right) \right) ^{d/2}}e^{-\frac{1}{2}\frac{\left\vert x\right\vert ^{2}}{(\lambda t)}}
\end{equation*}
\indent\textbf{\textit{Step 2.}} In this step we prove that 
\begin{equation*}
\left\vert \left( \nabla \mathcal{P}_{t}V^{N}\right) \left( x\right)
\right\vert \leq \frac{C_{T,R,V}}{\sqrt{t}}\epsilon _{N}^{-d} e^{-\frac{|x|}{8\textcolor{black}{T}}}
\end{equation*}
for all $x\in \mathbb{R}^{d}$ and $t \in \textcolor{black}{(} 0, T]$, for a suitable
constant $C_{T,R,V}>0$. From the bound for $\left\vert \nabla G_{t}\left( x\right) \right\vert$ in \textbf{\textit{Step 1}} we obtain
\begin{eqnarray*}
\left\vert \left( \nabla \mathcal{P}_{t}V^{N}\right) \left( x\right)
\right\vert &\leq &\int_{\mathbb{R}^{d}}\left\vert \nabla G_{t}\left(
x-y\right) \right\vert V^{N}\left( y\right) dy=\int_{B_{R}(0)
}\left\vert \nabla G_{t}\left( x-y\right) \right\vert \epsilon
_{N}^{-d}V\left( \epsilon _{N}^{-1}y\right) dy \\
&\leq &\frac{2^{d/2}}{\sqrt{t}}\epsilon _{N}^{-d}\left\Vert V\right\Vert
_{\infty }\int_{B_{R}(0)}\frac{1}{\left( 2\pi \left( 2t\right)
\right) ^{d/2}}e^{-\frac{1}{2}\frac{\left\vert x-y\right\vert ^{2}}{%
\left( 2t\right) }}\,dy.
\end{eqnarray*}%
If $\left\vert x\right\vert \leq R+1$, we bound the integral from above by
the integral on the full space, which is equal to one, and deduce%
\begin{equation*}
\sup_{\left\vert x\right\vert \leq R+1}\left\vert \left( \nabla \mathcal{P}%
_{t}V^{N}\right) \left( x\right) \right\vert \leq \frac{2^{d/2}}{\sqrt{t}}%
\epsilon _{N}^{-d}\left\Vert V\right\Vert _{\infty }.
\end{equation*}%
If $\left\vert x\right\vert >R+1$ and $\left\vert y\right\vert \leq R$, then
(we oversimplify to make expressions easier in the sequel) $\left\vert
x-y\right\vert ^{2}\geq \left\vert x-y\right\vert \geq \left\vert
x\right\vert -R$. Therefore, for $\left\vert x\right\vert >R+1$, 
\begin{equation*}
\left\vert \left( \nabla \mathcal{P}_{t}V^{N}\right) \left( x\right)
\right\vert \leq \frac{2^{d/2}}{\sqrt{t}}\epsilon _{N}^{-d}\left\Vert
V\right\Vert _{\infty }\left\vert B_{R}(0) \right\vert \frac{1}{%
\left( 2\pi \left( 2t\right) \right) ^{d/2}}e^{-\frac{1}{2}\frac{
\left\vert x\right\vert -R}{\left( 2t\right) }}
\end{equation*}
One show that there is $C_{T}>0$ such that for $t\in \left[ 0,T\right] $ and 
$\left\vert x\right\vert >R+1$, one has%
\begin{equation*}
\frac{1}{\left( 2\pi \left( 2t\right) \right) ^{d/2}}e^{-\frac{1}{2}\frac{
\left\vert x\right\vert -R}{\left( 2t\right) }} \leq C_{T}e^{-\frac{1}{8T}\left( \left\vert x\right\vert -R\right)}
\end{equation*}%
Indeed the left-hand-side is controlled (up to a constant)\ by $\left( \frac{%
\left\vert x\right\vert -R}{2t}\right) ^{d/2}e^{-\frac{1}{2}\frac{%
\left\vert x\right\vert -R}{2t}} $ (because $\left\vert x\right\vert
-R\geq 1$) and the function $r^{d/2}e^{-\frac{1}{2}r} $ is
bounded above by $e^{-\frac{1}{4}r}$, up to a constant;
finally, $e^{-\frac{1}{4}\frac{\left\vert x\right\vert -R}{2t}}
\leq e^{-\frac{1}{8T}\left( \left\vert x\right\vert
-R\right)}$.

Hence
\begin{equation*}
\left\vert \left( \nabla \mathcal{P}_{t}V^{N}\right) \left( x\right)
\right\vert \leq \frac{C_{T,R,V}}{\sqrt{t}}\epsilon _{N}^{-d} e^{-
\frac{\left\vert x\right\vert -R}{8T}}=\frac{C_{T,R,V}^{\prime }}{%
\sqrt{t}}\epsilon _{N}^{-d} e^{-\frac{\left\vert x\right\vert }{8T}}
\end{equation*}
Remaning the constant $C_{T,R,V}^{\prime }$, the same bound is true for $%
\left\vert x\right\vert \leq R+1$, hence it is true for all $x$ and all $t \in \textcolor{black}{(} 0, T]$.\\
\indent\textbf{\textit{Step 3.}} We complete the proof of \eqref{eq:bound1_heat_kernel}. In addition to the bound found in \textbf{\textit{Step 2}} we have
\begin{eqnarray*}
\left\vert \left( \nabla \mathcal{P}_{t}V^{N}\right) \left( x\right)
\right\vert &\leq &\int_{\mathbb{R}^{d}}G_{t}\left( x-y\right) \left\vert
\nabla V^{N}\left( y\right) \right\vert dy=\epsilon _{N}^{-d-1}\int_{\mathbb{%
R}^{d}}G_{t}\left( x-y\right) \left\vert \left( \nabla V\right) \left(
\epsilon _{N}^{-1}y\right) \right\vert dy \\
&\leq &\epsilon _{N}^{-d-1}\left\Vert \nabla V\right\Vert _{\infty
}\int_{B_{R}(0)}G_{t}\left( x-y\right) dy.
\end{eqnarray*}
Arguing as above we get,
\begin{equation*}
\left\vert \left( \nabla \mathcal{P}_{t}V^{N}\right) \left( x\right)
\right\vert \leq C_{T,R,V}\epsilon _{N}^{-d-1} e^{-\frac{\left\vert
x\right\vert }{8T}}
\end{equation*}
where if necessary we have renamed the constant $C_{T,R,V}$. Now, taken $%
\delta \in \left( 0,1\right) $, we use both inequalities for $\left\vert
\left( \nabla \mathcal{P}_{t}V^{N}\right) \left( x\right) \right\vert $ to
get 
\begin{equation*}
\left\vert \left( \nabla \mathcal{P}_{t}V^{N}\right) \left( x\right)
\right\vert =\left\vert \left( \nabla \mathcal{P}_{t}V^{N}\right) \left(
x\right) \right\vert ^{1-\delta }\left\vert \left( \nabla \mathcal{P}%
_{t}V^{N}\right) \left( x\right) \right\vert ^{\delta }\leq \frac{C_{T,R,V}}{%
t^{\frac{1-\delta }{2}}}\epsilon _{N}^{-d-\delta }e^{-\frac{\left\vert
x\right\vert }{8T}}.
\end{equation*}
\indent\textbf{\textit{Step 4.}}  Finally we prove \eqref{eq:bound2_heat_kernel}. We note first that 
\begin{eqnarray*}
\left\vert \left( \nabla \mathcal{P}_{t}V^{N}\right) \left( x\right) -\left(
\nabla \mathcal{P}_{t}V^{N}\right) \left( x+h\right) \right\vert &\leq
&\sup_{\left\vert \xi \right\vert \leq \left\vert h\right\vert }\left\vert
D^{2}\mathcal{P}_{t}V^{N}\left( x+\xi \right) \right\vert \left\vert
h\right\vert \\
&\leq &\sup_{\left\vert \xi \right\vert \leq \left\vert h\right\vert }\int_{%
\mathbb{R}^{d}}\left\vert D^{2}G_{t}\left( x+\xi -y\right) \right\vert
V^{N}\left( y\right) dy\left\vert h\right\vert \\
&\leq &\frac{C}{t}\frac{\left\vert h\right\vert \epsilon _{N}^{-d}\left\Vert
V\right\Vert _{\infty }}{\left( 2\pi \left( \lambda t\right) \right) ^{d/2}}%
\sup_{\left\vert \xi \right\vert \leq \left\vert h\right\vert }\int_{B_R(0) } e^{-\frac{1}{2}\frac{\left\vert x+\xi -y\right\vert ^{2}%
}{\left( \lambda t\right) }}\,dy \\
&\leq &\frac{C_{T,R,V}}{t}\left\vert h\right\vert \epsilon
_{N}^{-d}\sup_{\left\vert \xi \right\vert \leq \left\vert h\right\vert }e^{-\frac{\left\vert x+\xi \right\vert }{4\lambda T}} \\
&\leq &\frac{C_{T,R,V}}{t}\left\vert h\right\vert \epsilon
_{N}^{-d}\sup_{\left\vert \xi \right\vert \leq \left\vert h\right\vert }e^{-\frac{\left\vert x\right\vert -\left\vert \xi \right\vert }{4\lambda
T}}\\
&=&\frac{C_{T,R,V}}{t}\left\vert h\right\vert \epsilon _{N}^{-d} e^{
\frac{1}{4\lambda T}} e^{-\frac{\left\vert x\right\vert }{%
4\lambda T}} \\
&=&\frac{C_{T,R,V}^{\prime }}{t}\left\vert h\right\vert \epsilon
_{N}^{-d} e^{-\frac{\left\vert x\right\vert }{4\lambda T}} .
\end{eqnarray*}
On the other hand, it holds:
\begin{eqnarray*}
\left\vert \left( \nabla \mathcal{P}_{t}V^{N}\right) \left( x\right) -\left(
\nabla \mathcal{P}_{t}V^{N}\right) \left( x+h\right) \right\vert &\leq
&\left\vert \left( \nabla \mathcal{P}_{t}V^{N}\right) \left( x\right)
\right\vert +\left\vert \left( \nabla \mathcal{P}_{t}V^{N}\right) \left(
x+h\right) \right\vert \\
&\leq &\frac{C_{T,R,V}}{t^{\frac{1-\delta }{2}}}\epsilon _{N}^{-d-\delta
} e^{-\frac{\left\vert x\right\vert }{8T}} +\frac{C_{T,R,V}}{%
t^{\frac{1-\delta }{2}}}\epsilon _{N}^{-d-\delta }e^{-\frac{%
\left\vert x+h\right\vert }{8T}} \\
&\leq &\frac{C_{T,R,V}^{\prime }}{t^{\frac{1-\delta }{2}}}\epsilon
_{N}^{-d-\delta } e^{-\frac{\left\vert x\right\vert }{8T}}
\end{eqnarray*}%
because $\left\vert h\right\vert \leq 1$. Therefore, for every
(small) $\gamma \in \left( 0,1\right) $, 
\begin{eqnarray*}
\left\vert \left( \nabla \mathcal{P}_{t}V^{N}\right) \left( x\right) -\left(
\nabla \mathcal{P}_{t}V^{N}\right) \left( x+h\right) \right\vert &\leq &%
\frac{C_{T,R,V}^{\prime }}{t^{\frac{1-\delta }{2}\left( 1-\gamma \right) }}%
\frac{1}{t^{\gamma }}\left\vert h\right\vert ^{\gamma }\epsilon _{N}^{\left(
-d-\delta \right) \left( 1-\gamma \right) }\epsilon _{N}^{-d\gamma }e^{-\frac{\left\vert x\right\vert }{4\left( 2\wedge \lambda \right) T}} \\
&=&\frac{C_{T,R,V}^{\prime }}{t^{\frac{1}{2}\left( 1+\gamma \right) -\frac{%
\delta }{2}\left( 1-\gamma \right) }}\left\vert h\right\vert ^{\gamma
}\epsilon _{N}^{-d-\delta \left( 1-\gamma \right) }e^{-\frac{
\left\vert x\right\vert }{4\left( 2\wedge \lambda \right) T}},
\end{eqnarray*}
which completes the proof.
\end{proof}

\section{H\"{o}lder-type seminorm bounds-2}

\label{app:holder_heat_bound_1} \noindent Let $N\in\mathbb{N}$. This section
collects some results on H\"{o}lder type semi-norm for convolution of the type
$V^{N}\ast\mu_{N}$, where $V^{N}$ satisfies to hypothesis \text{(H3)}, i.e.
$V^{N}(x)=\epsilon_{N}^{-d}V(\epsilon_{N}^{-1}x)$ with $\epsilon_{N}>0$,
$\lim_{N\rightarrow\infty}\epsilon_{N}=0$, $V\in\text{C}^{1}(\mathbb{R}%
^{d})\cap\mathcal{P}(\mathbb{R}^{d})$. In addition, $\mu_{N}\in\mathcal{P}%
(\mathbb{R}^{d})$. In what follows, for pedagogical reasons, we first treat
the case in which the probability measure $\mu_{N}$ is deterministic, then we
analyse the case in which $\mu_{N}$ is stochastic; the results' proofs in the
latter case are less elementary.\newline\indent We make the following
\textbf{remark}. If $\mu\in\mathcal{P}(\mathbb{R}^{d})$, then $V^{N}\ast\mu
\in\text{C}^{1}(\mathbb{R}^{d})$. Moreover, if $(\mu_{N})_{N\in\mathbb{N}%
}\subset\mathcal{P}(\mathbb{R}^{d}) $ converges weakly to $\mu\in
\mathcal{P}(\mathbb{R}^{d})$ as $N\rightarrow\infty$, then
\[
\lim_{N\rightarrow\infty}\left\langle V^{N}\ast\mu_{N},\varphi\right\rangle
=\left\langle \mu,\varphi\right\rangle \quad\text{for all}\quad\varphi
\in\text{C}_{c}(\mathbb{R}^{d}).
\]
\noindent Indeed, $\left\langle V^{N}\ast\mu_{N},\varphi\right\rangle
=\left\langle \mu_{N},V^{N,-}\ast\varphi\right\rangle $ where $V^{N,-}\left(
x\right)  =V^{N}\left(  -x\right)  $; then $V^{N,-}\ast\varphi\rightarrow\varphi$
uniformly on $\mathbb{R}^{d}$ as $N\rightarrow\infty$ and thus $\left\langle
\mu_{N},V^{N,-}\ast\varphi\right\rangle $ converges to $\left\langle
\mu,\varphi\right\rangle $.\newline\indent Let, as usual, $\overline{B}%
_{R}(0)$ be the closed ball of radius $R$ centred around zero. Spaces like
$\text{C}_{\ell oc}^{\gamma}(\mathbb{R}^{d})$, namely with the $\ell oc$
specification, are Polish spaces; the convergence in this spaces is the
convergence in the corresponding topologies over $\overline{B}_{R}(0)$ for
each $R>0$. In addition, let
\[
\text{C}_{\ell oc}^{\gamma-}(\mathbb{R}^{d})\doteq\cap_{\substack{\gamma
^{^{\prime}}<\gamma}}\text{C}_{\ell oc}^{\gamma^{^{\prime}}}(\mathbb{R}^{d})
\]
and endow it with the natural metric which yields convergence in each
$\text{C}_{\ell oc}^{\gamma^{^{\prime}}}(\mathbb{R}^{d})$. Recall that, by
$\left\Vert f\right\Vert _{\gamma}$ we mean the sum of the supremum norm
$\left\Vert f\right\Vert _{\infty}$ on full space $\mathbb{R}^{d}$ plus the
$\gamma$-H\"{o}lder seminorm on $\mathbb{R}^{d}$.

\begin{lemma}
Let $\left(  \mu_{N}\right)  _{N\in\mathbb{N}}\subset\mathcal{P}\left(
\mathbb{R}^{d}\right)  $ be a sequence converging weakly to $\mu\in
\mathcal{P}(\mathbb{R}^{d})$. Set $p_{N}=V^{N}\ast\mu_{N}$. Let $\gamma
\in(0,1)$ be such that there exists $K>0$ for which
\[
\left\Vert p_{N}\right\Vert _{\gamma}\leq K
\]
for all $N\in\mathbb{N}$. Then $\mu$ is absolutely continuous w.r.t. Lebesgue
measure with density $p\in\text{C}_{\ell oc}^{\gamma-}\left(  \mathbb{R}%
^{d}\right)  $ and $\left\Vert p\right\Vert _{\infty}\leq K$. Moreover,
$p_{N}\rightarrow p$ in $\text{C}_{loc}^{\gamma-}\left(  \mathbb{R}%
^{d}\right)  $.
\end{lemma}

\begin{proof}
First, notice that for every $R>0$ and $\gamma^{^{\prime}}<\gamma$ the space
$C^{\gamma}(\overline{B}_{R}(0))$ is compactly embedded into $C^{\gamma
^{^{\prime}}}(\overline{B}_{R}(0))$. Take any subsequence $(p_{N_{k}}%
)_{k\in\mathbb{N}}$. Thanks to the previous compactness result, together with
a diagonal procedure on a subsequence of radius $(R_{i})_{i\in\mathbb{N}}$,
$R_{i}\rightarrow\infty$ as $i\rightarrow\infty$ and a sequence of exponents
$\gamma_{i}^{^{\prime}}<\gamma$ such that $\gamma_{i}^{^{\prime}}%
\rightarrow\gamma$ as $i\rightarrow\infty$, we may prove that there exists a
subsequence $\left(  p_{N_{k}^{^{\prime}}}\right)  _{k\in\mathbb{N}}$ which
converges in $\text{C}^{\gamma^{^{\prime}}}(\mathbb{R}^{d})$ for every
$\gamma^{\prime}<\gamma$, to a function $p\in\text{C}_{\ell oc}^{\gamma
^{^{\prime}}}(\mathbb{R}^{d})$; a priori, the function $p$ depends on the
subsequence. Therefore (see the \textbf{remark} above)
\[
\left\langle \mu,\varphi\right\rangle =\lim_{k\rightarrow\infty}\langle
p_{N_{k}^{^{\prime}}},\varphi\rangle=\left\langle p,\varphi\right\rangle
\]
for every $\varphi\in\text{C}_{c}\left(  \mathbb{R}^{d}\right)  $. Hence,
$\mu$ is absolutely continuous with respect the Lebesgue measure with density
$p$. Notice that the properties $p\geq0$ a.e. and $p\in L^{1}(\mathbb{R}^{d})$
follow from the identity $\left\langle \mu,\varphi\right\rangle =\left\langle
p,\varphi\right\rangle $ for every $\varphi\in\text{C}_{c}\left(
\mathbb{R}^{d}\right)  $. This identify uniquely $p$, independently of the
subsequence. Since the convergence in $\text{C}_{\ell oc}^{\gamma^{\prime}%
}\left(  \mathbb{R}^{d}\right)  $ is metric, we deduce that the whole sequence
$\left(  p_{N}\right)  $ converges to $p$ in $C_{\ell oc}^{\gamma^{\prime}%
}\left(  \mathbb{R}^{d}\right)  $.

Finally, the previous convergence implies pointwise convergence, hence%
\[
\left\vert p\left(  x\right)  \right\vert =\lim_{N\rightarrow\infty}\left\vert
p_{N}\left(  x\right)  \right\vert \leq K
\]
This proves $\left\Vert p\right\Vert _{\infty}\leq K$.
\end{proof}

Now, we state and prove the previous lemma in the case in which $(\mu_{N})_{N
\in\mathbb{N}} \subset\mathcal{P}(\mathbb{R}^{d})$ is a random sequence.
Recall that a random probability measure is a random variable from $(\Omega,
\mathcal{F}, \mathbb{P})$ to $\mathcal{P}(\mathbb{R}^{d})$, considered as a
Polish space with a metric inducing weak convergence of measures. Instead, a
random function $p$ of class $\text{C}_{\ell oc}^{\gamma}\left(
\mathbb{R}^{d}\right)  $ is a random variable from $\left(  \Omega
,\mathcal{F},\mathbb{P}\right)  $ to $\text{C}_{\ell oc}^{\gamma}\left(
\mathbb{R}^{d}\right)  $.

\begin{lemma}
\label{lem:upgrade} Let $\left(  \mu_{N}\right)  _{N\in\mathbb{N}}%
\subset\mathcal{P}\left(  \mathbb{R}^{d}\right)  $ be a sequence of
\textbf{random} probability measures converging in law, in the weak topology
of $\mathcal{P}(\mathbb{R}^{d})$, to a random $\mu\in\mathcal{P}%
(\mathbb{R}^{d})$. Introduce the random differentiable functions $p_{N}%
\doteq V^{N}\ast\mu_{N}$. Let $\gamma\in(0,1)$, $q\geq2$ be such that there exists
a constant $K>0$ for which
\begin{equation}
\mathbb{E}\left[  \left\Vert p_{N}\right\Vert _{\gamma}^{q}\right]  \leq
K\label{assumption appendix D q mean}%
\end{equation}
for all $N\in\mathbb{N}$. Then there exists a random function $p$ of class
$\text{C}_{\ell oc}^{\gamma-}\left(  \mathbb{R}^{d}\right)  $ such that, with
probability one, $\mu\left(  dx\right)  =p\left(  x\right)  dx$; and for every
$q^{\prime}<q$ we have%
\begin{equation}
\mathbb{E}\left[  \left\Vert p\right\Vert _{\infty}^{q^{\prime}}\right]
<K^{q^{\prime}/q}.\label{thesis appendix D q prime mean}%
\end{equation}
Moreover, $p_{N}$ converges to $p$ in law, in the topology of $\text{C}_{\ell
oc}^{\gamma-}\left(  \mathbb{R}^{d}\right)  $; and when $p$ is deterministic
(so that $p_{N}$ converges to $p$ also in probability) we have%
\begin{equation}
\lim_{N\rightarrow\infty}\mathbb{E}\left[  \left\Vert p_{N}-p\right\Vert
_{\text{C}\left(  \overline{B}_{R}(0)\right)  }^{q^{^{\prime}}}\right]
=0\label{thesis appendix D q prime mean 2}%
\end{equation}
for every $q^{\prime}<q$ and $R>0$.
\end{lemma}

\begin{proof}
Let us denote by $P_{N}$ the law of $p_{N}$ on Borel sets of $\text{C}%
^{\gamma}(\mathbb{R}^{d})$, by $\pi_{N}$ and $\pi$ the laws of $\mu_{N}$ and
$\mu$ on Borel sets $\mathcal{P}(\mathbb{R}^{d})$, respectively. We know that
$\pi_{N}$ converges weakly to $\pi$. Set
\[
\mathcal{K}_{R}\doteq\{f\in\text{C}^{\gamma}(\mathbb{R}^{d})\,:\,\Vert
f\Vert_{C^{\gamma}(\mathbb{R}^{d})}\leq R\}.
\]
$\mathcal{K}_{R}$ is pre-compact in $\text{C}_{\ell oc}^{\gamma-}%
(\mathbb{R}^{d})$. By assumption (\ref{assumption appendix D q mean}) and
Markov inequality,
\[
P_{N}(\mathcal{K}_{R}^{C})\leq\frac{K}{R^{q}}.
\]
Then the family $(P_{N})_{N\in\mathbb{N}}$ is tight in $C_{\ell oc}^{\gamma
-}(\mathbb{R}^{d})$. Let $(P_{N_{k}})_{k\in\mathbb{N}}$ be any subsequence
converging weakly in the topology of $\text{C}_{\ell oc}^{\gamma-}%
(\mathbb{R}^{d})$ to some measure $P$, which, in principle, depends a priori
on the subsequence. More precisely, denote by $Q_{N}$ the joint law of the
vector $\left(  p_{N},\mu_{N}\right)  $ on Borel sets of $C_{\ell oc}%
^{\gamma-}(\mathbb{R}^{d})\times\mathcal{P}(\mathbb{R}^{d})$. Since we already
know that $\mu_{N}$ converges weakly, hence it is precompact, we can extract
$\left(  n_{k}\right)  _{k\in\mathbb{N}}$ such that $Q_{N_{\textcolor{black}{k}}}$ converges
weakly to a probability measure $Q$ on Borel sets of $C_{\ell oc}^{\gamma
-}(\mathbb{R}^{d})\times\mathcal{P}(\mathbb{R}^{d})$. The second marginal of
$Q $ is $\pi$, the first marginal will be called $P$, as above. The first
marginal of $Q_{N_{\textcolor{black}{k}}}$ is $P_{N_{k}}$ and converges weakly to $P$; the second
marginal is $\pi_{N_{k}}$ and converges weakly to $\pi$. Notice that at this
stage we do not know yet $\mu$ has a density and that $P$ is the law of such
density. Concerning uniqueness, $\mu$ is the unique limit point (in law) of
$\mu_{N}$, but $P$ a priori is not the unique weak limit point of $P_{N}$.

By Skorohod representation theorem, there exists a probability space
$(\widetilde{\Omega},\widetilde{\mathcal{F}},\widetilde{\mathbb{P}})$, random
variables $\left(  \widetilde{p}_{N_{\textcolor{black}{k}}},\widetilde{\mu}_{N_{\textcolor{black}{k}}}\right)  $ and
$\left(  \widetilde{p},\widetilde{\mu}\right)  $ from $(\widetilde{\Omega
},\widetilde{\mathcal{F}},\widetilde{\mathbb{P}})$ to $C_{\ell oc}^{\gamma
-}\left(  \mathbb{R}^{d}\right)  \times\mathcal{P}\left(  \mathbb{R}%
^{d}\right)  $, with laws $Q_{N_{\textcolor{black}{k}}}$ and $Q$ respectively, such that $\left(
\widetilde{p}_{N_{k}},\widetilde{\mu}_{{N_{\textcolor{black}{k}}}}\right)  \rightarrow\left(
\widetilde{p},\widetilde{\mu}\right)  $ as $k\rightarrow\infty$ in
$\text{C}_{\ell oc}^{\gamma-}\left(  \mathbb{R}^{d}\right)  \times
\mathcal{P}\left(  \mathbb{R}^{d}\right)  $, $\widetilde{\mathbb{P}}$-a.s. The
link $p_{N_{k}}=V^{N_{k}}\ast\mu_{N_{k}}$ is preserved under this change of
basis: $\widetilde{p}_{N_{k}}=V^{N_{k}}\ast\widetilde{\mu}_{N_{k}}$ with
$\widetilde{\mathbb{P}}$ probability one. Indeed, denoting by
$\widetilde{\mathbb{E}}[\,\cdot\,]$ the mathematical expectation on
$(\widetilde{\Omega},\widetilde{\mathcal{F}},\widetilde{\mathbb{P}})$,
\[
\widetilde{\mathbb{E}}\left[  1\wedge\left\Vert V^{N_{k}}\ast\widetilde{\mu
}_{N_{k}}-\widetilde{p}_{N_{k}}\right\Vert _{\text{C}\left(  \overline{B}%
_{R}(0)\right)  }\right]  =\mathbb{E}\left[  1\wedge\left\Vert V^{N_{k}}%
\ast\mu_{N_{k}}-p_{N_{k}}\right\Vert _{\text{C}\left(  \overline{B}%
_{R}(0)\right)  }\right]  =0
\]
(the first identity is true because $\left(  \widetilde{p}_{N_{k}%
},\widetilde{\mu}_{N_{k}}\right)  $ and $\left(  p_{N_{k}},\mu_{N_{k}}\right)
$ have the same law; second identity is true because $p_{N_{k}}=V^{N_{k}}%
\ast\mu_{N_{k}}$). Hence $\widetilde{p}_{N_{k}}=V^{N_{k}}\ast\widetilde{\mu
}_{N_{k}}$, $\widetilde{\mathbb{P}}$-a.s.

The novelty on $(\widetilde{\Omega},\widetilde{\mathcal{F}}%
,\widetilde{\mathbb{P}})$ is that we have the random variable $\widetilde{p}$,
not only $\widetilde{\mu}$. Let us prove that the former is the density of the
latter. From the \textbf{remark} above, with $\widetilde{\mathbb{P}}$
probability one, since $\widetilde{\mu}_{N_{k}}$ converges weakly to
$\widetilde{\mu}$ we have
\[
\lim_{k\rightarrow\infty}\left\langle V^{N_{k}}\ast\widetilde{\mu}_{N_{k}%
},\varphi\right\rangle =\left\langle \widetilde{\mu},\varphi\right\rangle
\]
for all $\varphi\in\text{C}_{c}(\mathbb{R}^{d})$. But at the same time, being
$V^{N_{k}}\ast\widetilde{\mu}_{N_{k}}=\widetilde{p}_{N_{k}}$ and
$\widetilde{p}_{N_{k}}$ converges to $\widetilde{p}$ in C$_{\ell oc}^{\gamma
-}\left(  \mathbb{R}^{d}\right)  $, we have%
\[
\lim_{k\rightarrow\infty}\left\langle V^{N_{k}}\ast\widetilde{\mu}_{N_{k}%
},\varphi\right\rangle =\left\langle \widetilde{p},\varphi\right\rangle
\]
for all $\varphi\in\text{C}_{c}(\mathbb{R}^{d})$. Therefore,
\[
\left\langle \widetilde{p},\varphi\right\rangle =\left\langle \widetilde{\mu
},\varphi\right\rangle \text{ for all }\varphi\in\text{C}_{c}(\mathbb{R}^{d})
\]
with $\widetilde{\mathbb{P}}$ probability one. It implies that,
$\widetilde{\mathbb{P}}$-a.s., the measure $\widetilde{\mu}$ has density
$\widetilde{p}\in\text{C}_{\ell oc}^{\gamma-}\left(  \mathbb{R}^{d}\right)  $;
the property that $\widetilde{p}$ is a probability density follows from the
same identity, by suitable choice of $\varphi\in\text{C}_{c}\left(
\mathbb{R}^{d}\right)  $.

Call $\Lambda$ the subset of $\text{C}_{\ell oc}^{\gamma-}\left(
\mathbb{R}^{d}\right)  \times\mathcal{P}\left(  \mathbb{R}^{d}\right)  $ such
that the first element is the density of the second. Call $\Lambda_{2}$ the
set of elements of $\mathcal{P}\left(  \mathbb{R}^{d}\right)  $ that have a
density of class $\text{C}_{\ell oc}^{\gamma-}\left(  \mathbb{R}^{d}\right)
$. The sets $\Lambda$ and $\Lambda_{2}$ are in bijection. The two sets are
measurable in the corresponding spaces and the bijection is bi-measurable.
Therefore a probability measure on C$_{\ell oc}^{\gamma-}\left(
\mathbb{R}^{d}\right)  \times\mathcal{P}\left(  \mathbb{R}^{d}\right)  $,
concentrated on $\Lambda$, corresponds uniquely to a probability measure on
$\mathcal{P}\left(  \mathbb{R}^{d}\right)  $ concentrated on $\Lambda_{2}$, by
this bijection. It follows that $Q$ is uniquely determined by its second
marginal $\pi$, which is unique a priori. This proves that $Q$ is independent
of the subsequence $\left(  n_{k}\right)  _{k\in\mathbb{N}}$ and thus the full
sequence $(Q_{N})_{N\in\mathbb{N}}$ converges, to a single $Q$.

We can now prove that $\mu$ has a density, $\widetilde{\mathbb{P}}$-a.s. We
have proved that the law of $\widetilde{\mu}$ is concentrated on $\Lambda_{2}%
$; but, being $Q$ the law of $\left(  \widetilde{p},\widetilde{\mu}\right)  $
and having $Q$ second marginal $\pi$, the law of $\widetilde{\mu} $ is $\pi$.
Hence $\pi$, which is also the law of $\mu$, is concentrated on $\Lambda_{2}$.
Namely, $\mathbb{P}$-a.e. realization of $\mu$ has a density $p$, of class
$\text{C}_{\ell oc}^{\gamma-}(\mathbb{R}^{d})$. The random element $\left(
p,\mu\right)  $ is the image of $\mu$ under the bijection above, hence it has
law $Q$. It follows, from the weak convergence of $(Q_{N})_{N\in\mathbb{N}}$
to $Q$, that $p_{N}$ converges to $p$ in law.

It remains to prove (\ref{thesis appendix D q prime mean}) and
(\ref{thesis appendix D q prime mean 2}). Let us prove
(\ref{thesis appendix D q prime mean}). The sequence of r.v.'s $\left\{
\sup_{\left\vert x\right\vert \leq n}\left\vert p\left(  x\right)  \right\vert
^{q^{\prime}}\right\}  _{n\in\mathbb{N}}$ is non decreasing and non-negative,
and converges a.s. to $\sup_{x\in\mathbb{R}^{d}}\left\vert p\left(  x\right)
\right\vert ^{q^{\prime}}$, hence by Beppo-Levi theorem%
\[
\mathbb{E}\left[  \sup_{x\in\mathbb{R}^{d}}\left\vert p\left(  x\right)
\right\vert ^{q^{\prime}}\right]  =\lim_{n\rightarrow\infty}\mathbb{E}\left[
\sup_{\left\vert x\right\vert \leq n}\left\vert p\left(  x\right)  \right\vert
^{q^{\prime}}\right]  .
\]
Therefore (using also the fact that $\widetilde{p}$ and $p$ have the same law,
the first marginal of $Q$ above) it is sufficient to find a constant $C>0$,
independent of $R$, such that
\[
\mathbb{E}\left[  \sup_{\left\vert x\right\vert \leq R}\left\vert
\widetilde{p}\left(  x\right)  \right\vert ^{q^{\prime}}\right]  \leq C
\]
for every $R>0$. But we know that $\sup_{\left\vert x\right\vert \leq
R}\left\vert \widetilde{p}^{N_{K}}\left(  x\right)  \right\vert ^{q^{\prime}}$
converges a.s. to $\sup_{\left\vert x\right\vert \leq R}\left\vert
\widetilde{p}\left(  x\right)  \right\vert ^{q^{\prime}}$. Moreover, we know
that there exists $\gamma>1$ such that
\[
\mathbb{E}\left[  \left(  \sup_{\left\vert x\right\vert \leq R}\left\vert
\widetilde{p}^{N}\left(  x\right)  \right\vert ^{q^{\prime}}\right)  ^{\gamma
}\right]  \leq K
\]
(take $\gamma=q/q^{\prime}$ and use assumption
(\ref{assumption appendix D q mean})). Hence, by Vitali convergence theorem,
we get
\[
\mathbb{E}\left[  \sup_{\left\vert x\right\vert \leq R}\left\vert
\widetilde{p}\left(  x\right)  \right\vert ^{q^{\prime}}\right]
=\lim_{N\rightarrow\infty}\mathbb{E}\left[  \sup_{\left\vert x\right\vert \leq
R}\left\vert \widetilde{p}^{N}\left(  x\right)  \right\vert ^{q^{\prime}%
}\right]  \leq K^{1/\gamma}.
\]

Finally, (\ref{thesis appendix D q prime mean 2}) is proved similarly, under
the additional assumption that $p$ is deterministic. In this case $p^{N}$
converges to $p$ in probability, not only in law, in $\text{C}_{\ell
oc}^{\gamma-}(\mathbb{R}^{d})$. In particular, $\sup_{\left\vert x\right\vert
\leq R}\left\vert p^{N}\left(  x\right)  -p\left(  x\right)  \right\vert
^{q^{\prime}}$ converges to zero in probability. Since $\sup_{\left\vert
x\right\vert \leq R}\left\vert p^{N}\left(  x\right)  -p\left(  x\right)
\right\vert ^{q^{\prime}}$ is uniformly integrable, by Vitali theorem it
converges to zero in average.
\end{proof}

\section{Relaxed Controls}\label{app:relaxed_controls}
In the proof of Theorem \ref{th:Nash_equilibria_MFG} we use the concept of \textit{relaxed controls}. In this section we briefly recall the definition of relaxed controls are; for more details, see, for instance, \cite{karoui} and \cite{kushner1990numerical}. Let $\mathcal{S}$ be a Polish space and let $\mathcal{R}_{\mathcal{S}}$ be the space of all deterministic $\mathcal{S}$-valued relaxed controls over the time interval $[0, T]$, that is, 
$$\mathcal{R}_{\mathcal{S}} \doteq \{r\,:\,r\,\text{positive measure on $\mathcal{B}(\mathcal{S} \times [0,T])$\,:\,$r(\mathcal{S} \times [0,t]) = t$, $t \in [0,T]$}\}.$$
If $r \in \mathcal{R}_{\mathcal{S}}$, then the time derivative of $r$ exists almost everywhere as a measurable mapping $ \overset{\cdot}{r}_t : [0, T] \rightarrow \mathcal{P}(\mathcal{S})$ such that $r(dy, dt) = \overset{\cdot}{r}_t(dy)\,dt$. The topology of weak convergence of measure turns $\mathcal{R}_{\mathcal{S}}$ into a Polish space. In addition, the space $\mathcal{R}_{\mathcal{S}}$ is compact if $\mathcal{S}$ is compact. Finally, any $\mathcal{S}$-valued $(\mathcal{F}_t)$-adapted process $\alpha$ defined on some filtered probability space $(\Omega, \mathcal{F}, \mathbb{P})$ induces a $\mathcal{R}_{\mathcal{S}}$-valued random variable $\rho$, the corresponding stochastic relaxed control, according to:
$$\rho_{\omega}(B \times I) \doteq \int_{I} \delta_{\alpha(t,\omega)}(B)\,dt,$$
where $B \in \mathcal{B}(\Gamma)$ with $\Gamma$ the set of control actions, or action space, $I \in \mathcal{B}([0, T])$ and $\omega \in \Omega$. The random measure $\rho$ is $(\mathcal{F}_t)$-adapted in the sense that its restriction to $\mathcal{S} \times [0,t]$ is $\mathcal{F}_t$-measurable for every $t \in [0,T]$.
\newpage
{
\bibliographystyle{Chicago}
\bibliography{moderateBib}
}

\end{document}